\documentclass{article}
\usepackage{authblk}
\usepackage[margin=1in]{geometry}
\usepackage{amsmath}
\usepackage{amssymb}
\usepackage{amsthm}
\usepackage{amsopn}
\usepackage{cleveref}
\usepackage{algorithm} 
\usepackage{algpseudocode}
\usepackage{xcolor}
\usepackage{subfig}
\usepackage{graphicx} 
\algrenewcommand\algorithmicrequire{\textbf{Input:}}
\algrenewcommand\algorithmicensure{\textbf{Output:}}

\newcommand{\ind}{\,\mbox{d}}

\newtheorem{lemma}{Lemma}
\newtheorem{theorem}{Theorem}
\newtheorem{corollary}{Corollary}

\title{A theoretical and computational framework for three dimensional inverse medium scattering using the linearized low-rank structure}

 \author[a]{Yuyuan Zhou}
\affil[a]{Academy of Mathematics and Systems Science,         Chinese Academy of Sciences, Beijing, 100190, China and School of Mathematical Sciences, University of Chinese Academy of Sciences, Beijing 100049, China}
\author[b]{Lorenzo Audibert}
\affil[b]{PRISME, EDF R\&D, 6 Quai Watier, 78400 Chatou, France and UMA, Inria, ENSTA Paris, Institut Polytechnique de Paris, 91120 Palaiseau, France.
}
\author[c]{Shixu Meng}
\affil[c]{Department of Mathematical Sciences, The University of Texas at Dallas, 75025 Richardson,  USA.
}
\author[a]{Bo Zhang}

\date{\today}

\begin{document}

\maketitle

\begin{abstract}
In this work we propose a theoretical and computational framework for solving the three dimensional inverse medium scattering problem, based on a set of data-driven basis arising from the linearized problem. This set of data-driven basis consists of generalizations of prolate spheroidal wave functions to three dimensions (3D PSWFs), the main ingredients to explore a low-rank approximation of the inverse solution. We first establish the fundamentals of the inverse scattering analysis, including regularity in a customized Sobolev space and new a priori estimate. This is followed by a computational framework showcasing computing the 3D PSWFs and the low-rank approximation of the inverse solution. These results rely heavily on the fact that the 3D PSWFs are eigenfunctions of both a restricted Fourier integral operator and a Sturm-Liouville differential operator. Furthermore we propose a Tikhonov regularization method with a customized penalty norm and a localized imaging technique to image a targeting object despite the possible presence of its surroundings. Finally various numerical examples are provided to demonstrate the potential of the proposed method. 
\end{abstract}

\section{Introduction}

Inverse scattering merits important applications in geophysical exploration, tunnel imaging, medical diagnosis, non-destructive testing and many others. It is a nonlinear and ill-posed problem, imposing challenges in both its analysis and inversion algorithm. Close to the study of this work, the linear sampling method \cite{colton1996simple} is a simple method for solving the inverse problem in the resonance region where Born approximation or physical optics is inadequate; it is worth mentioning other similar sampling methods, such as the factorization method \cite{kirsch1998characterization}, the generalized linear sampling method \cite{audibert2014generalized}, and  an alternative linear sampling method \cite{arens2015indicator}, among many others \cite{griesmaier2011multi,harris2020orthogonality,ito2012direct,liu2018extended,liu2022modified,potthast2010study}; for a more complete picture, we refer to the monographs \cite{cakoni2014qualitative,cakoni2011linear,colton2012inverse,kirsch2007factorization}. In a recent work \cite{audibert2024shape} it was shown that, for a linear inverse problem, a formulation of the linear sampling method is capable of reconstructing the parameter, extending its shape reconstruction capability. The ill-posedness nature of inverse scattering has been elucidated in the work \cite{kirsch2017remarks}, showing that the ill-posedness of the inverse scattering problem is similar to the ill-posedness of its linearized (i.e. Born) approximation; the work \cite{moskow2008convergence} on inverse Born series demonstrated this property via a similar problem. 
To further shed light on the ill-posedness and to bridge the gap between data science and inverse scattering, the recent work \cite{meng23data} studied an intrinsic data-driven basis for the linearized inverse scattering problem and demonstrated a new explicit a priori estimate similar to the Holder-Logarithmic estimate \cite{novikov22}; such a set of data-driven basis naturally leads to a low-rank structure for a computational inversion technique \cite{zhou2024exploring}, facilitating both computational efficiency, robustness and a theoretical Lipchitz stability within a low-rank space. Such a linearized low-rank structure merits further application \cite{cakoni2025recovery} in solving nonlinear inverse scattering problems in connection with the inverse Born series. The above-mentioned low-rank structure has been developed in two dimensions.

In this work we explore the low-rank structure for the three dimensional inverse scattering problem. Similar to the two dimensional case, the set of data-driven basis consists of generalization of prolate spheroidal wave functions to three dimensions \cite{Slepian64}, and we refer to such functions as 3D PSWFs in this work. Exploring the low-rank structure in three dimensions is a highly nontrivial study since both the analysis and the computation are more complicated. Our study depends heavily on the fact that the 3D PSWFs are eigenfunctions of both a restricted Fourier integral operator and a Sturm-Liouville differential operator, which is usually referred to as the dual property of PSWFs in the one dimensional case \cite{osipov2013prolate}. Taking full advantage of the dual property, this work contributes to the following: (1) establishing the fundamentals of the inverse scattering analysis, including a regularity estimate using the customized Sobolev space, a new a priori estimate similar to the Holder-Logarithmic estimate \cite{novikov22}, and a new Lipchitz stability estimate in a low-rank space; (2) proposing a computational framework based on the low-rank structure given by a finite set of 3D PSWFs, sharing similar spirit to the singular value decomposition but facilitating robustness and efficiency thanks to a related Sturm-Liouville problem; (3) demonstrating a Tikhonov regularization method   with a $H^s_c$ penalty term (where $H^s_c$ is a customized Sobolev space equivalent to the standard Sobolev space $H^s$) with a straightforward modification of the low-rank approximation; (4) facilitating a localized imaging technique (i.e., imaging targeting objects inside a ball when there are possible objects around) via the double orthogonality property of the 3D PSWFs (i.e., the 3D PSWFs are mutually orthogonal in both $L^2(\mathbb{R}^3)$ and the unit ball). As a next step, we plan to combine this linearized low-rank structure with other nonlinear inversion techniques to address more challenges. In a broader perspective, the low-rank structure is in the spirit of learning low dimensional features by machine learning \cite{goodfellow2016deep,vapnik1995support}; see recent machine learning approaches to inverse scattering, to name a few \cite{desai2025neural,khoo2019switchnet,li2024reconstruction,liu2022deterministic,meng2024kernel,sun2024learned,zhou2025recovery}.

The remaining of the paper is organized as follows. In Section \ref{section: model}, we introduce the mathematical model of inverse scattering in three dimensions. Section \ref{section: analysis} consists of various aspects of the analysis, including the introduction of the 3D PSWFs, the customized Sobolev space, the low-rank approximation with a priori estimate, the Tikhonov regularization method, and the localized imaging technique. Section \ref{section: computation} is mainly on the computational development, including computing the 3D PSWFs system and the approximate solution in the low-rank space. Finally various numerical experiments are conducted in Section \ref{section: numerical}, illustrating the full potential of the proposed low-rank structure.

\section{Mathematical model of inverse scattering in three dimensions} \label{section: model}
Let $k>0$ be  the wave number and the contrast function $q(x)\in L^\infty (\Omega)$ be supported in an open and bounded set $\Omega$. Without loss of Generality we assume that $\Omega \subset B(0,1)$;  here $B(0,1)\subset \mathbb{R}^3$ is the unit ball.   $\partial \Omega$ is assumed to be  Lipschitz  and  $\mathbb{R}^3\setminus\bar{\Omega}$  is connected. Let  $u^{i}(x;\hat{\theta};k)=e^{ikx\cdot \hat{\theta}}$ be a plane wave  with incident direction $\hat{\theta}\in\mathbb{S}^2:=\{x\in\mathbb{R}^3:|x|=1\}$. In the context of acoustic scattering, the refractive index is related to $1+q$ and the direct scattering  due to the plane wave  $u^{i}(x;\hat{\theta};k)=e^{ikx\cdot \hat{\theta}}$  is to determine the total field $u^t\in H_{loc}^1(\mathbb{R}^3)$ such that
 \begin{equation}\label{eq: Direct problem}
     \begin{aligned}
           \Delta u^t+k^2(1+q)u^t&=0 \quad ~\mbox{in} \quad \mathbb{R}^3,\\
    u^t=u^{i}+u^s&,\\
     \lim_{r:=|x|\to\infty}r\Big(\frac{\partial u^s}{\partial r}-iku^s \Big)&=0.
     \end{aligned}
 \end{equation}
Here  $u^s$ represents the scattered wave field. A solution which satisfies the third equation, i.e., the Sommerfeld condition, uniformly for all directions $\hat{x}:=x/|x|\in \mathbb{S}^2$ is called a radiating solution.  

The above problem can be casted as the following problem to find a radiating solution $u^s\in H^1_{loc}(\mathbb{R}^3)$ to 
 \begin{equation}\label{eq: full model}
     \Delta u^s+k^2(1+q)u^s=-k^2qf,
 \end{equation}
 where $f\in L^\infty(\mathbb{R}^3)$. To get the direct scattering problem \eqref{eq: Direct problem}, one simply sets $f$ by $e^{ikx\cdot \hat{\theta}}$.   According to \cite{colton2012inverse}, there exists a unique radiating solution to \eqref{eq: Direct problem} and the solution satisfies the Lippmann-Schwinger integral equation,
\begin{equation*}
    u^s(x;\hat{\theta};k)=k^2\int_{\mathbb{R}^3}\Phi(x,y)(u^s(y;\hat{\theta};k)+e^{ik\hat{\theta}\cdot y})q(y)dy,\quad x\in \mathbb{R}^3,
\end{equation*}
where 
\begin{equation*}
    \Phi(x,y):=\frac{1}{4\pi}\frac{e^{ik|x-y|}}{|x-y|},~x\neq y,
\end{equation*}
is the fundamental function for the Helmholtz equation.
The scattered field has the following asymptotic behavior (cf.\cite{colton2012inverse})
\begin{equation*}
    u^s(x;\hat{\theta};k)=\frac{e^{ikr}}{r}\Big(u^\infty (\hat{x};\hat{\theta};k)+\mathcal{O}\Big(\frac{1}{r}\Big) \Big)\mbox{  as  } r=|x|\to\infty,
\end{equation*}
which holds uniformly with respect to all  directions $\hat{x} = x/|x| \in\mathbb{S}^2$. It is known \cite{colton2012inverse} that the far-field pattern can be represented by 
\begin{equation} \label{full far field}
u^\infty(\hat{x};\hat{\theta};k)=\frac{k^2}{4\pi}\int_{\mathbb{R}^3}e^{-ik\hat{x}\cdot y}\Big(u^{i}(y;\hat{\theta};k)+u^s(y;\hat{\theta};k)\Big)q(y)dy.
\end{equation}

The multi-static data at a fixed frequency over the full aperture are given by
\begin{equation*}
    \{u^\infty(\hat{x};\hat{\theta};k):\,\hat{x},\hat{\theta}\in\mathbb{S}^2\}.
\end{equation*}
The inverse scattering problem  is to determine the contrast $q$ from these far field data. It is known that there exists a unique solution to this three dimensional inverse scattering problem, see \cite{colton2012inverse}. In practical scenarios, far-field data is commonly encountered as a discrete dataset
\begin{equation*}
 \{u^\infty(\hat{x}_m;\hat{\theta}_\ell;k):m=1,2,\dots,N_1,~\ell=1,2\dots,N_2\}.    
\end{equation*}
The problem is nonlinear and severely ill-posed.  Recent work \cite{kirsch2017remarks} indicates that the eigenvalues of the data operator and these of its linearized version (i.e., Born data operator) decay at the same rate.  Therefore, we are motivated to investigate the ill-posed nature of the Born model as a step towards the full model.   The Born approximation $u^s_b(x;\hat{\theta};k)$ is the unique radiating solution to the Born model
\begin{equation}\label{eq: Born model}
    \Delta u^s_b+k^2 u_b^s=-k^2qe^{ikx\cdot\hat{\theta}}\mbox{ in } \mathbb{R}^3.
\end{equation}
Similarly, $u^s_b(x;\hat{\theta};k)$ has the asymptotic behavior 
\begin{equation*}
    u^s_b(x;\hat{\theta};k)=\frac{e^{ikr}}{r}\Big(u_b^\infty (\hat{x};\hat{\theta};k)+\mathcal{O}\Big(\frac{1}{r}\Big)\Big)\mbox{  as  } r=|x|\to\infty,
\end{equation*}
where the Born far field pattern is given by
\begin{equation}\label{born far field}
    u_b^\infty (\hat{x};\hat{\theta};k)=\frac{k^2}{4\pi}\int_{\mathbb{R}^3}e^{ik(\hat{\theta}-\hat{x})\cdot y}q(y)dy.
\end{equation}
The method of inverse Born series \cite{moskow2008convergence} can be applied where the key step is to understand the ill-posed Born inverse scattering problem \eqref{born far field}; see \cite{cakoni2025recovery} for the inverse Born series for solving inverse scattering with two unknown coefficients in two dimensions. We remark that we formulate the problem for a contrast in a unit ball; for a general contrast, it can be formulated via appropriate scaling, see for instance \cite{meng23data}.

Loosely speaking, the formula \eqref{born far field} reveals that the Born far field pattern $ u_b^\infty (\hat{x};\hat{\theta};k)$  is precisely the (scaled-) Fourier transform of  contrast $q$ evaluated only at $k(\hat{\theta}-\hat{x})$; the set of all possible $k(\hat{\theta}-\hat{x})$ leads to a ball of radius $2k$. Such a restricted Fourier integral operator leads to a data-driven basis for the inverse scattering problem in two dimensions \cite{meng23data}. A low-rank structure can be derived using a finite set of the basis functions, leading to a recent computational approach \cite{zhou2024exploring}. However it is highly non-trivial to generalize to the three dimensional case. In this work we aim to establish a theoretical and computational framework in three dimensions, with applications to Tikhonov regularization with customized penalty term and localized imaging.

Now we give the following formulation. Introduce $c=2k$ and
\begin{equation} \label{section intro Born processed data}
    u_b(p;c) := \frac{4 \pi}{k^2} u_b^\infty (\hat{x};\hat{\theta};k) \quad \mbox{if} \quad p = \frac{\hat{\theta}-\hat{x}}{2}.
\end{equation}
Since $\hat{x},\hat{\theta}\in\mathbb{S}^2$, then it follows directly that $\{\frac{\hat{\theta}-\hat{x}}{2}: \hat{x},\hat{\theta}\in\mathbb{S}^2\} = \overline{B}$  where $B$ denotes the unit ball $B(0,1)$.
Note that even though there are many possible $(\hat{\theta},\hat{x})$ which yield the same $p=\frac{\hat{\theta}-\hat{x}}{2}\in B$, however the value of $u_b(p;c)$ is uniquely determined due to \eqref{born far field}. By a direct calculation, one verifies that
\begin{equation} \label{section intro data q to u_b}
    u_b(p;c) = \int_{B} e^{i c p \cdot \tilde{p}} q(\tilde{p}) {\rm d} \tilde{p}.
\end{equation}
This is the formulation of the linearized inverse problem of reconstructing the unknown $q$ from the processed data $u_b(p;c)$ given by \eqref{section intro data q to u_b}. For the full model \eqref{eq: full model}, the full far field data (with potential noise) $u^\infty (\hat{x};\hat{\theta};k)$ will lead to a perturbation $u_b^\delta$ such that $\|u_b^\delta - u_b\| \le \delta$ for some perturbation level $\delta>0$. We will demonstrate the a priori estimate and computational frameworks for general perturbed data $u_b^\delta$.

\section{A priori estimate and reconstruction formula} \label{section: analysis}
\subsection{Generalization of prolate spheroidal wave functions} \label{subsection: concept of 3D PSWFs}
It is convenient to introduce the restricted Fourier integral operator  $\mathcal{F}^c:L^2(B)\mapsto L^2(B)$   by
\begin{equation*}
    \mathcal{F}^c [g](x)=\int_{B}e^{ic x\cdot y}g(y)dy,~x\in B,
\end{equation*}
where $c>0$ is the bandwidth parameter. The prolate spheroidal wave functions were studied in a series of work \cite{slepian1961prolate,Slepian64}, and in particular   \cite{Slepian64} showed that  the restricted Fourier integral operator is associated with the eigensystem $\{\psi_{m,n,l}(x;c),\alpha_{m,n}(c)\}^{l\in\mathbb{I}(m)}_{m,n\in \mathbb{N}}$ such that
\begin{align}\label{eigen_R_Fourier}
        \mathcal{F}^c [\psi_{m,n,l}](x)=\alpha_{m,n}(c)\psi_{m,n,l}(x;c), 
    \end{align}
where $\mathbb{N}=\{0,1,2,3,\dots\}$ and $\mathbb{I}(m)=\{l\in\mathbb{Z}:|l|\leq m\}$.  The eigenfunctions $\{\psi_{m,n,l}\}_{m,n\in \mathbb{N}}^{l\in\mathbb{I}(m)}$ are real-valued, analytical and  form a complete orthogonal system of $L^2(B)$, and the eigenvalues  $\{\alpha_{m,n}(c)\}$ satisfy following property,
\begin{equation*}
           \alpha_{m,n}(c)\rightarrow 0 \mbox{    as    }  m,n\longrightarrow +\infty,  \quad \mbox{and} \quad  |\alpha_{m,n_1}(c)|>|\alpha_{m,n_2}(c)|>0 \mbox{ for } n_1>n_2.
    \end{equation*}
    The above eigensystem  precisely leads to the data-driven basis $\{\psi_{m,n,l}(x;c)\}^{l\in\mathbb{I}(m)}_{m,n\in \mathbb{N}}$, and a low-rank structure can be formulated using a finite set of these basis functions.  In this work the eigenfunctions are normalized such that 
    $$\int_{B(0,1)}\psi^2_{m,n,l}(y;c) {\rm ~d} y=1, \quad  ~m,n\in \mathbb{N},l\in\mathbb{I}(m).
    $$
From now on, we refer  to the eigenfunction $\psi_{m,n,l}(x;c)$ as the 3D PSWF and the associated eigenvalue $\alpha_{m,n}(c)$ as the prolate eigenvalue.

One of the most important properties of the family of PSWFs is the so-called dual property \cite{osipov2013prolate}: the PSWFs are eigenfunctions of an integral operator and of a differential operator at the same time. In our case, the 3D PSWFs are also  eigenfunctions of the following Sturm-Liouville differential operator such that
\begin{equation}\label{eq: sturm-liouvill}
    \mathcal{D}_c [\psi_{m,n,l}](x)=\chi_{m,n} \psi_{m,n,l}(x),\quad x\in B,
\end{equation}
where $\mathcal{D}_{c}$ in spherical coordinate $(r,\theta,\phi)$ is given by
\begin{align*}
    \mathcal{D}_{c} := -(1-r^2)\partial_r^2-\frac{2}{r}\partial_r+4r\partial_r-\frac{1}{r^2}\Delta_{\mathbb{S}^2}+c^2r^2
\end{align*}
with the Laplace–Beltrami operator $\Delta_{\mathbb{S}^2}$ in spherical coordinate given by 
\begin{equation*}
        \Delta_{\mathbb{S}^2} = \frac{1}{\sin \theta} \frac{\partial}{\partial \theta} \Big( \sin \theta \frac{\partial}{\partial \theta}  \Big) + \frac{1}{\sin^2 \theta} \frac{\partial^2}{\partial \phi^2}.
\end{equation*}
Here with the spherical coordinate, one represents $x  = (x_1,x_2,x_3)^T = r \hat{x}$ in cartesian coordinate by
$$
x_1 = r \sin \theta \cos \phi, \quad x_2 = r\sin \theta \sin \phi, \quad x_3 =  r \cos \theta.
$$
The eigenvalue $\chi_{m,n}(c)$ will be refereed to as the Sturm-Liouville eigenvalue and they are sorted according to $ 0<\chi_{m,0}(c)<\chi_{m,1}(c)<\chi_{m,2}(c)<\cdots $ for any $m=0,1,\cdots$. The Sturm-Liouville eigenvalues approach to $+\infty$ in contrast to the prolate eigenvalues.

This dual property is due to the fact that the restricted Fourier integral operator $\mathcal{F}_c$ and the Sturm-Liouville differential operator $\mathcal{D}_c$ commutes with each other, i.e.,
$$
\langle \mathcal{F}_c g, \mathcal{D}^*_c h \rangle = \langle \mathcal{D}_c g, \mathcal{F}^*_c h \rangle
$$
for any pair of smooth functions $g$ and $h$. The  restricted Fourier integral operator $\mathcal{F}_c$, stemming from the inverse scattering problem, can be seen in the ``main stage''; the seemly unrelated Sturm-Liouville differential operator $\mathcal{D}_c$, is the ``backstage'' hero in both the analysis and computation. 
More details about the family of PSWFs can be found in \cite{greengard2024generalized,osipov2013prolate,Slepian64,ZLWZ20}. In the following, we will demonstrate the theoretical and computational framework, using exactly the interplay between the two different operators. In the following, we simply drop the parameter $c$ when there is no confusion.

\subsection{Customized Sobolev space and approximation theory} \label{subsection: comtomized Sobolev space}
With the help of Sturm-Liouville theory, one can define the following customized Sobolev space
\begin{equation*}
    H_{c}^s(B):=\{ u \in L^2(B): \sum^{l\in\mathbb{I}(m)}_{m,n\in \mathbb{N}}\chi^s_{m,n}|\langle u, \psi_{m,n,\ell} \rangle|^2 < \infty  \}
\end{equation*}
for any integer $s=1,2,\cdots$,
where $\langle u, \psi_{m,n,\ell} \rangle = \int_B u(x) \overline{\psi}_{m,n,\ell}(x) {\rm d}x$ (here we have formally adopted the conjugate in the inner product but note that $\psi_{m,n,\ell}$ is real-valued). By interpolation theory \cite{mclean2000strongly}, the above Sobolev space $H_{c}^s(B)$ is well defined for any real number $s\ge0$.

Now for any function $u\in H_{c}^s(B)$, and its approximation
\begin{equation} \label{section: theory approximation in H_c^s}
    \pi_\epsilon u = \sum_{\chi_{m,n}\le 1/\epsilon} \langle u, \psi_{m,n,\ell} \rangle \psi_{m,n,\ell}
\end{equation}
with a positive parameter $\epsilon>0$, one can estimate the approximation error in $L^2(B)$.
\begin{lemma} \label{lemma bound pi_epsilon u - u}
Let $\epsilon>0$ be a small positive parameter, and $\pi_\epsilon u $ be the approximation given by \eqref{section: theory approximation in H_c^s}, then for any $u \in H_c^s(B)$, it holds that
\begin{equation}
    \|\pi_\epsilon u - u\|_{L^2(B)} < \epsilon^{s/2}\|u\|_{H_c^s(B)}.
\end{equation}
\end{lemma}
\begin{proof}
For any $u \in H_c^s(B)$, it follows that
\begin{eqnarray*}
    \|\pi_\epsilon u - u\|_{L^2(B)} &=& \Big\| \sum_{\chi_{m,n}> 1/\epsilon} \langle u, \psi_{m,n,\ell} \rangle \psi_{m,n,\ell} \Big\|_{L^2(B)} = \sum_{\chi_{m,n} > 1/\epsilon} |\langle u, \psi_{m,n,\ell} \rangle|^2 \\
    &\le& \epsilon^s\sum_{\chi_{m,n} > 1/\epsilon} \chi_{m,n}^s|\langle u, \psi_{m,n,\ell} \rangle|^2 \le \epsilon^s\|u\|^2_{H_c^s(B)}.
\end{eqnarray*}
This completes the proof.
\end{proof}

The customized Sobolev space $H_c^s(B)$ can be related to the standard Sobolev space. In the inverse scattering problem, the unknown contrast may encounter non-smoothness, which motivates us to consider functions in $H^s(B)$ with $0<s<1$.
\begin{lemma} \label{lemma H_c^s bounded by H^s}
    If $u\in H^s(B)$ for any $0<s<1$, then $u\in H_c^s(B)$ and 
    \begin{equation*}
      \|u\|_{H_c^s(B)} \le (1+c^2)^{s/2} \|u\|_{H^s(B)} .
    \end{equation*}
\end{lemma}
\begin{proof}
By the interpolation theory for Sobolev spaces \cite{mclean2000strongly}, it is sufficient to prove the lemma for $s=1$. For any smooth $u$, it follows from the definition of the Sturm-Liouville differential operator that
\begin{eqnarray*}
    \|u\|^2_{H_c^1(B)} & =& \langle \mathcal{D}_cu,u  \rangle = \int_B \Big( -(1-r^2)\partial_r^2u-\frac{2}{r}\partial_r u + 4r\partial_r u-\frac{1}{r^2}\Delta_{\mathbb{S}^2}u + c^2r^2 u \Big) \overline{u} {\rm d} x \\
    &=&   \int_0^{\pi} \int_0^{2\pi} \int_0^1 \Big( 
    - \sin \theta (r^2-r^4) \partial_r^2u - 2\sin\theta ~r \partial_r u  + 4\sin\theta  ~r^3  \partial_r u\\
    && \hspace{2.5cm} - \sin\theta  \Big( \frac{1}{\sin \theta} \frac{\partial}{\partial \theta} \Big( \sin \theta \frac{\partial u}{\partial \theta}  \Big) + \frac{1}{\sin^2 \theta} \frac{\partial^2 u}{\partial \phi^2 } \Big) + c^2 r^4  \sin\theta u
    \Big) \overline{u} ~{\rm d} r  \,{\rm d} \phi \, {\rm d} \theta.
\end{eqnarray*}
To simplify the above formula, integration by parts yields that
\begin{eqnarray*}
    &&\int_0^{\pi} \int_0^{2\pi} \int_0^1 \Big( 
    - \sin \theta (r^2-r^4) \partial_r^2u - 2\sin\theta ~r \partial_r u  + 4\sin\theta  ~r^3  \partial_r u
    \Big) \overline{u} ~ {\rm d} r \, {\rm d} \phi \, {\rm d} \theta \\
    && = \int_0^{\pi} \int_0^{2\pi} \int_0^1   
     \sin \theta (r^2-r^4) \partial_r u  
    \partial_r\overline{u} ~ {\rm d} r \, {\rm d} \phi \, {\rm d} \theta, 
\end{eqnarray*}
and
\begin{eqnarray*}
    &&\int_0^{\pi} \int_0^{2\pi} \int_0^1 (- \sin\theta)  \Big( \frac{1}{\sin \theta} \frac{\partial}{\partial \theta} \Big( \sin \theta \frac{\partial u}{\partial \theta}  \Big) + \frac{1}{\sin^2 \theta} \frac{\partial^2 u}{\partial \phi^2 } \Big) \overline{u} ~ {\rm d} r \, {\rm d} \phi \, {\rm d} \theta \\
    && = \int_0^{\pi} \int_0^{2\pi} \int_0^1 \Big(  
     \sin \theta   \partial_\theta u \partial_\theta \overline{u}   + \frac{1}{\sin \theta}   \partial_\phi u \partial_\phi \overline{u} \Big)
      ~ {\rm d} r \, {\rm d} \phi \, {\rm d} \theta. 
\end{eqnarray*}
Use the above two identities one can conclude that
\begin{eqnarray*}
    \|u\|^2_{H_c^1(B)} & =& \langle \mathcal{D}_cu,u  \rangle = \int_B \Big( -(1-r^2)\partial_r^2u-\frac{2}{r}\partial_r u + 4r\partial_r u-\frac{1}{r^2}\Delta_{\mathbb{S}^2}u + c^2r^2 u \Big) \overline{u} {\rm d} x \\
    &=&   \int_0^{\pi} \int_0^{2\pi} \int_0^1   
     \Big( \sin \theta (r^2-r^4) \partial_r u  
    \partial_r\overline{u}  + \sin \theta   \partial_\theta u \partial_\theta \overline{u}   + \frac{1}{\sin \theta}   \partial_\phi u \partial_\phi \overline{u} + c^2r^4 \sin\theta u \overline{u} \Big)~ {\rm d} r \, {\rm d} \phi \, {\rm d} \theta.
\end{eqnarray*}
To connect with the standard Sobolev norm, one can proceed to estimate the above right hand side by
\begin{eqnarray*}
  &&  \int_0^{\pi} \int_0^{2\pi} \int_0^1   
     \Big( \sin \theta (r^2-r^4) \partial_r u  
    \partial_r\overline{u}  + \sin \theta   \partial_\theta u \partial_\theta \overline{u}   + \frac{1}{\sin \theta}   \partial_\phi u \partial_\phi \overline{u} + c^2r^4 \sin\theta u \overline{u} \Big)~ {\rm d} r \, {\rm d} \phi \, {\rm d} \theta \\
      \le &&  \int_0^{\pi} \int_0^{2\pi} \int_0^1   
     \Big( \sin \theta ~r^2 |\partial_r u|^2  
      + \sin \theta   |\partial_\theta u|^2    + \frac{1}{\sin \theta}   |\partial_\phi u |^2   + c^2r^2 \sin\theta |u|^2  \Big)~ {\rm d} r \, {\rm d} \phi \, {\rm d} \theta \\
      =&& \|\nabla u\|^2_{L^2(B)} + c^2 \|  u\|^2_{L^2(B)} \le (c^2+1) \|  u\|^2_{H^1(B)},
\end{eqnarray*}
which allows to conclude that
$$
\|u\|^2_{H_c^1(B)} \le (c^2+1) \|  u\|^2_{H^1(B)}.
$$
This completes the proof.
\end{proof}
Finally, for any function $u \in H^s(B)$ with $0<s<1$, the following approximation error can be proved.
\begin{theorem} \label{theorem pi_epsilon u - u bound in H^s norm}
Let $\epsilon>0$ be a small positive parameter, and $\pi_\epsilon u $ be the approximation given by \eqref{section: theory approximation in H_c^s}, then for any $u \in H^s(B)$ with $0<s<1$, it holds that
\begin{equation}
    \|\pi_\epsilon u - u\|_{L^2(B)} < \epsilon^{s/2} (1+c^2)^{s/2}\|u\|_{H^s(B)}.
\end{equation}    
\end{theorem}
\begin{proof}
From Lemma \ref{lemma bound pi_epsilon u - u}, it follows that
$
\|\pi_\epsilon u - u\|_{L^2(B)} < \epsilon^{s/2}  \|u\|_{H_c^s(B)}.
$
Together with Lemma \ref{lemma H_c^s bounded by H^s}, one can complete the proof by  $
\|\pi_\epsilon u - u\|_{L^2(B)} < \epsilon^{s/2}  \|u\|_{H_c^s(B)} \le  \epsilon^{s/2} (1+c^2)^{s/2}\|u\|_{H^s(B)}$.
\end{proof}

\subsection{Low-rank approximation and a priori estimate} \label{subsection: stability}
The data are inevitability corrupted by noise, this motivates us to consider noisy processed data $u_b^\delta$ such that
$$
\|u_b^\delta - u_b\|_{L^2(B)} \le \delta.
$$
For example, the full far field data $u^\infty (\hat{x};\hat{\theta};k)$ will lead to such a perturbation $u_b^\delta$.

Since the unknown contrast $q$ is related to the processed data $u_b$ via \eqref{section intro data q to u_b} and that the restricted Fourier integral operator has a low-rank structure, we are motivated to investigate a low-rank approximation of the noisy processed data given by 
\begin{equation*} 
    \pi_\epsilon u_b^\delta = \sum_{\chi_{m,n}\le 1/\epsilon} \langle u_b^\delta, \psi_{m,n,\ell} \rangle \psi_{m,n,\ell},
\end{equation*}
which leads to the low-rank approximation of the unknown $q$ given by
\begin{equation} \label{section stability formula pi_alpha q}
    \pi_{\epsilon} q^\delta = \sum_{\chi_{m,n}\le 1/\epsilon} \frac{\langle u_b^\delta, \psi_{m,n,\ell} \rangle}{\alpha_{m,n}}  \psi_{m,n,\ell}.
\end{equation}

In the following, we estimate the low-rank approximation error $q - \pi_{\epsilon} q^\delta$ under one of the following two assumptions:  assumption (A) $q \in H_c^s(B)$ or assumption (B) $q \in \mbox{span}\{ \psi_{m,n,\ell}: |\alpha_{m,n}| \ge \sigma \}$.

\begin{theorem} \label{theorem stability infinite dimension}
Assume that  $q \in H_c^s(B)$. Let $\|u_b^\delta - u_b\|_{L^2(B)} \le \delta$ for some $\delta>0$ and the low-rank approximation of the unknown be given by \eqref{section stability formula pi_alpha q} with some small regularization parameter $\epsilon>0$.  Then it holds that
\begin{equation*}
 \| q - \pi_{\epsilon} q^\delta \|_{L^2(B)} \le   \epsilon^{s/2} (1+c^2)^{s/2}\|q\|_{H^s(B)} + \frac{\delta}{\inf\{|\alpha_{m,n}|: \chi_{m,n}\le 1/\epsilon\}}
\end{equation*}
\end{theorem}
\begin{proof}
Note that
 $$
 \| q - \pi_{\epsilon} q^\delta \|_{L^2(B)} \le \| q - \pi_{\epsilon} q  \|_{L^2(B)} + \| \pi_{\epsilon} q - \pi_{\epsilon} q^\delta \|_{L^2(B)},
 $$
 then we estimate each term on the right hand side as follows (where we will drop the $L^2(B)$-norm when there is no confusion).

The first term on the right hand can be estimated by Theorem \ref{theorem pi_epsilon u - u bound in H^s norm} where
\begin{eqnarray*}
    \| q - \pi_{\epsilon} q  \|^2_{L^2(B)} \le \epsilon^{s/2} (1+c^2)^{s/2}\|q\|_{H^s(B)}.
\end{eqnarray*}
The other term can be estimated by
\begin{eqnarray*}
    \| \pi_{\epsilon} q - \pi_{\epsilon} q^\delta \|   = \Big\| \sum_{\chi_{m,n}\le 1/\epsilon} \frac{\langle u_b^\delta -u_b, \psi_{m,n,\ell} \rangle}{\alpha_{m,n}}  \psi_{m,n,\ell}  \Big\| \le \frac{\|u_b^\delta -u_b\|_{L^2(B)}}{\inf\{|\alpha_{m,n}|: \chi_{m,n}\le 1/\epsilon\}} \le \frac{ \delta}{\inf\{|\alpha_{m,n}|: \chi_{m,n}\le 1/\epsilon\}} 
\end{eqnarray*}
The above three inequalities allows to prove the Theorem. This completes the proof.
\end{proof}

\begin{corollary}
    Assume that  $q \in \mbox{span}\{ \psi_{m,n,\ell}: |\alpha_{m,n}| \ge \sigma \}$. Let $\|u_b^\delta - u_b\|_{L^2(B)} \le \delta$ for some $\delta>0$ and the low-rank approximation of the unknown be given by \eqref{section stability formula pi_alpha q} with the  regularization parameter $\epsilon= 1/\sup\{ \chi_{m,n}: |\alpha_{m,n}| \ge \sigma \}$.  Then it holds that
\begin{equation*}
 \| q - \pi_{\epsilon} q^\delta \|_{L^2(B)} \le   \frac{\delta}{\sigma}.
\end{equation*}
\end{corollary}
\begin{proof}
Since by assumption that $q \in \mbox{span}\{ \psi_{m,n,\ell}: |\alpha_{m,n}| \ge \sigma \}$, then with $\epsilon= 1/\sup\{ \chi_{m,n}: |\alpha_{m,n}| \ge \sigma \}$, one can derive that 
$\| q - \pi_{\epsilon} q^\delta \|_{L^2(B)} = \| \pi_{\epsilon} q - \pi_{\epsilon} q^\delta \|_{L^2(B)}$.
Following the proof in Theorem \ref{theorem stability infinite dimension}, one directly proves the Corollary. This completes the proof.
\end{proof}

\subsection{A versatile tool for regularization using $H^s_c$ penalty term} \label{subsection: Tikhonov}
The 3D PSWFs provide a versatile tool for regularization by penalizing the $H^s_c$-norm of solutions. Traditionally assuming  the a priori knowledge that $q \in H^s$, one may consider the cost functional  $\|\mathcal{F}^c q - u_b^\delta\|^2_{L^2} + \eta \|q\|^2_{H^s}$ where the penalty term is  $\eta\|q\|^2_{H^s}$ for some parameter $\eta>0$. However, noting the relation between the standard Sobolev space $H^s$ and the customized Sobolev space $H^s_c$ in Section \ref{subsection: comtomized Sobolev space}, we are motivated to consider the following cost functional 
$$
J_\eta(q) = \|\mathcal{F}^c q - u_b^\delta\|^2_{L^2} + \eta \|q\|^2_{H^s_c},
$$ 
assuming the a priori knowledge that $q \in H^s_c$. One  looks for $q^\delta_{\eta}$ that minimizes $J_\eta(q)$, i.e.,
$$
q^\delta_{\eta} = \arg \min_{q} J_\eta(q).
$$
Since the  customized Sobolev space $H_{c}^s(B)$ is conveniently given by
\begin{equation*}
    H_{c}^s(B)=\{ u \in L^2(B): \sum^{l\in\mathbb{I}(m)}_{m,n\in \mathbb{N}} \chi^s_{m,n}|\langle u, \psi_{m,n,\ell} \rangle|^2 < \infty  \},
\end{equation*}
one can reformulate the above minimization problem using the 3D PSWFs to derive  the following regularized solution 
\begin{equation} \label{section stability formula q_reg in H^s_c}
     q^\delta_{\eta} = \sum^{l\in\mathbb{I}(m)}_{m,n\in \mathbb{N}} \frac{\overline{\alpha}_{m,n}}{|\alpha_{m,n}|^2 + \eta~ \chi_{m,n}^s}  \langle u_b^\delta, \psi_{m,n,\ell} \rangle\psi_{m,n,\ell}.
\end{equation}
The difference between the low-rank regularized solution \eqref{section stability formula pi_alpha q} and the above regularized solution \eqref{section stability formula q_reg in H^s_c} is due to the additional term $\eta \chi_{m,n}^s$, and this will be a straightforward modification since the Sturm-Liouville eigenvalues $\chi_{m,n}$ have been conveniently computed according to Section \ref{subsection: computation of 3D PSWFs system}. Such convenient analytical and computational modification is again thanks to  the dual property of 3D PSWFs.  The idea of deriving a priori estimate for \eqref{section stability formula q_reg in H^s_c} is expected to be similar to the low-rank approximation and we omit it here. Numerical examples based on \eqref{section stability formula q_reg in H^s_c} will  be provided.

\subsection{A Localized imaging technique} \label{subsection: localized imaging}
The PSWFs are known to be double orthogonal \cite{Slepian64}. This double orthogonality property will allow a localized imaging technique -- imaging targeting objects inside a ball when there are possible complex objects around.  This is motivated by, for example,  fast imaging of a specific target despite the possible presence  of surrounding objects (where one avoids parameterizing the surroundings or does not know the region to parameterize in).  To be more precise, following \cite{Slepian64} or \cite{meng23data} one can use the identity \eqref{eigen_R_Fourier}
to define the 3D PSWFs in $\mathbb{R}^3$ naturally via 
\begin{equation} \label{def 3D PSWF in R3}
 \psi_{m,n,\ell}(x) = \frac{1}{\alpha_{m,n}}\int_B e^{i c x \cdot y} \psi_{m,n,\ell} (y) \ind y, \qquad x \in \mathbb{R}^3.
\end{equation}
One can check via Fourier transform and the above \eqref{def 3D PSWF in R3} to verify 
the so-called double orthogonality property (see also \cite{Slepian64})
\begin{eqnarray}
    \int_B \psi_{m,n,\ell} \overline{\psi}_{\widetilde{m}, \widetilde{n}, \widetilde{\ell}} \ind x =  \int_{\mathbb{R}^3} \psi_{m,n,\ell} \overline{\psi}_{\widetilde{m}, \widetilde{n}, \widetilde{\ell}} \ind x = 0 \qquad &\mbox{when}& \qquad  \{m,n,\ell\} \not= \{\widetilde{m}, \widetilde{n}, \widetilde{\ell}\}, \label{prop 3D PSWF double orthogonal 1}\\
        1 = \int_B \psi_{m,n,\ell} \overline{\psi}_{\widetilde{m}, \widetilde{n}, \widetilde{\ell}} \ind x = \frac{|\alpha_{m,n}|^2}{(2\pi/c)^3}  \int_{\mathbb{R}^3} \psi_{m,n,\ell} \overline{\psi}_{\widetilde{m}, \widetilde{n}, \widetilde{\ell}} \ind x  \qquad &\mbox{when}& \qquad  \{m,n,\ell\} = \{\widetilde{m}, \widetilde{n}, \widetilde{\ell}\}. \label{prop 3D PSWF double orthogonal 2}
\end{eqnarray}
The consequence of \eqref{prop 3D PSWF double orthogonal 2} is that the square energy (i.e., the square of the $L^2$ norm) of $\psi_{m,n,\ell}$ in $B$ is $\frac{|\alpha_{m,n}|^2}{(2\pi)^3}$ of the square energy  of $\psi_{m,n,\ell}$  in $\mathbb{R}^3$, which implies that: for very small $|\alpha_{m,n}| \ll 1$, the principle energy of $\psi_{m,n,\ell}$ is in $\mathbb{R}^3\backslash \overline{B}$; for dominant prolate eigenvalues $|\alpha_{m,n}|$ close to $|\alpha_{0,0}|$, the principle energy of $\psi_{m,n,\ell}$ is in $B$. Intuitively speaking, the major information of unknowns in $B$ is carried over by the 3D PSWFs with dominant prolate eigenvalues.

Having the above double orthogonality property, one can study the problem of imaging localized objects inside a ball when there are other objects around.  The Born data are still given by \eqref{section intro data q to u_b} by changing the domain of integration to $\mathbb{R}^3$ so that
\begin{equation} \label{def Born region focusing}
    u_b(p;c) = \int_{\mathbb{R}^3} e^{i c p \cdot \tilde{p}} q(\tilde{p}) {\rm d} \tilde{p}.
\end{equation}

Suppose that the unknown is given by $q=q_{\rm i} + q_{\rm o}$ where   $q_{\rm i}$ is supported in $B$ and $q_{\rm o}$ is supported in $\mathbb{R}^3 \backslash \overline{B}$.
Denote by $\langle ,  \rangle_{\mathbb{R}^3}$ the inner product in $L^2(\mathbb{R}^3)$ and $\|\cdot\|_{\mathbb{R}^3}$ the corresponding $L^2(\mathbb{R}^3)$ norm. We add subscript $B$ to indicate the inner product and norm are for the domain $B$. Equation \eqref{def Born region focusing} leads to
\begin{eqnarray*}
    \langle u_b,\psi_{m,n,\ell} \rangle_B &=& \int_B \int_{\mathbb{R}^3} e^{i c x \cdot y} (q_{\rm i}(y) + q_{{\rm o}}(y)) \psi_{m,n,\ell}(x)\ind y \ind x   \\
    &=& \int_{B} \int_B e^{icx\cdot y} q_{\rm i}(y) \ind y  \psi_{m,n,\ell}(x) \ind x + \int_{B} \int_{\mathbb{R}^3 \backslash \overline{B}} e^{icx\cdot y} q_{\rm o}(y) \ind y \psi_{m,n,\ell}(x) \ind x \\
    &=& \int_{B}  \alpha_{m,n} \psi_{m,n,\ell}(y)  q_{\rm i}(y) \ind y    +  \int_{\mathbb{R}^3 \backslash \overline{B}}  \alpha_{m,n} \psi_{m,n,\ell}(y)  q_{\rm o}(y) \ind y  
\end{eqnarray*}
and one can arrive at
\begin{equation} \label{section region focusing series expansion equation}
    \frac{1}{\alpha_{m,n}} \langle u_b,\psi_{m,n,\ell} \rangle_B  = \langle q_{\rm i},\psi_{m,n,\ell} \rangle_{\mathbb{R}^3} + \langle q_{\rm o},\psi_{m,n,\ell} \rangle_{\mathbb{R}^3},
\end{equation}
noting that $q_{\rm i}$ is supported in $B$ and $q_{\rm o}$ is supported in $\mathbb{R}^3 \backslash \overline{B}$.
When the principal energy of $\psi_{m,n,\ell}$ is in the unit ball $B$, one can expect that the last quantity $\langle q_{\rm o},\psi_{m,n,\ell} \rangle_{\mathbb{R}^3}$ in \eqref{section region focusing series expansion equation} is small so that $\frac{1}{\alpha_{m,n}} \langle u_b,\psi_{m,n,\ell} \rangle_B$ gives an approximation of $\langle q_{\rm i},\psi_{m,n,\ell} \rangle_{\mathbb{R}^3} = \langle q_{\rm i},\psi_{m,n,\ell} \rangle_{B}$. This leads to the  approximation
$$
q_{\rm i}^\sigma(x) =  \sum_{|\alpha_{m,n}|>\sigma} \frac{1}{\alpha_{m,n}} \langle u_b,\psi_{m,n,\ell} \rangle_{B} ~ \psi_{m,n,\ell}(x)
$$
where $\sigma$ is chosen close to the dominant prolate eigenvalues. We will demonstrate this property in the numerical study with $\sigma = 0.9 |\alpha_{0,0}|$.

\section{Computational framework} \label{section: computation}
In this section we develop the computational framework for computing the low-rank approximation of the inverse solution. The computation of the 3D PSWFs will be based on the dual property where the 3D PSWFs are both eigenvalues of the restricted Fourier integral operator and the Sturm-Liouville differential operator. The low-rank approximation will be derived by processing the far field data and projecting the processed data onto a low-rank space.
\subsection{Computation of 3D PSWFs and prolate eigenvalues} \label{subsection: computation of 3D PSWFs system}

The 3D PSWF  $\psi_{m,n,\ell} (x;c)$ can be represented by the following series of  ball polynomial (cf. \cite{Slepian64}, \cite{ZLWZ20})
\begin{equation}\label{eq: expansion}
    \psi_{m,n,\ell}(x;c)=\sum_{j=0}^{\infty}\beta_j^{m,n}(c)P_{m,j,\ell}(x),\quad x\in B,
\end{equation}
where the ball polynomial is defined by $P_{m,j,\ell}(x)=r^m P_j^{(m)}(2r^2-1)Y_{m,\ell}(\hat{x})$. Here $Y_{m,l}(\hat{x})$  represents the (real-valued) spherical harmonics and  $P_n^{(m)} $ represents the normalized Jacobi polynomial, whose definitions are given below.

\textbf{Spherical harmonics}. The complex spherical harmonics can be found in \cite{colton2012inverse}, which can be used to derive the following real-valued spherical harmonic $Y_{m,\ell}(\hat{x})$ given by
\begin{align}\label{spherical_harmonic}
    Y_{m,\ell}(\hat{x})=\left\{
            \begin{array}{cc}
            \sqrt{\frac{2m+1}{4\pi}} P_m(\cos\theta), & \ell=0 \\
                \sqrt{\frac{2m+1}{2\pi}\frac{(m-|\ell|)!}{(m+|\ell|)!}} P_m^{\ell}(\cos\theta)\cos(\ell\phi), & \ell=1,2,\cdots,m\\
                \sqrt{\frac{2m+1}{2\pi}\frac{(m-|\ell|)!}{(m+|\ell|)!}} P_m^{|\ell|}(\cos\theta)\sin(\ell\phi),& \ell=-1,-2,\cdots,-m
            \end{array}\right. ,
\end{align}
where $P^{\ell}_m$ is the associated Legendre function given by
$$
P^{\ell}_m(r) = (1-r^2)^{\ell/2}\frac{d^\ell}{dx^\ell}P_m(r), \quad -1<r<1,  \quad 0\le \ell \le m
$$
with $P_m$ being the Legendre polynomial of order $m$, c.f., \cite{Abramowitz64,colton2012inverse}. Such
spherical harmonic $Y_{m,\ell}(\hat{x})$ is the eigenfunction of the Laplace-Beltrami operator, i.e., 
$$
\Delta_{\mathbb{S}^2} Y_{m,\ell}(\hat{x})=-m(m+1) Y_{m,\ell}(\hat{x}),
$$
and satisfies the orthonormal condition
$$\langle Y_{m,\ell},Y_{\tilde{m},\tilde{\ell}}\rangle_{L^2(\mathbb{S}^2)}=\delta_{m,\tilde{m}}\delta_{\ell,\tilde{\ell}}.$$
More details about  spherical harmonics can be found in \cite{Abramowitz64,colton2012inverse}.

\textbf{Jacobi polynomials}.
The normalized Jacobi polynomials $\{P^{(m)}_n(\eta)\}_{\eta\in (-1,1) }^{n\in\mathbb{N}}$ mentioned above  are eigenfunctions of the following Sturm-Liouville problem,
\begin{equation}\label{jacobi polynomial}
    -\frac{1}{w_{m+1/2}(\eta)}\partial_{\eta} \left( (1-\eta)w_{m+3/2}(\eta) \partial_{\eta} P^{(m)}_n(\eta) \right)=n(n+m+3/2)P^{(m)}_n(\eta), \quad -1<\eta<1
\end{equation}
where $w_{m+1/2}(\eta)= (1+\eta)^{m+1/2}$, and they satisfy the orthogonality condition and are normalized by
\begin{equation}\label{orthogonal_JacobiPolynomial}
    \int_{-1}^1  w_{m+1/2}(\eta) P_n^{(m)}(\eta)P_{\tilde{n}}^{(m)}(\eta)d\eta=2^{m+5/2}\delta_{n\tilde{n}}, \quad \forall n,\tilde{n} \in \mathbb{N}.
\end{equation}
The normalized Jacobi polynomials $\{P_n^{(m)}(\eta)\}_{\eta\in (-1,1) }$ satisfy the following three-term recurrence relation
\begin{align*}
         P^{(m)}_{n+1}(\eta)=\frac{ (x-b_n) P_n^{(m)}(\eta)-a_{n -1}P^{(m)}_{n-1}(\eta) }{a_n},\quad n\geq 1
\end{align*}
with the first two terms given by
$$
    P^{(m)}_{0}(\eta)=\frac{1}{h_0 },\qquad P^{(m)}_{1}(\eta)=\frac{(m+5/2)\eta-m-1/2}{2h_1},
$$
here $a_n$, $b_n$, and $h_n$ are given by 
\begin{align}\label{jacobi polynomial recurrence}
   \left\{
            \begin{array}{cc}
                a_{n}=&\frac{2(n+1)(n+m+3/2)}{(2 n+m+5/2)\sqrt{(2 n+m+3/2)(2 n+m+7/2)}} \\
                b_{n}=&\frac{(m+1/2)^{2}}{(2 n+m+1/2)(2 n+m+5/2)} \\
               h_{n}=& \frac{1}{\sqrt{2(2 n+m+3/2)}}
            \end{array}\right. , \qquad n\in\mathbb{N}.
\end{align}
For a more comprehensive introduction to special polynomials, we refer to \cite{Abramowitz64}. From the definition of the normalized Jacobi polynomial and the spherical harmonics, it is seen that $\| P_{m,j,l} \|_{L^2(B)}=1$.

\textbf{Computation of 3D PSWFs system}.
To compute $\{\beta_j^{m,n}\}_{j\in\mathbb{N}}$,  we plug the expansion \eqref{eq: expansion}  into \eqref{eq: sturm-liouvill}  and we can derive that the  coefficients  satisfy
\begin{eqnarray} \label{tridiagonal_linear_system}
       \Big(\gamma_{m+2j}+\frac{(1+b_j)c^2}{2}-\chi_{m,n}(c) \Big) \beta_j^{m,n}   +\frac{a_{j-1}c^2}{2}\beta_{j-1}^{m,n}+\frac{a_{j}c^2}{2}\beta_{j+1}^{m,n}=0, \quad j\geq 0
    \end{eqnarray}
where $\gamma_{m+2j}=(m+2j)(m+2j+3)$. Here $\chi_{m,n}(c)$  is  the Sturm-Liouville eigenvalue. This is an infinite linear system. 
In practice,
one can start with the approximation 
\begin{equation} \label{section: computing 3D PSWFs: polynomial expansion}
    \tilde{\psi}_{m,n,l}(x ; c)=\sum_{j=0}^{K} \tilde{\beta}_{j}^{m,n} P_{m,j,l}(x), \quad 2 n+m \leq N,
\end{equation}
 so that the coefficients  $\{\tilde{\beta}_{j}^{m,n}\}$  can be solved from the following tridiagonal linear system
\begin{equation}\label{eigensystem}
{A}{\tilde{\beta}^{m,n}}=\tilde{\chi}_{m,n}{\tilde{\beta}^{m,n}}
\end{equation}
with eigenvalues $\tilde{\chi}_{m,n}$,
where ${\tilde{\beta}^{m,n}}=(\tilde{\beta}_{0}^{m,n},\tilde{\beta}_{1}^{m,n},...,\tilde{\beta}_{K}^{m,n})^T$, ${A}$ is a $(K+1)\times(K+1)$ symmetric tridiagonal matrix whose nonzero entries are given by
\begin{eqnarray*}
    A_{j,j} =\gamma_{m+2j}+\frac{(1+b_j)c^2}{2}
    \quad  \mbox{and} \quad A_{j,j+1} =A_{j+1,j}=\frac{a_{j}c^2}{2}, \quad j\geq 0.
\end{eqnarray*}
Since the 3D PSWFs are  normalized to have unit $L^2(B)$-norm,  the coefficient vector ${\tilde{\beta}^{m,n}} $ are set to satisfy 
\begin{equation*}
    \| {\tilde{\beta}^{m,n}}\|_2= \Big(\sum_{j=0}^{K} |\tilde{\beta}_{j}^{m,n}|^2 \Big)^{1/2}=1.
\end{equation*}
The eigenvalues are sorted according to
\begin{equation*}
    \tilde{\chi}_{m,j}<\tilde{\chi}_{m,i}, \quad j<i.
\end{equation*}
In practive, the number $K$ of truncated terms in \eqref{section: computing 3D PSWFs: polynomial expansion} can be chosen as $K=\left\lceil\frac{\tilde{M}-m}{2}\right\rceil$, c.f., \cite{ZLWZ20}.

Finally to evaluate the prolate eigenvalues $\{\alpha_{m,n}\}_{m,n\in\mathbb{N}}$, the following formula was given in \cite{ZLWZ20}
\begin{equation}\label{eigenvalues}
   \tilde{\alpha}_{m,n}(c)=\frac{\pi^{3/2} (ic)^m }{2^{m-\frac{1}{2}}\sqrt{\Gamma(m+3/2)\Gamma(m+5/2)}}\frac{\tilde{\beta}_0^{m,n}}{\tilde{\varphi}_{m,n}(-1;c)}
\end{equation}
where $\tilde{\varphi}_{m,n}(-1;c)=\sum_{j=0}^{K} \tilde{\beta}^{m,n}_j P_j^{(m)}(-1)$ and note that $i$ is the imaginary unit here.  This formula with  $N=K=150$  is sufficient for our inverse problem.

We summarize the following algorithm for the evaluation of PSWFs and prolate eigenvalues.  
\begin{algorithm} 
	\caption{Evaluation of  3D PSWF and prolate eigenvalue}
\begin{algorithmic}[1]\label{Algorithm: PSWFs}
\Require  Parameter $c$, index of 3D PSWF $(m,n,l)$ and $x\in\mathbb{R}^2$. 
\Ensure  $\psi_{m,n,l}(x;c)$ and prolate eigenvalue $\alpha_{m,n}(c)$.
\State Compute the tridiagonal linear system \eqref{eigensystem} (and sort the eigenvalues from small to large) to obtain the eigenvector $\tilde{\beta}^{m,n}$ and the eigenvalue $\tilde{\chi}_{m,n}(c)$ corresponding to the given $(m,n)$. Here ${\tilde{\beta}^{m,n}}=(\tilde{\beta}_{0}^{m,n},\tilde{\beta}_{1}^{m,n},...,\tilde{\beta}_{K}^{m,n})^T$.  
\State  Evaluate $(P_{m,0}(2\|x\|^2-1),P_{m,1}(2\|x\|^2-1),\cdots,P_{m,K}(2\|x\|^2-1) )^T$ at $x\in\mathbb{R}^3$ by the recurrence \eqref{jacobi polynomial recurrence} and the spherical harmonics $Y_{m,l}(\hat{x})$ at $\hat{x}$  by \eqref{spherical_harmonic}.
\State Compute the approximation of 3D PSWFs by $\tilde{\psi}_{m,n,l}(x;c) = \|x  \|^mY_{m,l}(\hat{x})\sum_{j=0}^K   \tilde{\beta}_{j}^{m,n} P_j^{(m)}(2\|x\|^2-1)$.
\State Evaluate the prolate eigenvalue $\tilde{\alpha}_{m,n}(c)$ according to \eqref{eigenvalues} .
\end{algorithmic}
\end{algorithm} 
\subsection{Computation of inverse solution in the low-rank space} \label{subsection: numerical low-rank solution}
To evaluate the low-rank approximation \eqref{section stability formula pi_alpha q}, one needs to evaluate the projection $\langle u_b^\delta, \psi_{m,n,\ell} \rangle$. Here we follow the Gaussian product quadrature where the quadrature points are given by 
$$
\mathbb{P} :=\left\{ \sqrt{(1+t_i)/2}
\left(\begin{matrix}
  \sin\theta_s\cos\phi_j    \\
    \sin\theta_s\sin\phi_j\\
    \cos\theta_s
\end{matrix}
\right): \quad 
\begin{matrix}
  i = 0,1,\cdots, T-1    \\
    s = 0,1,\cdots, M_{\theta}-1 \\
    j = 0,1,\cdots, M_{\phi}-1
\end{matrix}
\right\} =\{p_n: n = 1,2,\cdots, TM_{\theta} M_{\phi}\},
$$
here each $p_n$ is uniquely associated with an index $(i,s,j)$, and $\{t_i, \theta_s, \phi_j\}_{i=0,s=0,j=0}^{T-1,M_{\theta}-1, M_{\phi}-1}$ are chosen according to
\begin{eqnarray}  
\left\{t_i,\omega_{t_i}\right\}_{i=0}^{T-1},\left\{\cos\theta_s,\omega_{\theta_s}\right\}_{s=0}^{M_\theta-1}: &&\mbox{Gauss-Legendre quadrature nodes and weights},    \label{quadrature: GL} \\
\left\{\phi_j=\frac{2j\pi}{M_\phi},\omega_{\phi_j}=\frac{2\pi}{M_\phi}\right\}_{j=0}^{M_\phi-1}: &&\mbox{trapezoidal quadrature nodes and weights}. \label{quadrature: Trapezoidal} 
\end{eqnarray}
With such a quadrature, the projection is evaluated by
\begin{equation*}
    \begin{aligned}
        \langle u_b^\delta, \psi_{m,n,\ell} \rangle &\approx \frac{\pi}{2\sqrt{2}M_\phi}\sum_{i=0}^{T-1}\sum_{j=0}^{M_\phi-1}\sum_{s=0}^{M_\theta-1} u_b^\delta(p_n;c)\psi_{m,n,l}(t_i,\phi_j,\theta_s) (1+t_i)^{1/2} w_{t_i}w_{\theta_s},
    \end{aligned}
\end{equation*}
where again $p_n$ is uniquely associated with an index $(i,s,j)$.

The data
$\{u_b^\delta(p_n;c): p_n \in \mathbb{P},~ n = 1,2,\cdots, TM_\phi M_\theta\}$ will be determined by  the   full or   Born  noisy far field data 
via the relation \eqref{section intro Born processed data} 
in the continuous case. In practice when the numbers of observation and incident directions are finite, we look for approximate or mock-quadrature nodes $\{\tilde{p}_n\}_{n=1}^{T\times M_\phi\times M_\theta}$ given by
\begin{equation}\label{eq: exctraction}
\tilde{p}_n=\frac{\hat{\theta}_{\ell^*}-\hat{x}_{j^*}}{2},~\mbox{where } (\ell^*,j^*)=\underset{1\leq j\leq N_1,~1\leq \ell\leq N_2}{\rm argmin}\Big\|p_n-\frac{\hat{\theta}_{\ell}-\hat{x}_{j}}{2}\Big\|_2,
\end{equation}
and we approximate the \textit{processed data} set $\{u_b^\delta(p_n;c): p_n \in P, n = 1,2,\cdots, TM_\phi M_\theta\}$ by 
\begin{equation}\label{eq: processed data}
\{u_b^\delta(\tilde{p}_n;c)=\frac{4\pi}{k^2}u^{\infty,\delta}(\hat{x}_{j^*};\hat{\theta}_{\ell^*};k):\tilde{p}_n=\frac{\hat{\theta}_{\ell^*}-\hat{x}_{j^*}}{2} ,1\leq n\leq T M_\phi  M_\theta\}.
\end{equation}
If the far field data set is of large scale such as when $N_1\times  N_2 \ge T\times M_\phi\times M_\theta$,  one can process the  far field data set of dimension $N_1\times  N_2$  to get a smaller processed data set through the relation \eqref{eq: exctraction}--\eqref{eq: processed data};  Otherwise if the far field data are limited, one can extend the data set through nearest interpolation for simplicity following by applying the relation \eqref{eq: exctraction}--\eqref{eq: processed data}. It is worth mentioning other options to use specific quadrature such as \cite{greengard2024generalized}, which is a possible future direction.

As a result, the numerical low-rank approximation is given by
\begin{equation*}
    \begin{aligned}
        q^{\sigma,\delta} &= \sum_{|\alpha_{m,n}|>\sigma}\frac{1}{\alpha_{m,n}}  \Big(\frac{\pi}{2\sqrt{2}M_\phi}\sum_{i=0}^{T-1}\sum_{j=0}^{M_\phi-1}\sum_{s=0}^{M_\theta-1} u_b^\delta(\tilde{p}_n;c)\psi_{m,n,l}(t_i,\phi_j,\theta_s) (1+t_i)^{1/2} w_{t_i}w_{\theta_s} \Big)\psi_{m,n,\ell},
    \end{aligned}
\end{equation*}
where each $\tilde{p}_n$ is uniquely associated with an index $(i,s,j)$.
We summarize the low-rank computational framework, from data entry to output, by   \Cref{Algorithm: LowRankIP}.
\begin{algorithm} 
	\caption{Low rank solution for inverse scattering problem}
\begin{algorithmic}[1] \label{Algorithm: LowRankIP}
\Require Wave number $k$, noisy full  far field data 
$\{u^{\infty,\delta}(\hat{x}_j;\hat{\theta}_\ell;k):j=1,2\cdots,N_1,\ell=1,2\cdots,N_2\}$,
or noisy Born far field pattern data 
$\{u_b^{\infty,\delta}(\hat{x}_j;\hat{\theta}_\ell;k):j=1,2\cdots,N_1,\ell=1,2\cdots,N_2\}$.
\Ensure  The  low-rank approximation $q^{\sigma,\delta}$ of the unknown.
\State Set $c=2k$ and spectral cutoff regularization parameter $\sigma$. Set $\mathcal{J}_\sigma =\{(m,n,\ell): |\alpha_{m,n}|>\sigma\}$.
\State (Precomputing) Apply \Cref{Algorithm: PSWFs} to evaluate the 3D PSWFs at quadrature nodes, i.e., 
$\{\psi_{m,n,l}(p_n;c):(m,n,l)\in\mathcal{J}_\sigma,1\leq n\leq T M_\phi M_\theta\}$ 
and the prolate eigenvalues
$\{\alpha_{m,n}(c):(m,n,0)\in\mathcal{J}_\sigma\}$.
\State Apply \eqref{eq: exctraction}--\eqref{eq: processed data} to obtain the  processed data $\{u_b^\delta(\tilde{p}_n;c)\}_{n=1}^{T  M_\phi  M_\theta}$.
\State Calculate the projection of processed data on the low-rank space   via \eqref{quadrature: GL}--\eqref{quadrature: Trapezoidal} and
\begin{equation*}
    \begin{aligned}
        u^{\sigma,\delta}_{m,n,l}&\approx \frac{\pi}{2\sqrt{2}M_\phi}\sum_{s=0}^{M_\theta-1} \sum_{i=0}^{T-1}\sum_{j=0}^{M_\phi-1} u_b^\delta(\tilde{p}_n;c)\psi_{m,n,l}(t_i,\phi_j,\theta_s) (1+t_i)^{1/2} w_{t_i}w_{\theta_s}, \qquad (m,n,l)\in\mathcal{J}_\sigma.
    \end{aligned}
\end{equation*}
\State Evaluate $q^{\sigma,\delta}_{m,n,l}= u^{\sigma,\delta}_{m,n,l}/\alpha_{m,n}$ to obtain the low-rank approximation 
$$q^{\sigma,\delta}(x;c)=\sum_{(m,n,l)\in\mathcal{J}_\sigma}q^{\sigma,\delta}_{m,n,l}\psi_{m,n,l}(x;c).$$
\end{algorithmic}
\end{algorithm}

\section{Numerical experiments} \label{section: numerical}
In this section, we provide various numerical examples using both  the Born far field data \eqref{born far field} and the full far field data \eqref{full far field}. We will illustrate the potential of the proposed method via studies of imaging resolution, robustness, comparison with an iterative method, regularization using customized penalty, and localized imaging.

\subsection{Data generation}
\subsubsection{Born processed data} \label{section numeric: born processed data}
Let us start with the simple case of Born processed data. One can generate the Born processed data by directly calculating the restricted Fourier integral \eqref{section intro data q to u_b} with the following given contrasts $q$. 
\begin{itemize}
\item Constant in a ball $q(x)=1_{x\in \Omega}(x)$, $\Omega=\{x\in \mathbb{R}^3: \|x\|_2<a\},~a<1$. This leads to 
\begin{equation*}
    u_b(x;c)=\left(\frac{2a\pi}{c\|x\|}\right) ^{3/2}J_{3/2}(ac\|x\|),
\end{equation*}
where $J_{3/2}$ is Bessel function of the first kind of order $\nu=3/2$, see \cite{Abramowitz64}. In the following numerical experiments, we set $a=0.5$.
\item Constant in a cube $q(x)=1_{x\in \Omega}(x)$, $\Omega=\{x\in \mathbb{R}^3:a_j\leq x_j\leq b_j,j=1,2,3\}\subset B(0,1)$. This leads to 
\begin{equation*}
    u_b(x;c)=\prod_{j=1}^3 \frac{1}{ic x_j}(e^{ic x_j\cdot b_j}-e^{ic x_j\cdot a_j}).
\end{equation*}
We test $a_j=-1/2,~b_j=1/2$ in  our numerical experiment.

Furthermore, we will test constant contrast supported in three nearby cubes: $
q=1_{\cup_{j=1}^3 \Omega_j}$ given by
\begin{equation}\label{eq: Three cubes}
    \left\{\begin{aligned}
        \Omega_1&:=\{(x_1,x_2,x_3)\in\mathbb{R}^3: |x_1|\leq 0.3,-0.5 < x_2 <  -0.025,0.1 < x_3 <  0.5\},\\
        \Omega_2&:=\{(x_1,x_2,x_3)\in\mathbb{R}^3: |x_1|\leq 0.3,0.025 <  x_2 <  0.5,0.1 <  x_3 <  0.5\},\\
        \Omega_3&:=\{(x_1,x_2,x_3)\in\mathbb{R}^3: |x_1|\leq 0.3,-0.235 <  x_2 <  0.235,-0.5 <  x_3 <  0.025\}
    \end{aligned}\right..
\end{equation}

\item  Oscillatory contrast  $q(x_1,x_2,x_3)= \sin (m\pi x_1) \cdot 1_\Omega(x)$, $\Omega=\{x\in \mathbb{R}^3:  |x_j| < 1/2,j=1,2,3\}$, $m\in \mathbb{Z}$.
This leads to
\begin{equation*}
    u_b(x;c)= -8i\left(\frac{\sin(cx_1+m\pi)/2}{cx_1+m\pi}-\frac{\sin(cx_1-m\pi)/{2}}{cx_1-m\pi} \right)\prod_{j=2}^3\frac{\sin (cx_j/2)}{cx_j} .
\end{equation*}
In the following numerical experiments, we test $m=8$.
\end{itemize}

To test robustness of the proposed method,   randomly distributed noise are added to the processed Born data $u_b$ to generate the noisy data
\begin{equation}\label{eq: add noise}
    u_b^{\tilde{\delta}}(p_n;c)=u_b(p_n;c)(1+ \delta \xi_n),~p_n \in \mathbb{P},
\end{equation}
where the set $\mathbb{P}$ is given in \Cref{subsection: numerical low-rank solution},
  $\xi_n\in [-1,1]$ is a uniformly distributed random number and $\delta\in [0,1)$ is referred to as the noise level. It follows that $\|u_b^{\tilde{\delta}} - u_b\| \le  \delta \| u_b \| = \tilde{\delta}$.
  In the following, the regularization parameter is chosen as $\sigma=\delta \alpha_{0,0}(c)$ when $\delta>0$ and $\sigma=0.1|\alpha_{0,0}(c)|$ when $\delta=0$.

\begin{figure}[htbp]  
\centering  
\subfloat[]{ \includegraphics[width=0.32\linewidth]{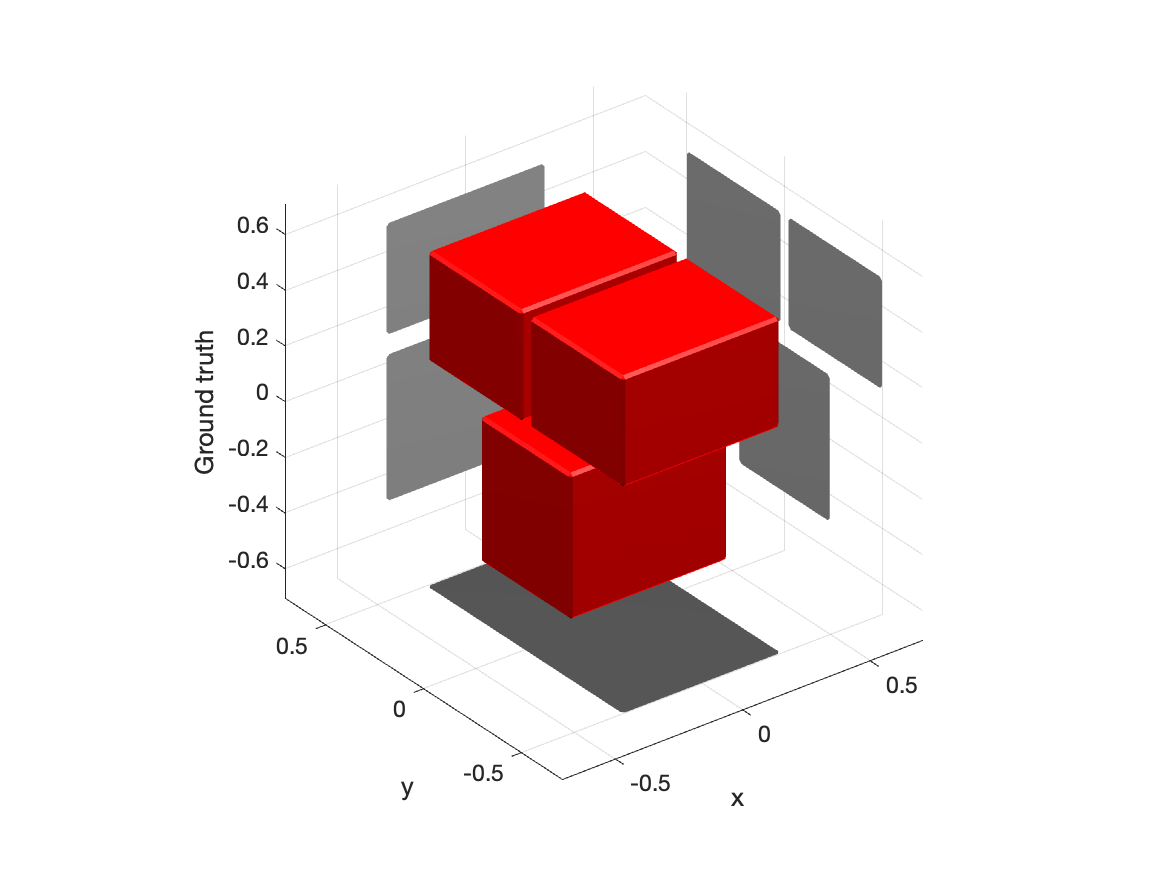} }
\subfloat[]{ \includegraphics[width=0.32\linewidth]{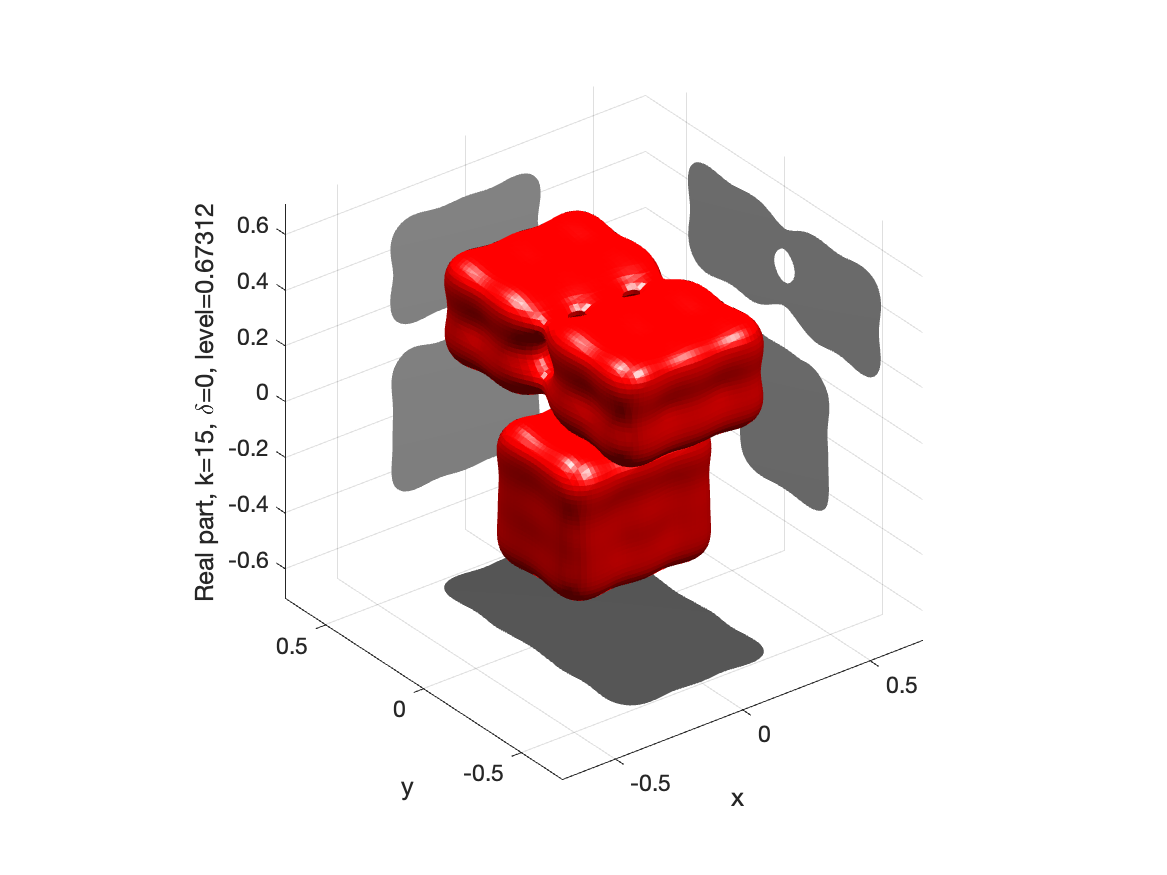} }
\subfloat[]{ \includegraphics[width=0.32\linewidth]{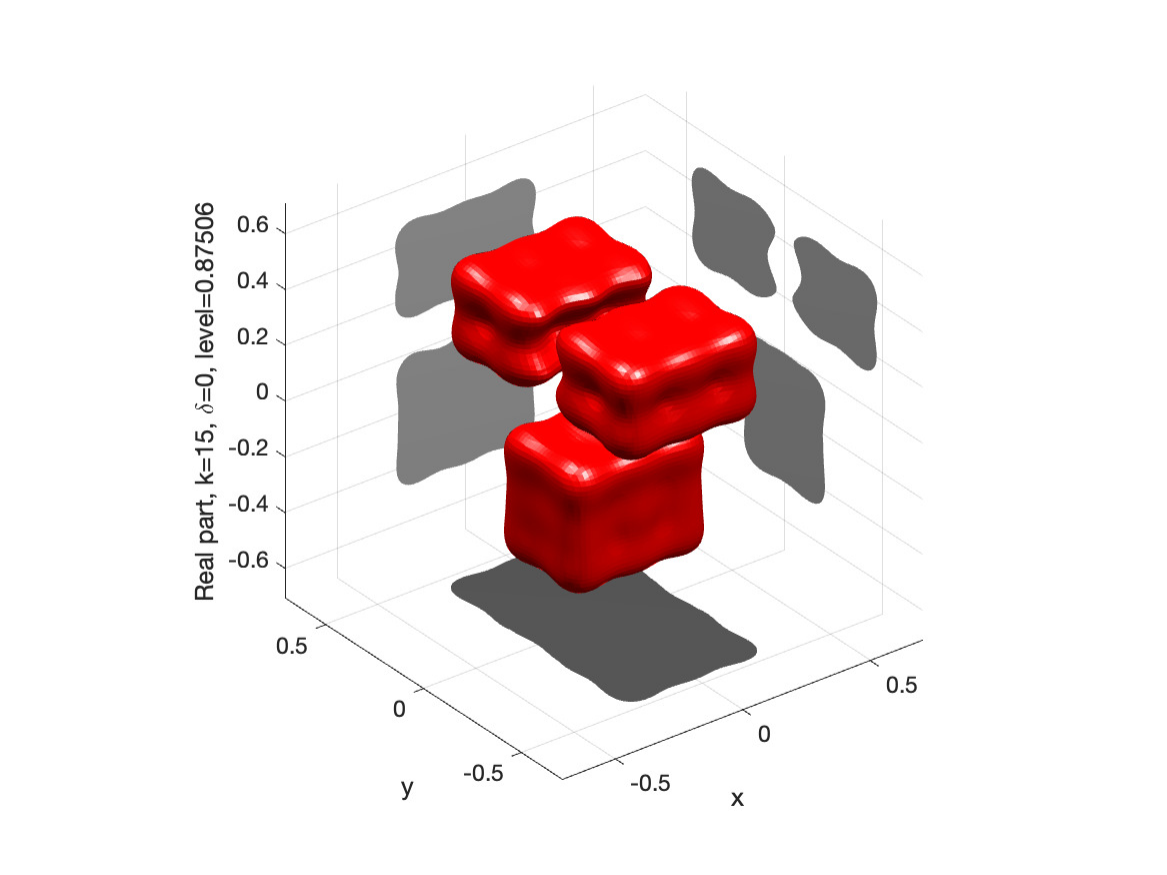} }\\
\subfloat[]{ \includegraphics[width=0.32\linewidth]{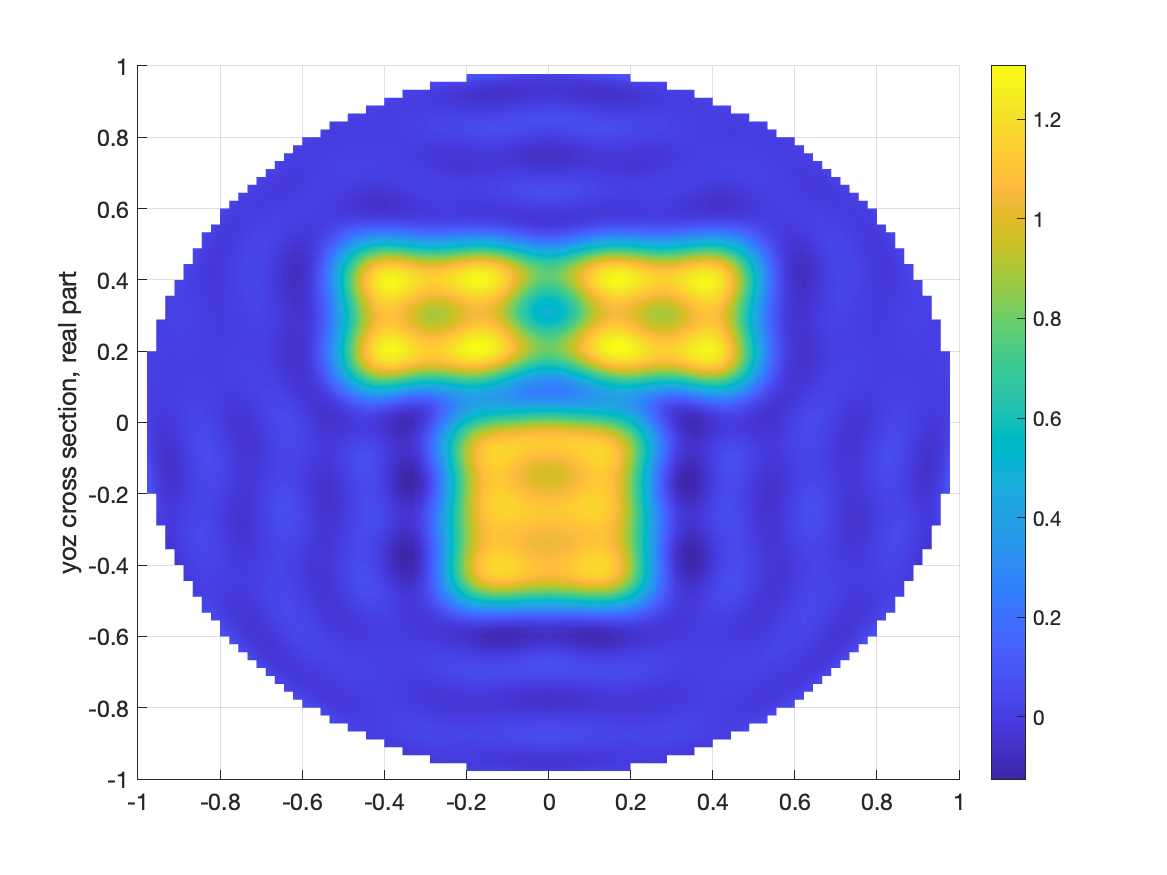} }
\subfloat[]{ \includegraphics[width=0.32\linewidth]{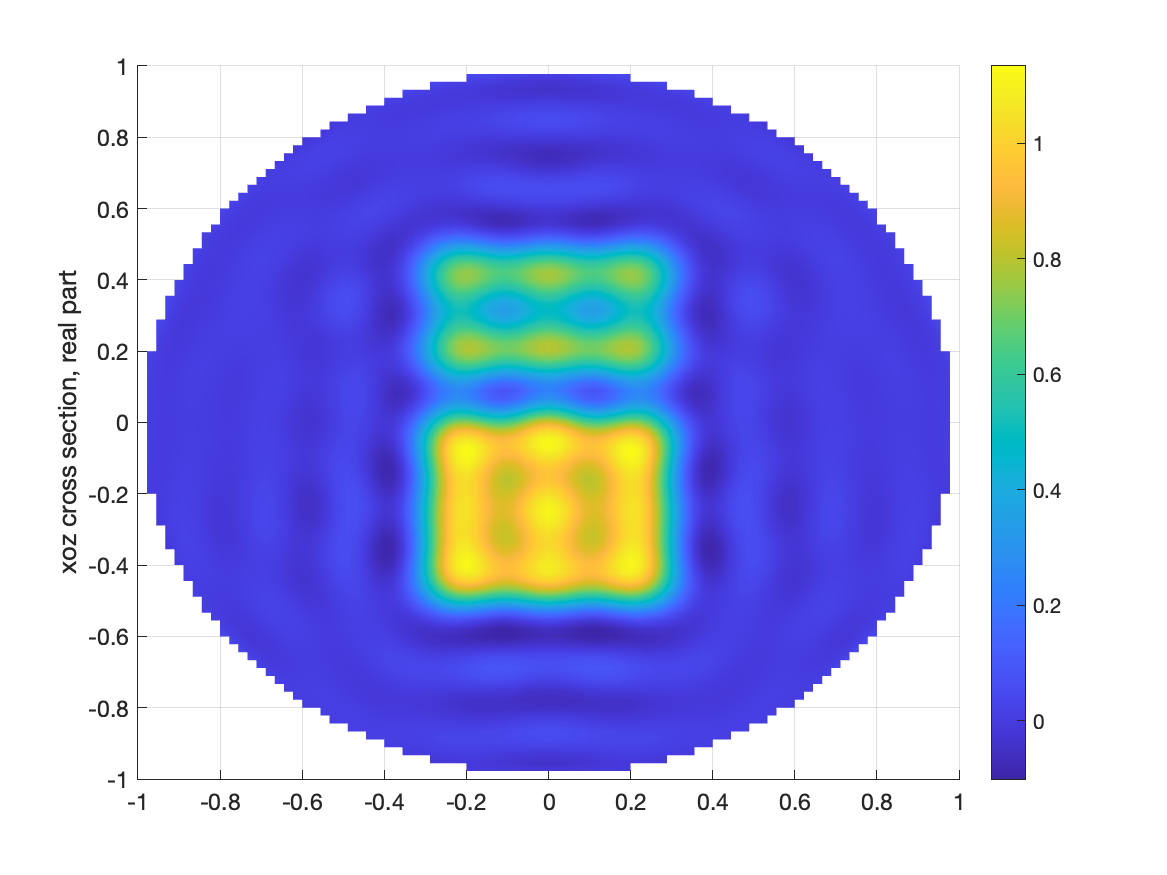} }
\subfloat[]{ \includegraphics[width=0.32\linewidth]{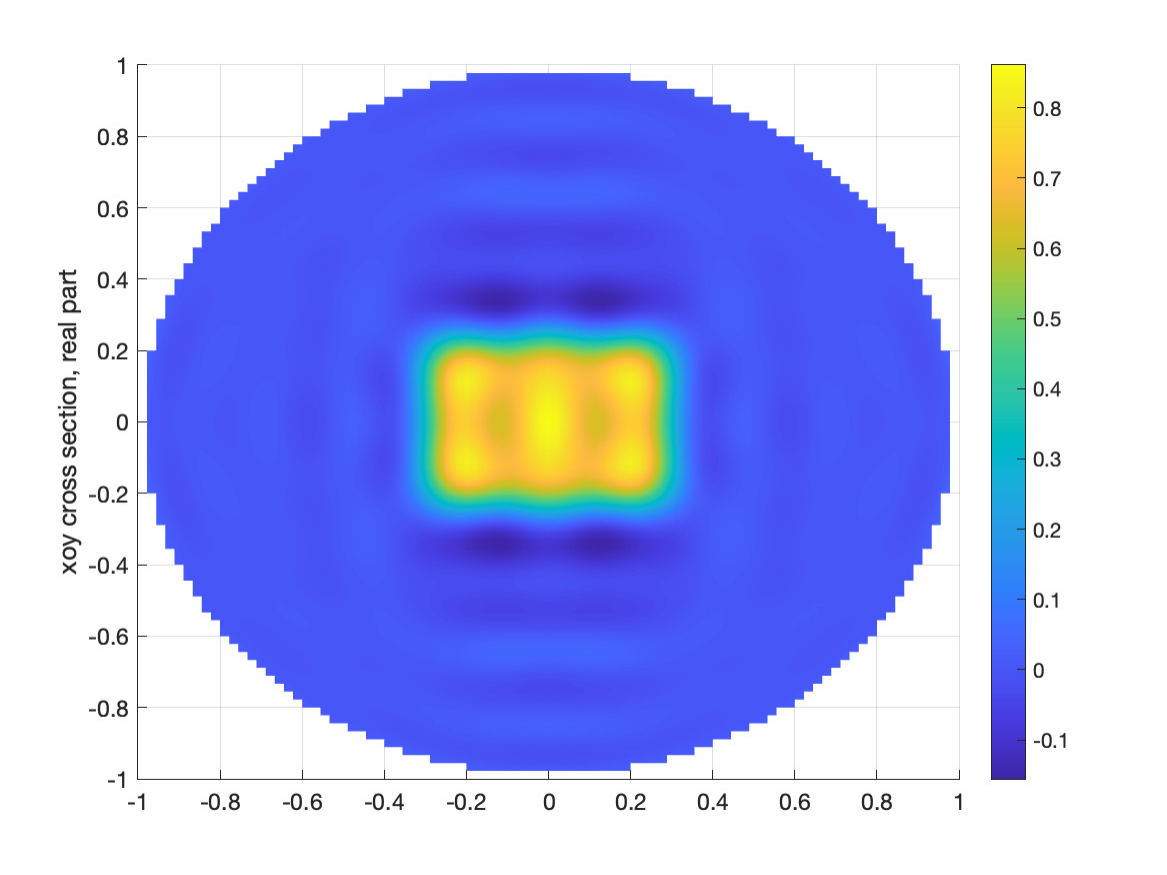} }\\
\subfloat[]{ \includegraphics[width=0.32\linewidth]{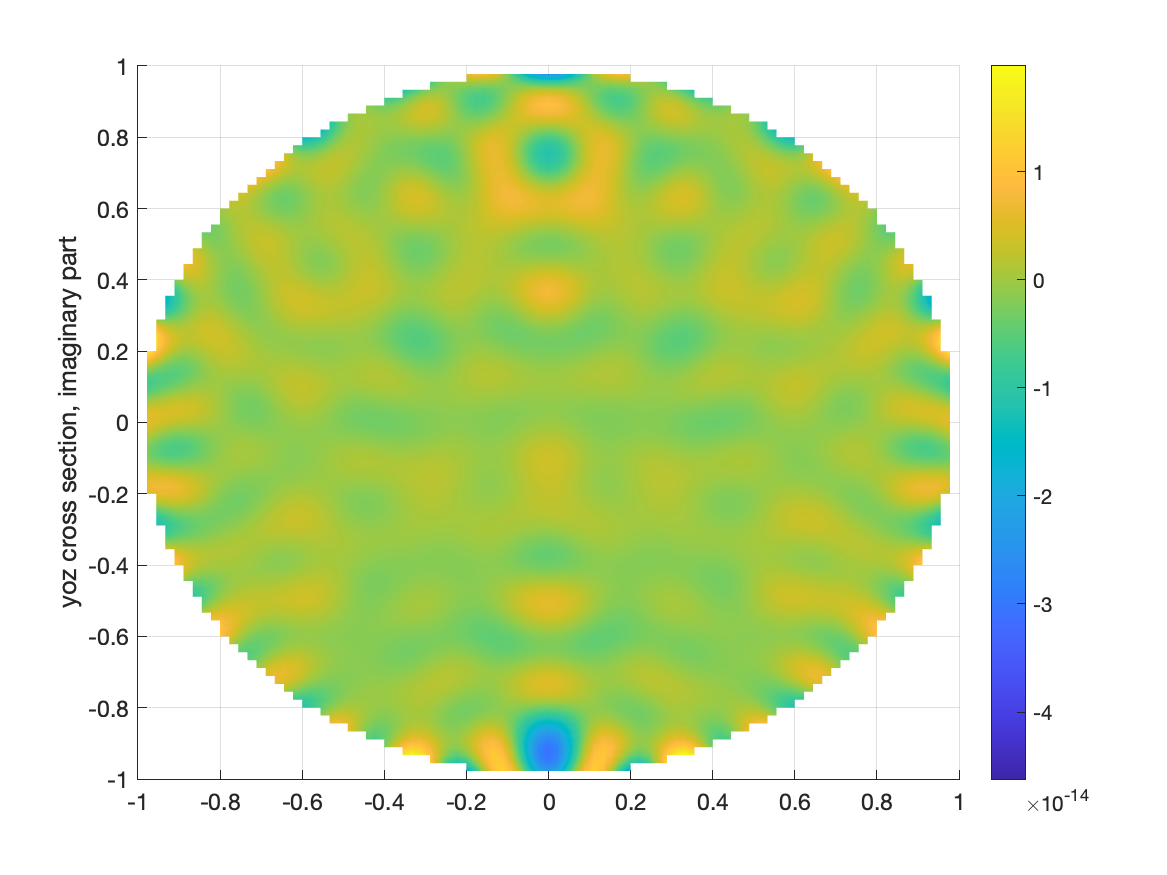} }
\subfloat[]{ \includegraphics[width=0.32\linewidth]{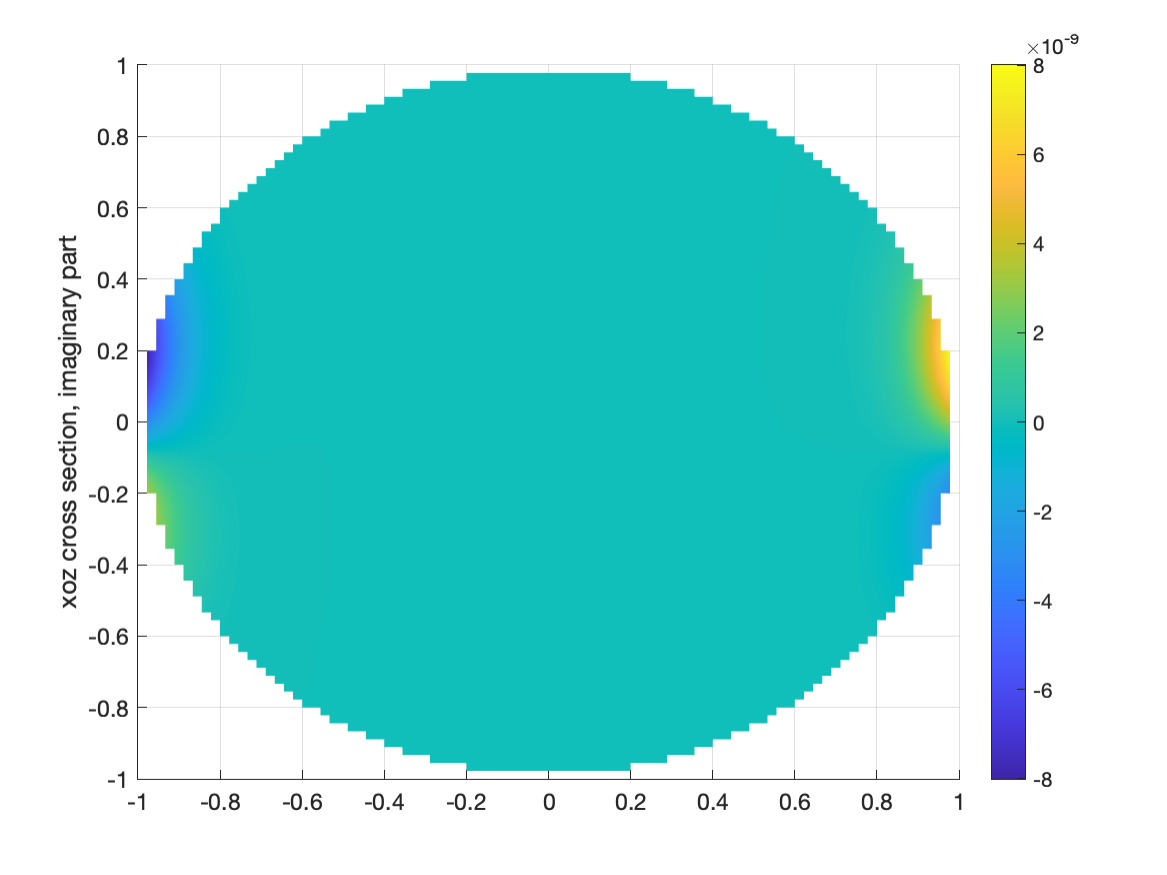} }
\subfloat[]{ \includegraphics[width=0.32\linewidth]{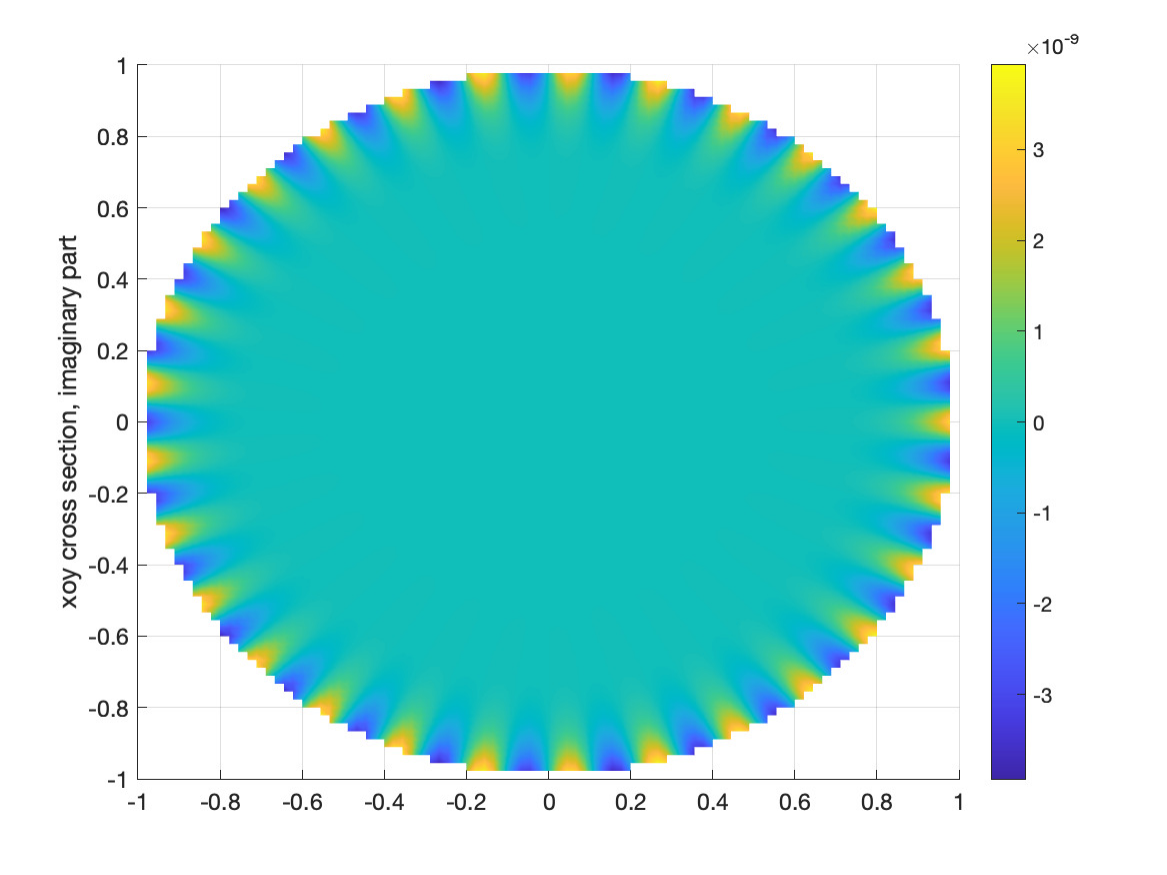} }\\
\caption{Reconstruction of three cubes with noiseless processed data, $k=15,~T=23,~M_\theta=31,~M_\phi=61$. The first row (a)   ground truth, (b) (c) isosurfaces of reconstruction (real part) with level value  $0.5M,~0.65M$, where $M=\mbox{max}~q^{\sigma,\delta}$; the second (resp. third) row, cross section view of the real (resp. imaginary) part.    }\label{figure: Three cubes, analytic data, noise free}
\end{figure}

\begin{figure}[htbp]  
\centering  
\subfloat[]{ \includegraphics[width=0.32\linewidth]{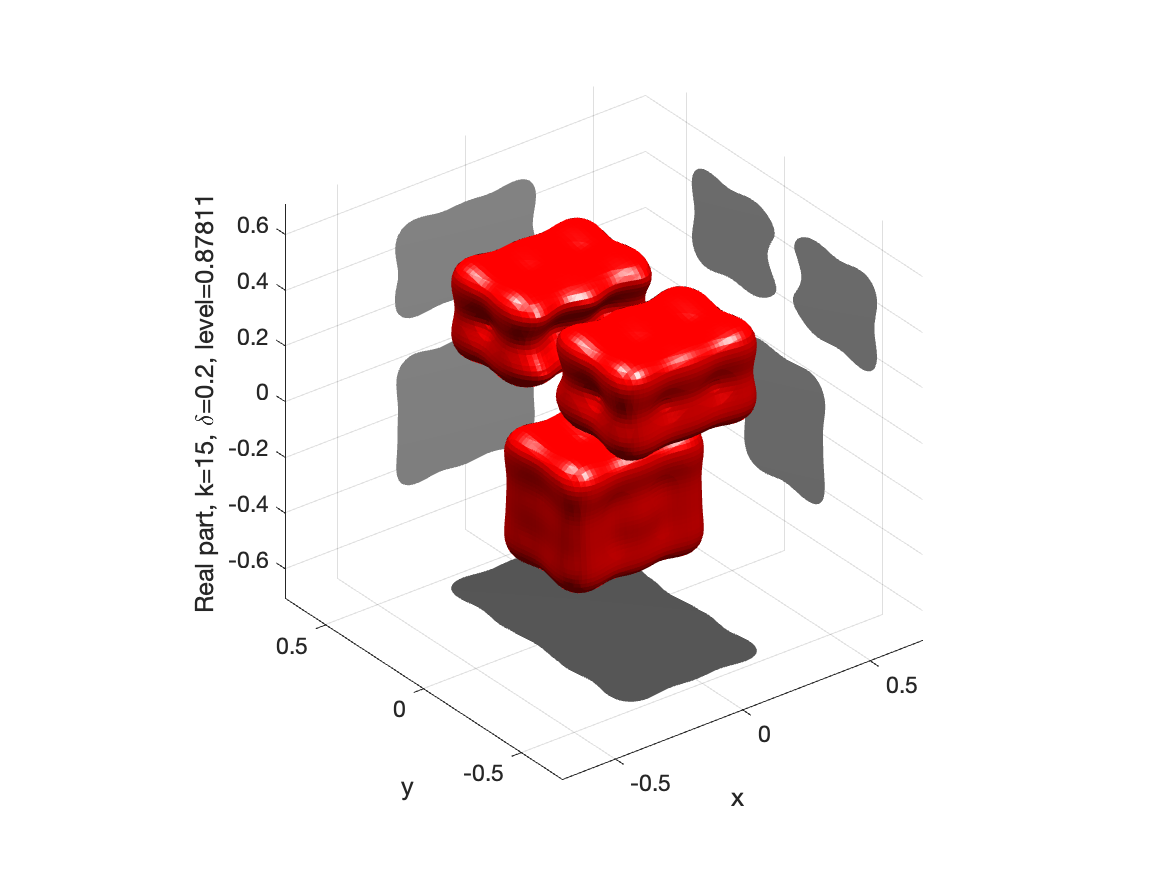} }
\subfloat[]{ \includegraphics[width=0.32\linewidth]{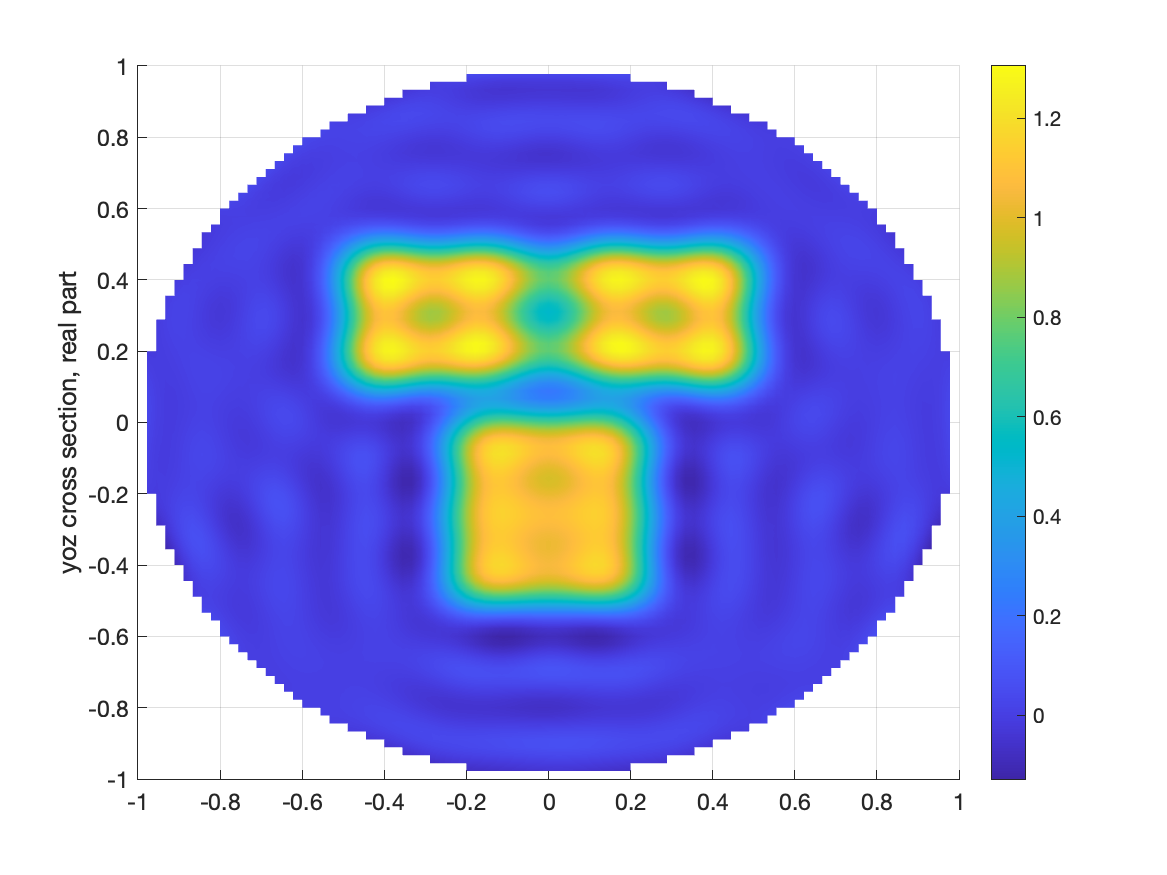} }
\subfloat[]{ \includegraphics[width=0.32\linewidth]{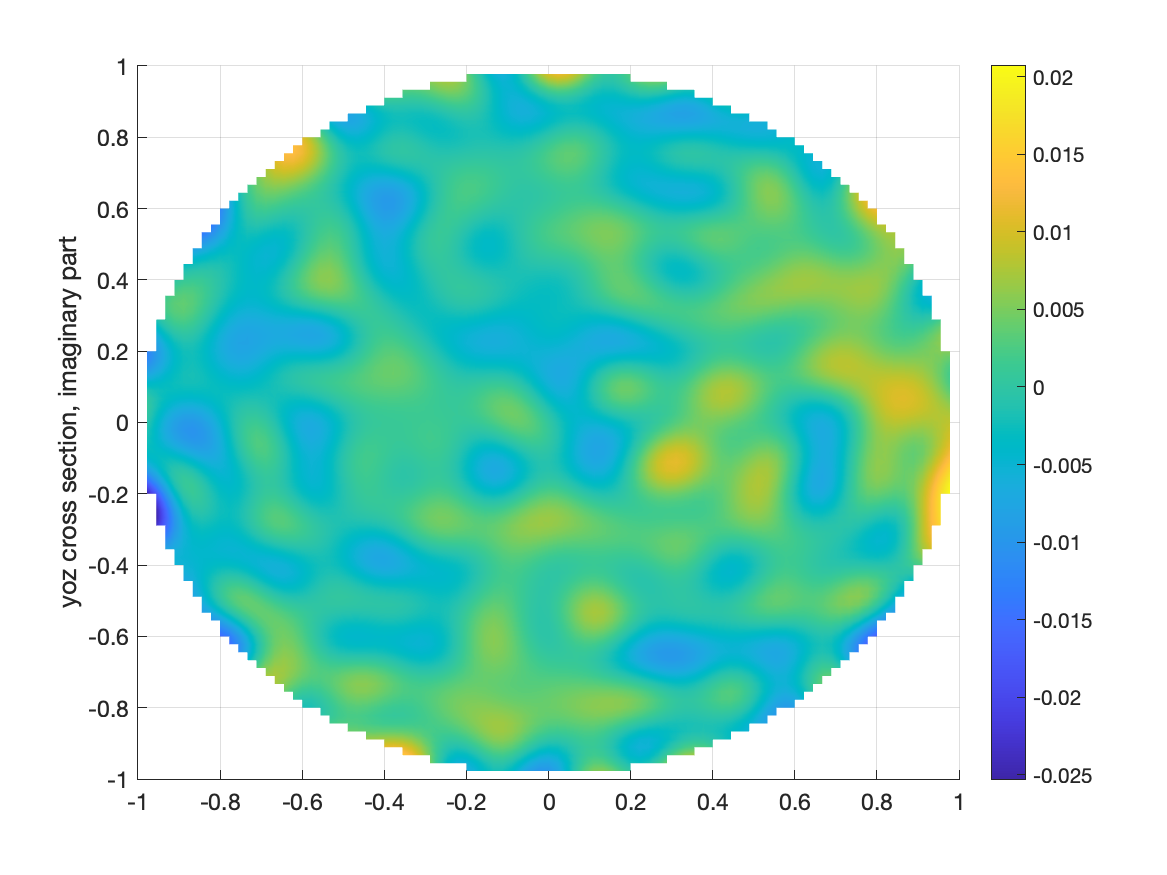} }
\caption{Reconstruction of three cubes with noisy processed data, $\delta=0.2$ and all other parameters are the same as Figure \ref{figure: Three cubes, analytic data, noise free}.   }\label{figure: Three cubes, noisy analytic data}
\end{figure}

\subsubsection{Far field data}
\textbf{Born far field data}. The discrete Born far field data are given at incident directions $\{\hat{\theta}_\ell\}_{\ell=1}^{N_2}$ and observation direction $\{\hat{x}_j\}_{j=1}^{N_1}$; both $\{\hat{\theta}_\ell\}_{\ell=1}^{N_2}$ and  $\{\hat{x}_j\}_{j=1}^{N_1}$ follow the Fibonacci lattice  on the sphere  $\mathbb{S}^2$, which was studied in geosciences \cite{gonzalez2010measurement}. Having the analytic formula in \eqref{section numeric: born processed data}, Born far field data can be   generated by,
\begin{equation}\label{eq: Born data}
u_b^\infty(\hat{x}_j;\hat{\theta}_\ell;k)=\frac{k^2}{4\pi}u_b\Big(\frac{\hat{\theta}_\ell-\hat{x}_j}{2};2k \Big),~1\leq j\leq N_1,~1\leq \ell \leq N_2.
\end{equation}

\vspace{1\baselineskip}

\noindent\textbf{Full far field data}.
To test the performance of the proposed method with the fully nonlinear model,
we use the Matlab toolbox IPscatt \cite{burgel2019algorithm} to generate the full far field data
$$
\{u^\infty(\hat{x}_j;\hat{\theta}_\ell;k):1\leq j\leq N_1,~1\leq \ell\leq N_2\}.
$$
To characterize the error between Born model \eqref{eq: Born model} and  the full model \eqref{eq: full model}, we follow \cite{burgel2019algorithm} to introduce the relative modeling error by
${\rm rel}(k)=\frac{\|U-U_b\|_2}{\|U\|_2}$,
where $U=(u^s_j(x_n))_{n=1}^{N  N_2},~U_b=(u^s_{b,j}(x_n))_{N_2 \times N}$; each $u^s_j(x_n)$ and   $u^s_{b,j}(x_n)$ with $j \in \{1,2,\cdots, N_2\}$ is the radiating solution of full model and the Born model due to the $j$-th incident wave $e^{ik\hat{\theta}_j\cdot x}$, respectively; here  $\{x_n\}_{n=1}^N,$ is a set of discretization points in the domain of interest $ [-\frac{\sqrt{2}}{2},\frac{\sqrt{2}}{2}]^3$, and $N$ is set to $91^3$ in the numerical experiments.

\vspace{1\baselineskip}

\noindent\textbf{Processed noisy data}. After generating the  far field  data, the processed data are further generated by 
$u_b(\tilde{p}_n;c)=\frac{4\pi}{k^2}u_b^\infty(\hat{x}_{j^*};\hat{\theta}_{\ell^*};k)$ according to \eqref{eq: exctraction}--\eqref{eq: processed data}, where we recall that  $\tilde{p}_n=\frac{\hat{\theta}_{\ell^*}-\hat{x}_{j^*}}{2}$ is the approximate  quadrature node.
Noisy data $u_b^{\tilde{\delta}}$ are again generated according to \eqref{eq: add noise}; in the case of full far field data, one simply repeats the above procedure with  $u_b(\cdot;c)$ being replaced   by $u(\cdot;c)$.
The regularization parameter is set to $\epsilon=0.9|\alpha_{0,0}|$ to maintain robustness, as we assume no prior information of modeling error.

\subsection{Numerical results}

\subsubsection{Reconstruction using Born processed data}
We first test the proposed method by reconstructing three cubes using directly the   Born  processed data given in \Cref{section numeric: born processed data}, with noise level  $\delta=0$ (Figure \ref{figure: Three cubes, analytic data, noise free}) and  $\delta=0.2$ (Figure \ref{figure: Three cubes, noisy analytic data}), respectively. The ground truth of the three cubes is given in Figure \ref{figure: Three cubes, analytic data, noise free}  (a), the reconstructions are plotted using both isosurface   and the cross-section views. The proposed method yields good reconstructions in this case when the smallest distance between cubes is $0.05$ which is smaller than half wavelength $\pi/k\approx 0.21$. The proposed method also performs well using noisy processed data, see Figure \ref{figure: Three cubes, noisy analytic data}.  This demonstrates the potential of the proposed method.

\subsubsection{Reconstruction using Born far field data}
The previous processed data are generated directly by analytic formula. To illustrate the effect of the data processing using \eqref{eq: exctraction}--\eqref{eq: processed data},
We next conduct experiments using the Born far field data as described in \eqref{eq: Born data} for three different cases. The number of incident and observation directions are $N_1=N_2=201$, which are distributed according to the Fibonacci lattice on the sphere. Additionally, noises are added to the data through \eqref{eq: add noise}. In Figure \ref{figure:Born data}, the left column is for a ball, the middle column is for a cube, and the right column is for an oscillating contrast, respectively. The top row shows the ground truth, the middle row shows the isosurface view of the reconstruction, and the bottom row shows the cross section view of the reconstruction, respectively.

\begin{figure}[htbp]  
\centering  
\subfloat[]{ \includegraphics[width=0.32\linewidth]{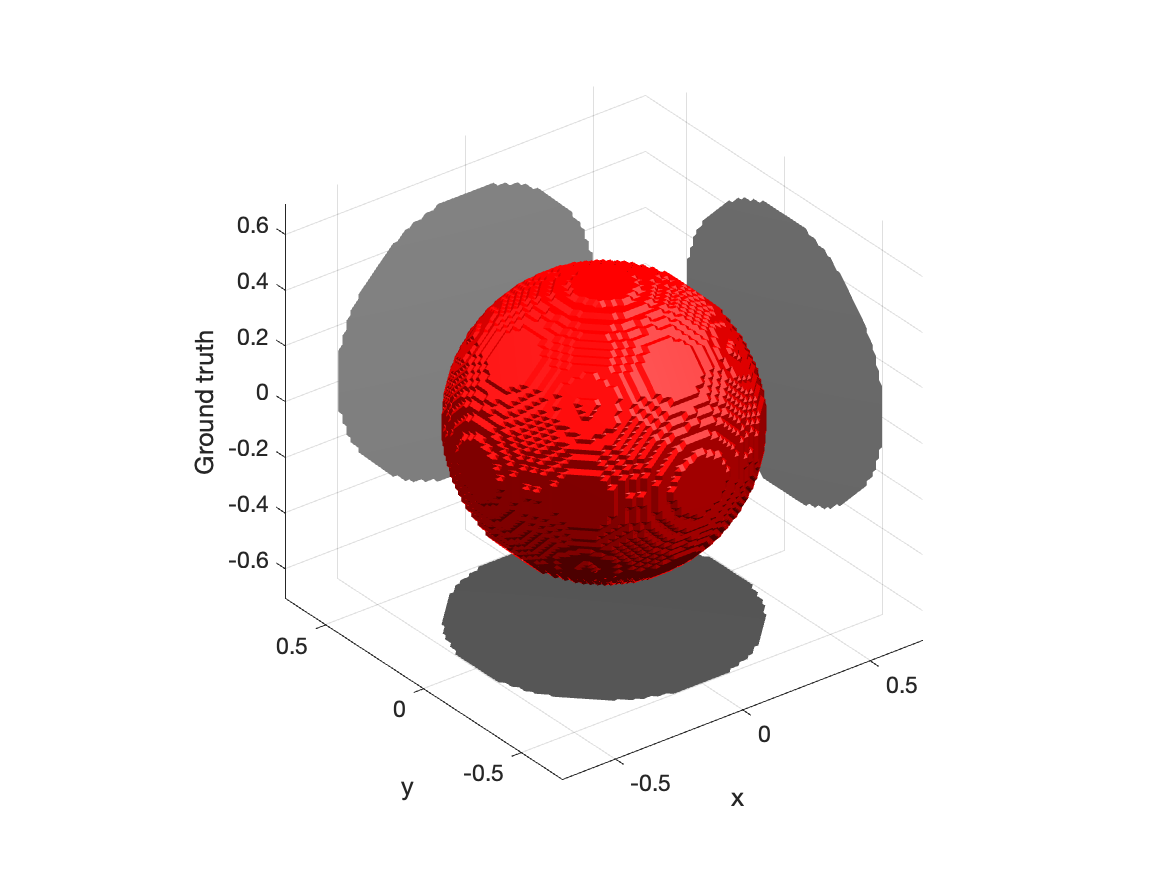} }
\subfloat[]{ \includegraphics[width=0.32\linewidth]{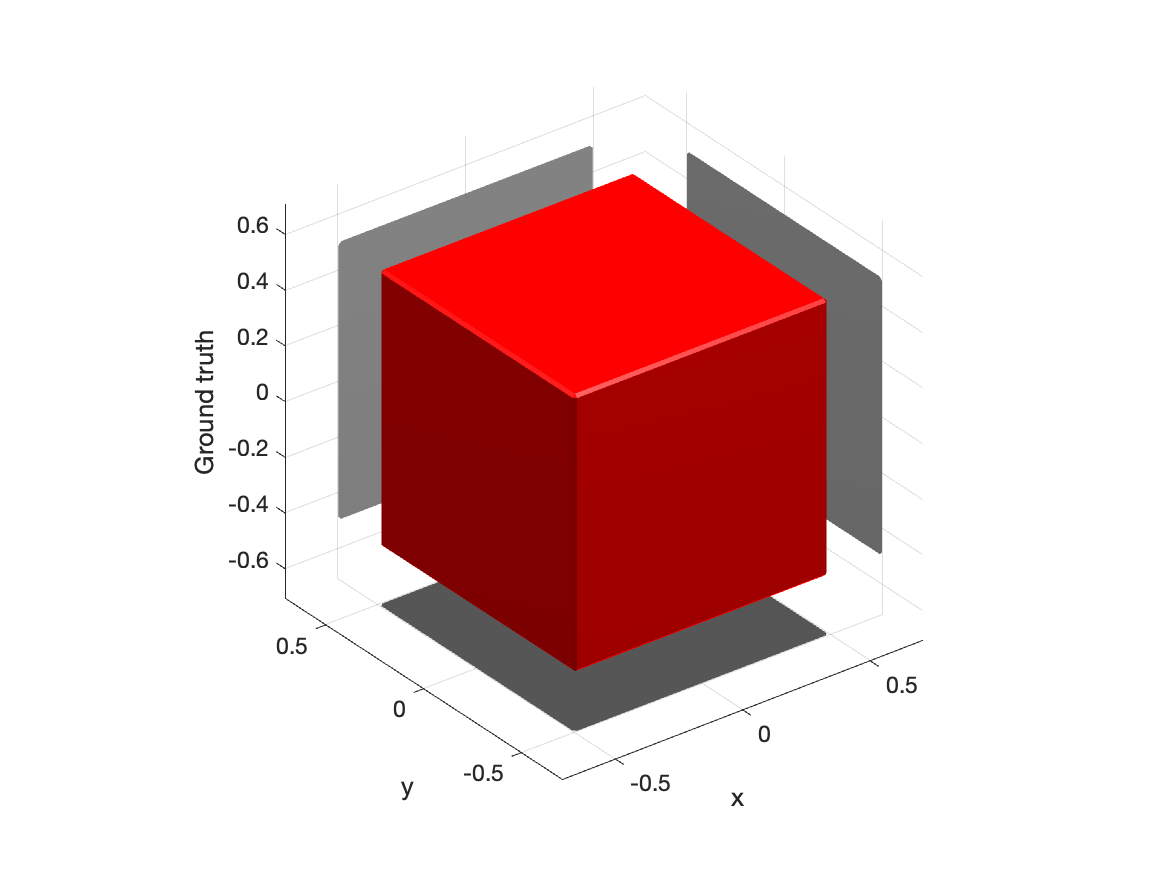} }
\subfloat[]{ \includegraphics[width=0.32\linewidth]{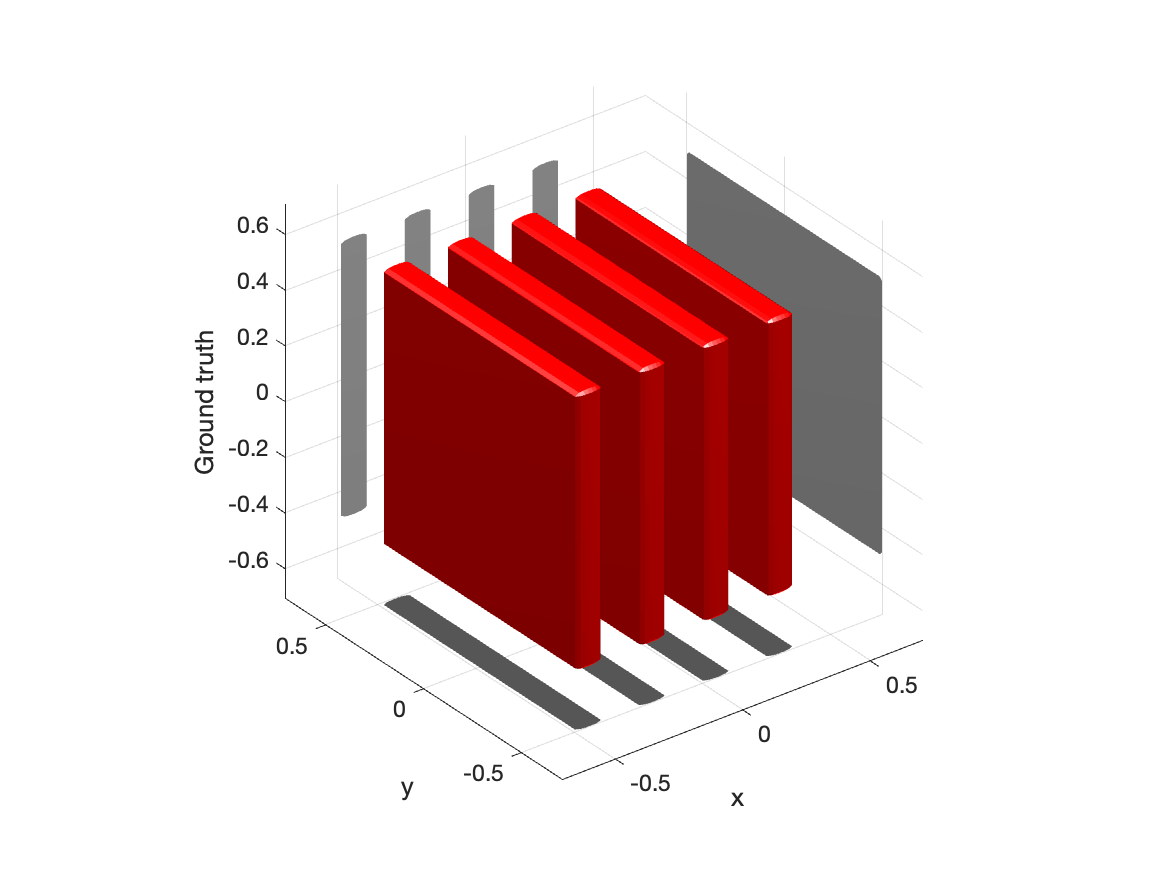} }\\
\subfloat[]{ \includegraphics[width=0.32\linewidth]{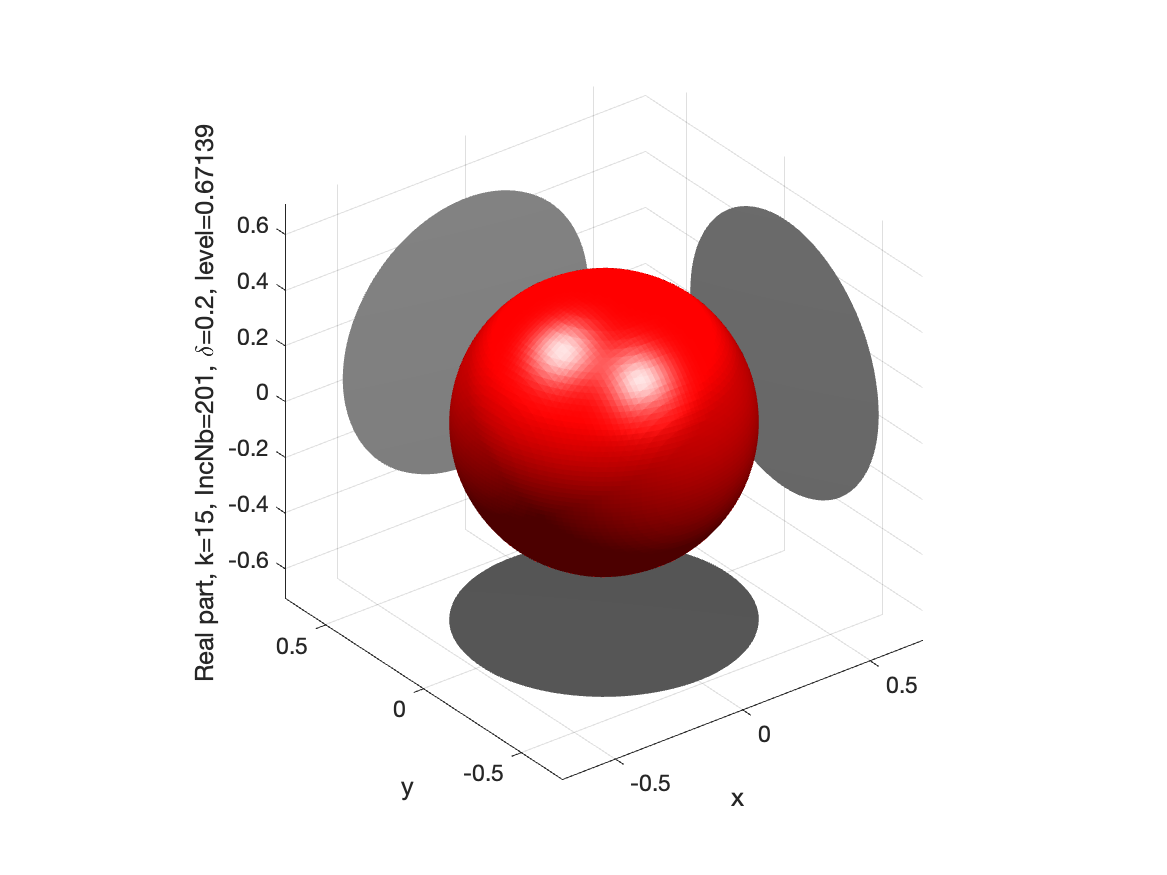} }
\subfloat[]{ \includegraphics[width=0.32\linewidth]{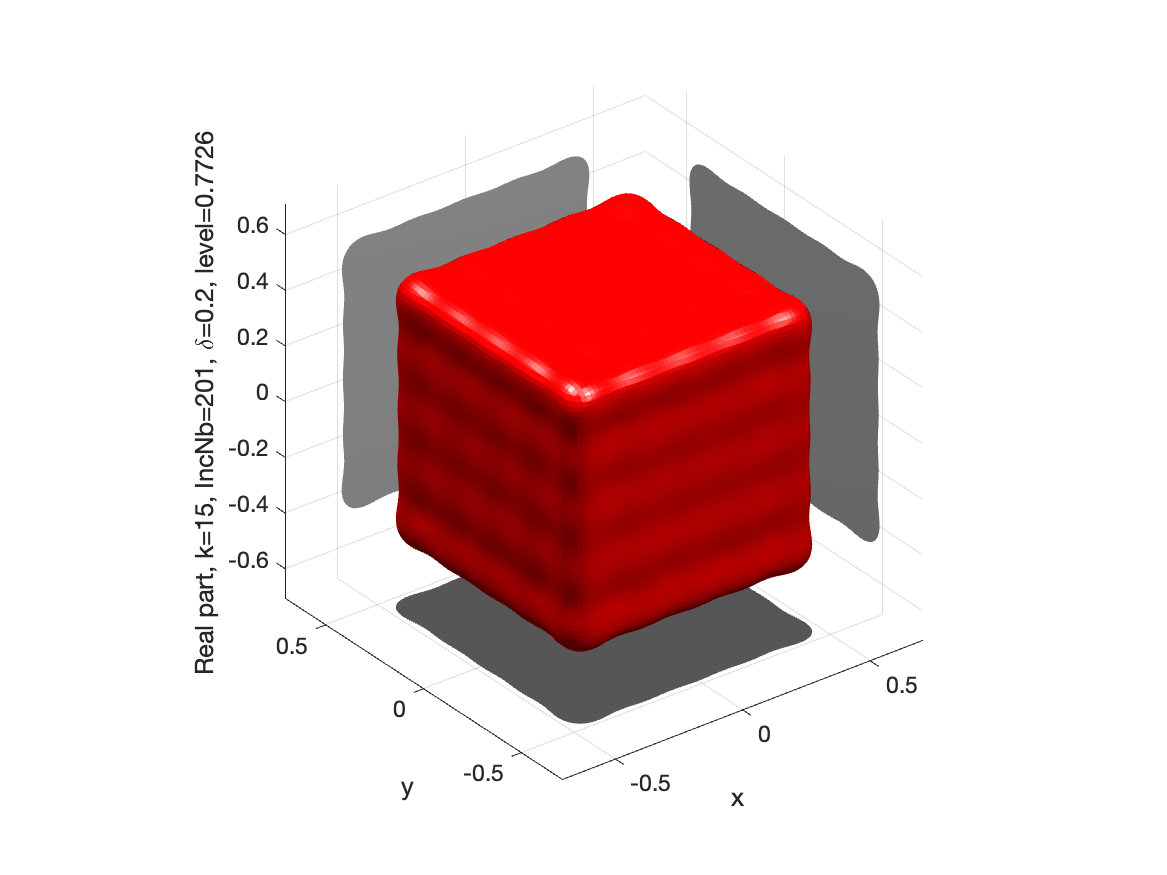} }
\subfloat[]{ \includegraphics[width=0.32\linewidth]{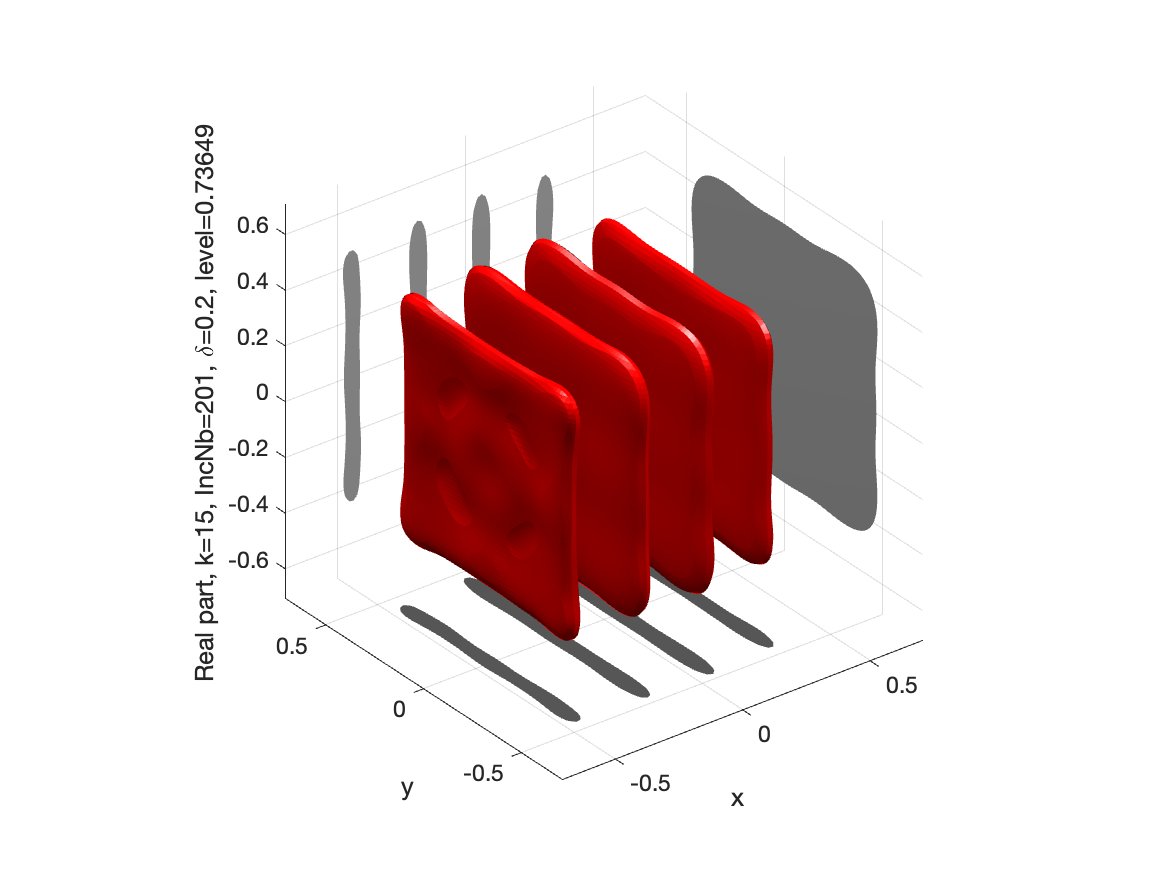} }\\
\subfloat[]{ \includegraphics[width=0.32\linewidth]{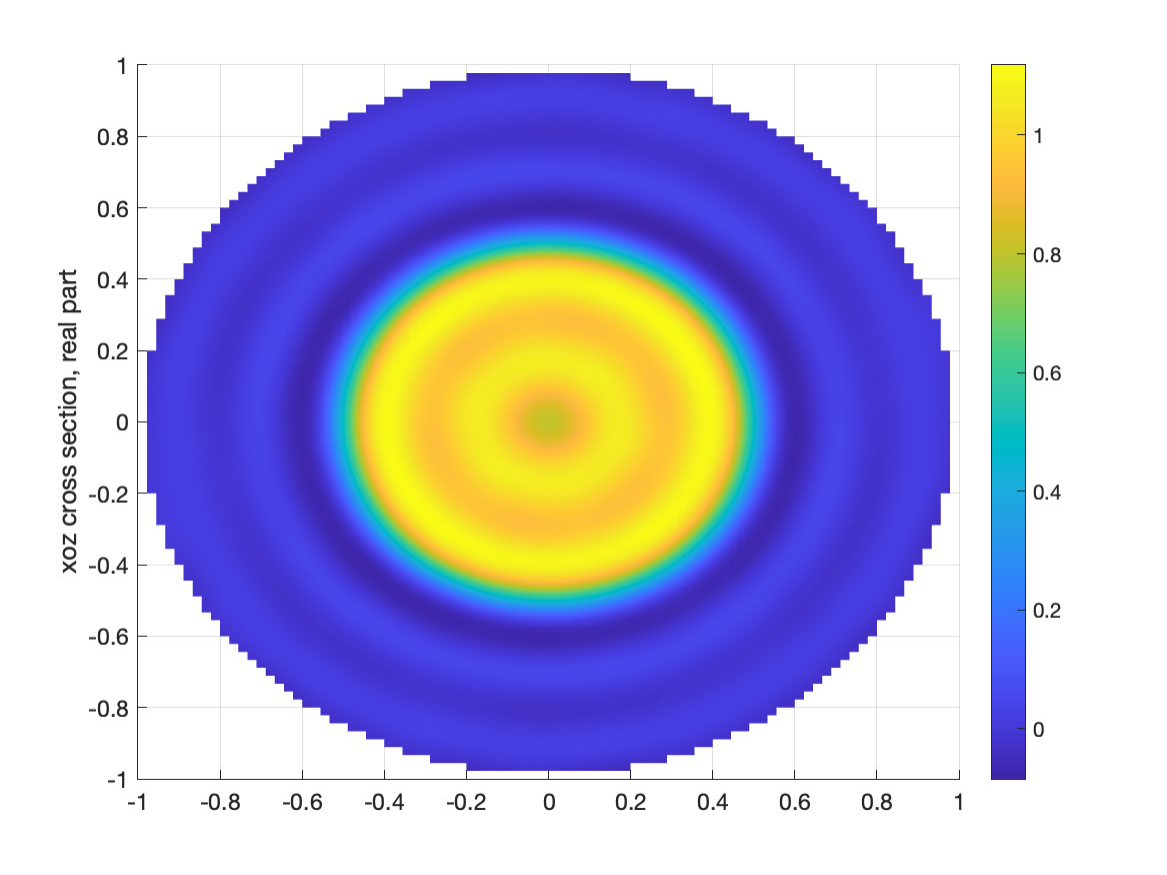 } }
\subfloat[]{ \includegraphics[width=0.32\linewidth]{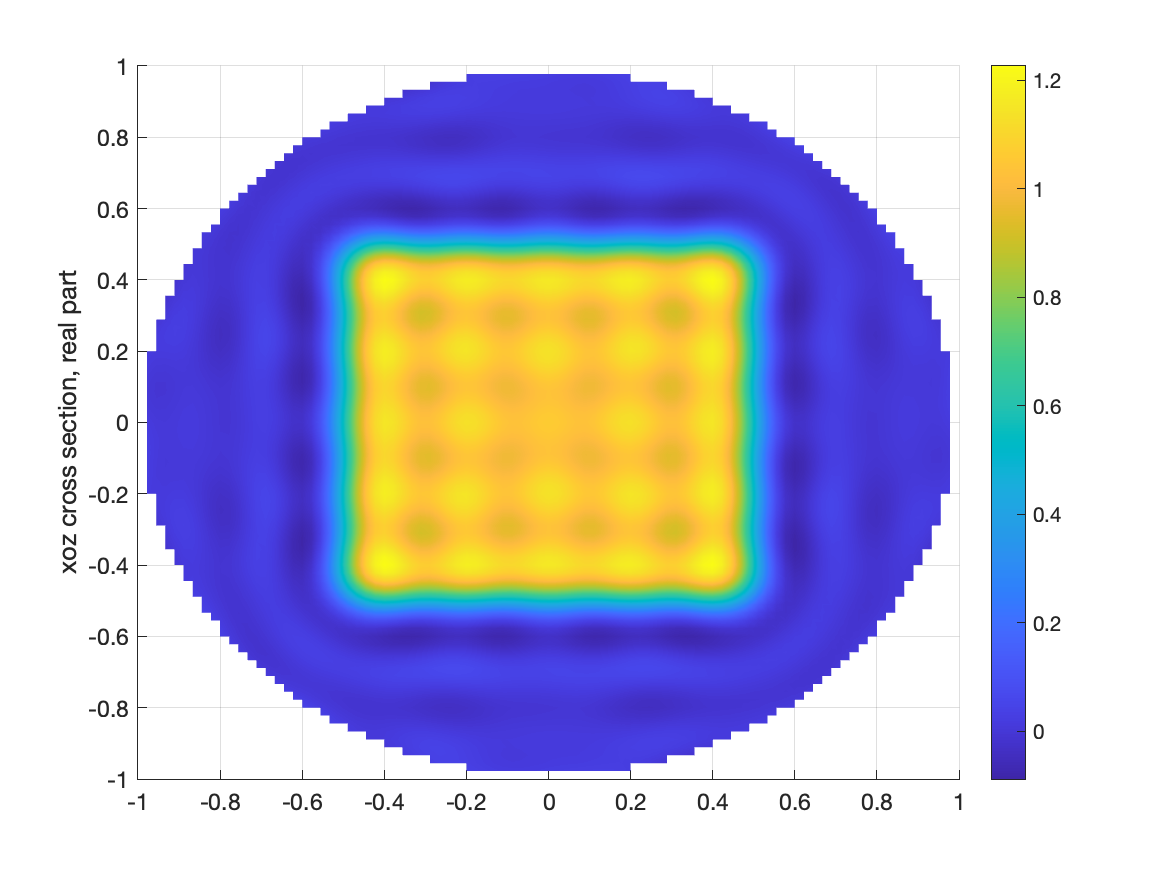} }
\subfloat[]{ \includegraphics[width=0.32\linewidth]{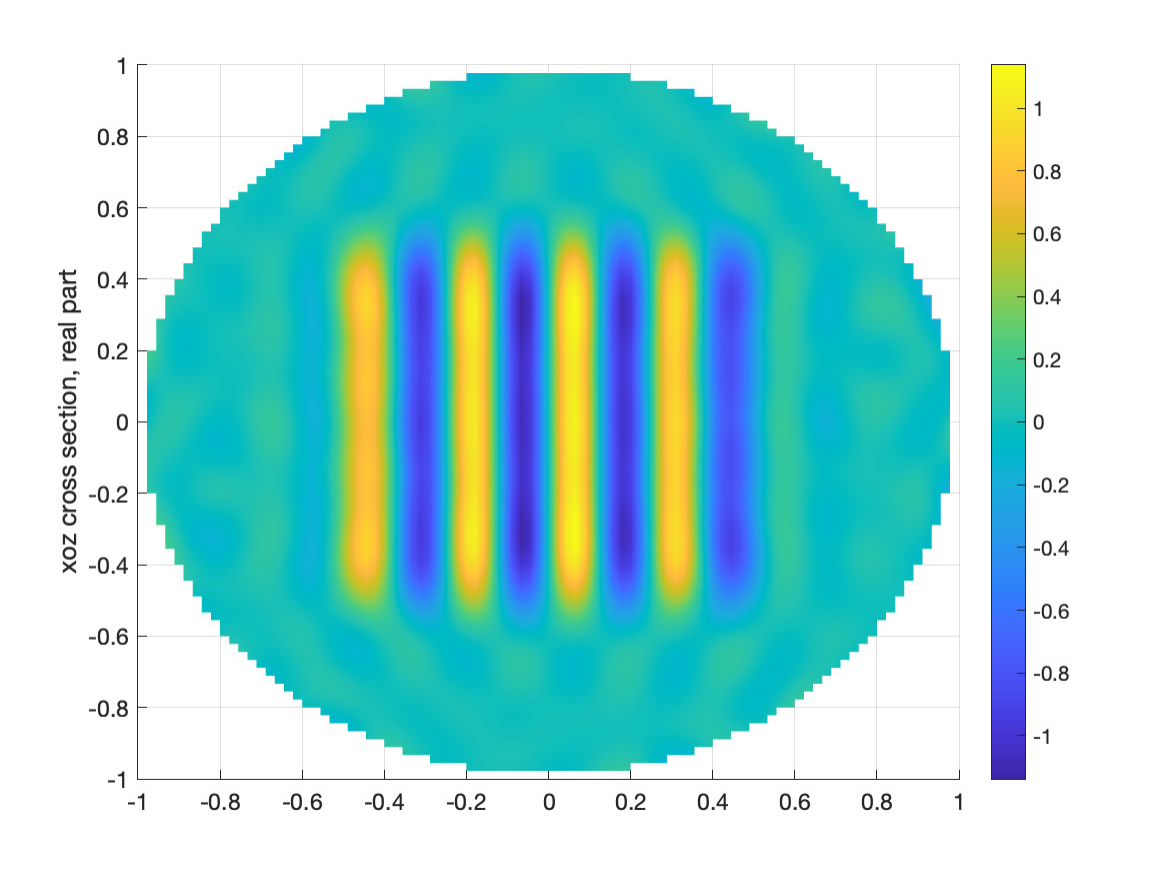} }
\caption{Reconstruction of three contrasts with  Born far field data with $k=15,~\delta=0.2$, and $N_1=N_2=201$. The first row:  isosurface view of ground truth;  the second row: isosurface of reconstruction; the third row:
cross section view of the reconstruction  of a ball, a cube and the oscillatory contrast, respectively.  }\label{figure:Born data}
\end{figure}

\subsubsection{Reconstruction using full far field data (nonlinear model)}
After successfully testing the Born data, we shift our focus to the full far field data where the modeling error is not overly large. In particular the largest amplitudes of the contrast ``cross3D" and ``three cubes" are $0.03$ and $0.02$, respectively. To shed light on the changes of resolution with respect to the wave number, we test $k\in \{10,~15,~20\}$. As is observed in Figure\ref{figure:full data cross3D} and Figure \ref{figure:full data 3cubes3D}, the resolution improves as the wave number $k$ increases.

\begin{figure}[htbp]  
\centering  
\subfloat[]{ \includegraphics[width=0.24\linewidth]{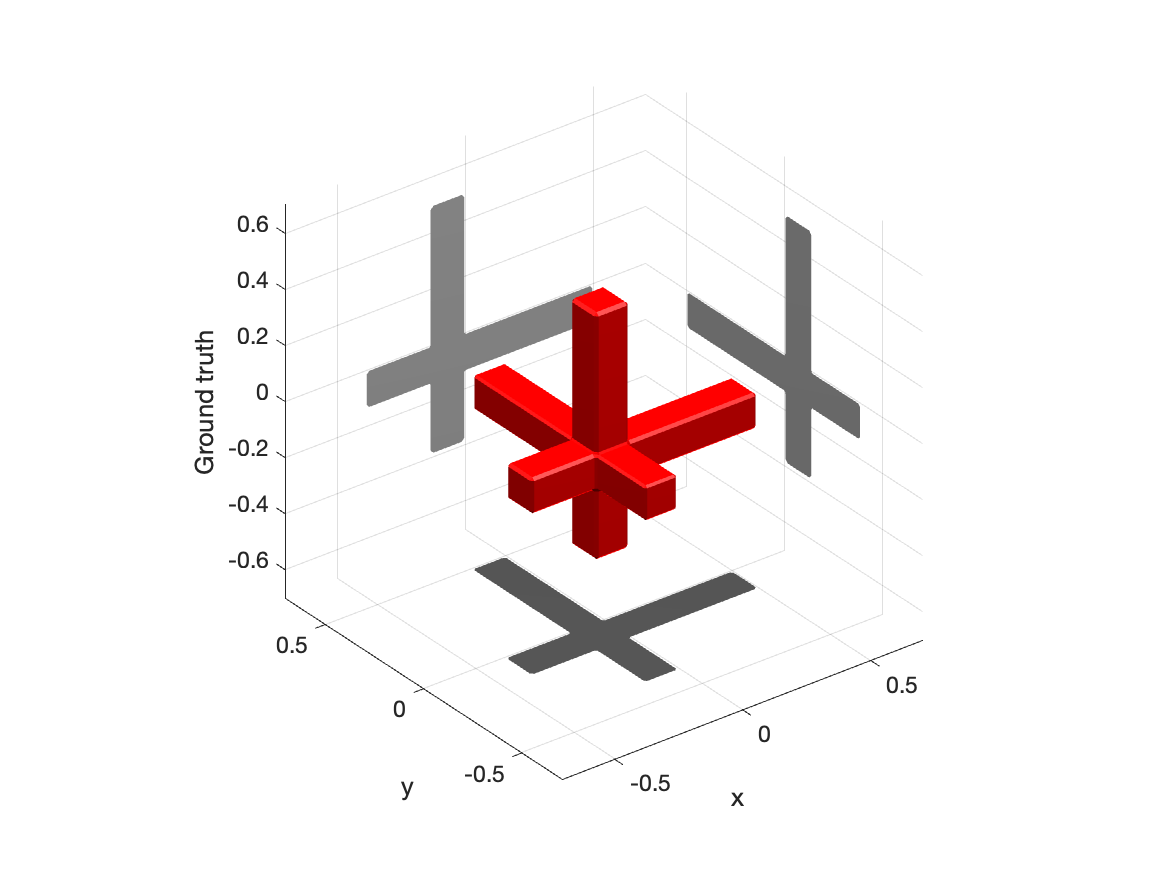} }
\subfloat[]{ \includegraphics[width=0.24\linewidth]{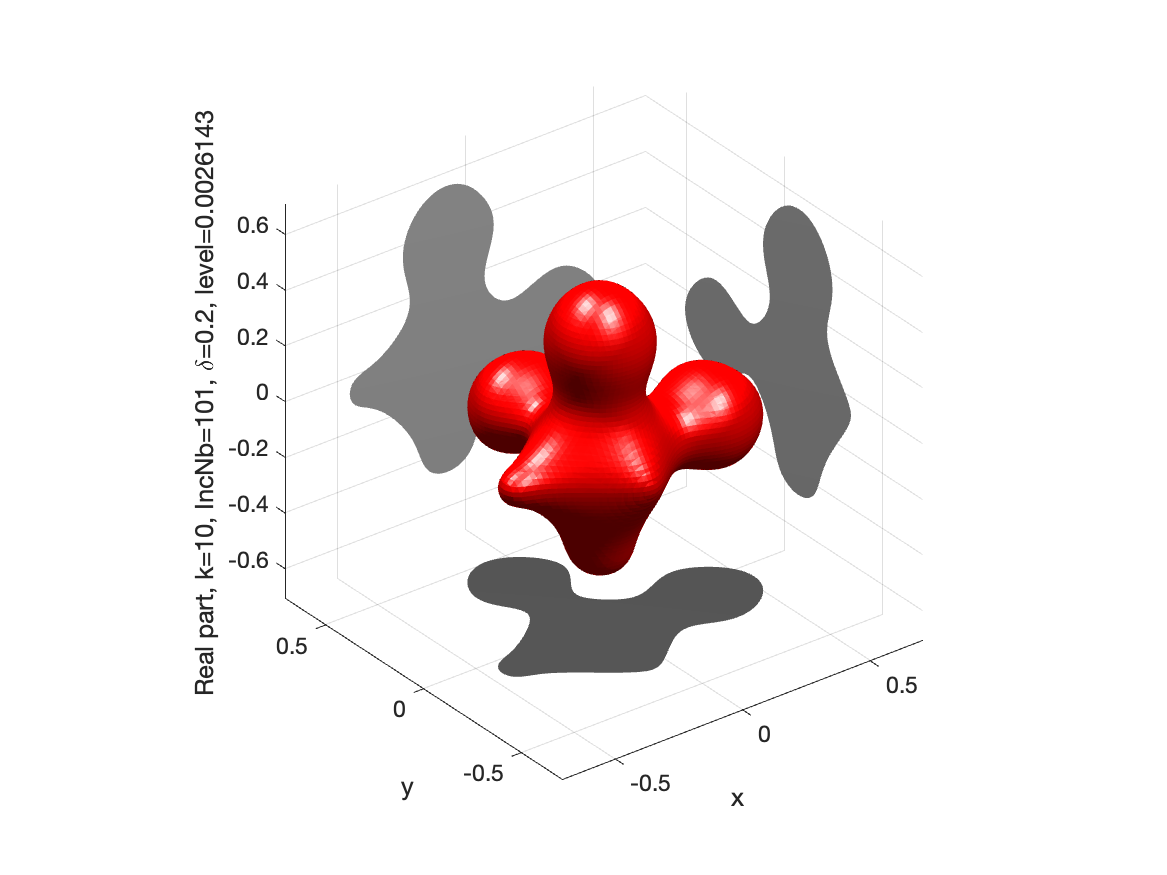} }
\subfloat[]{ \includegraphics[width=0.24\linewidth]{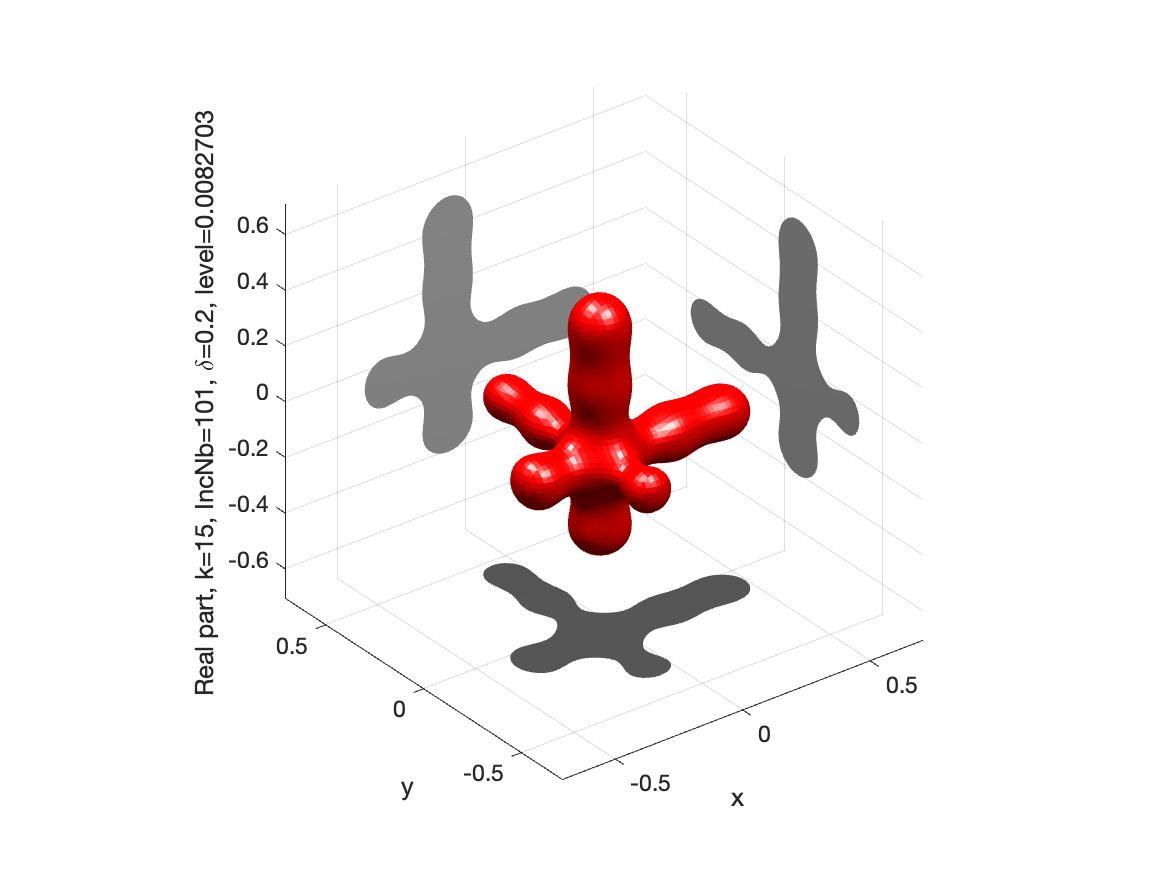} }
\subfloat[]{ \includegraphics[width=0.24\linewidth]{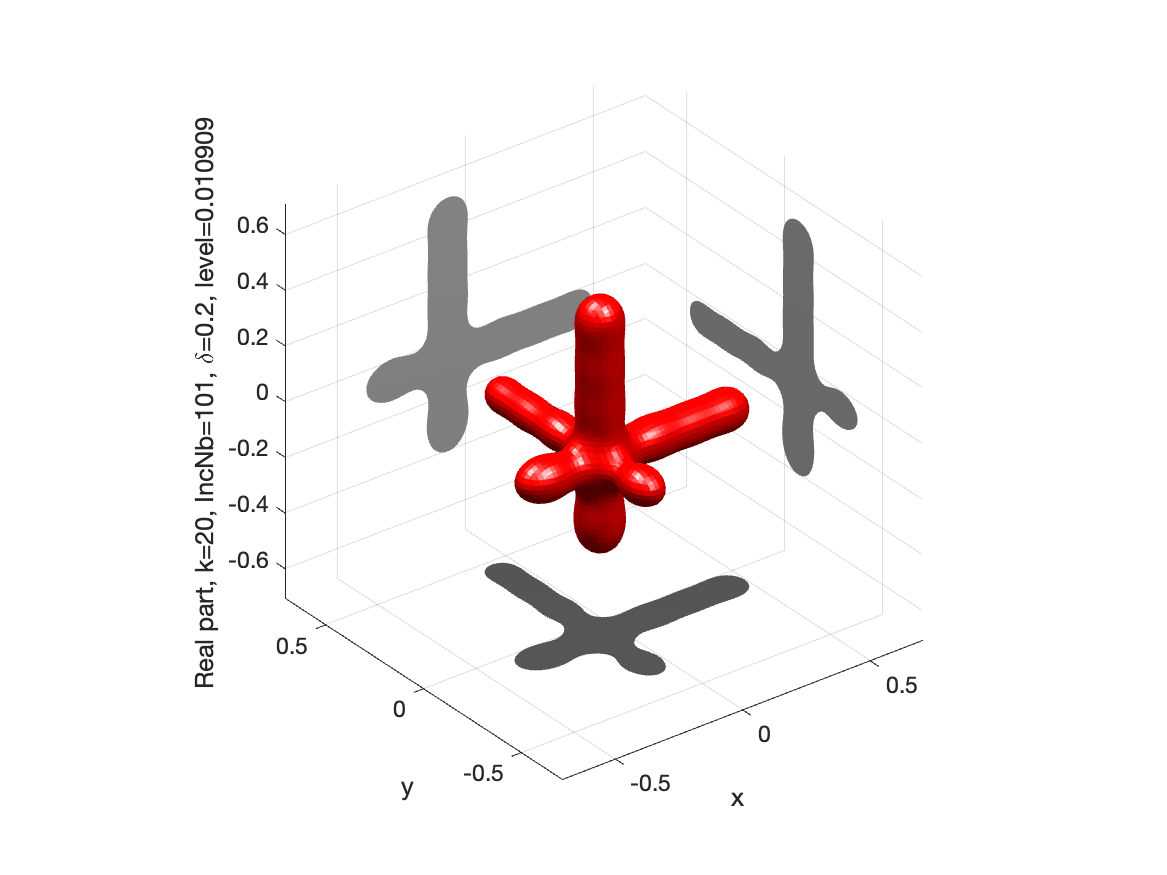} }\\

\subfloat[]{ \includegraphics[width=0.24\linewidth]{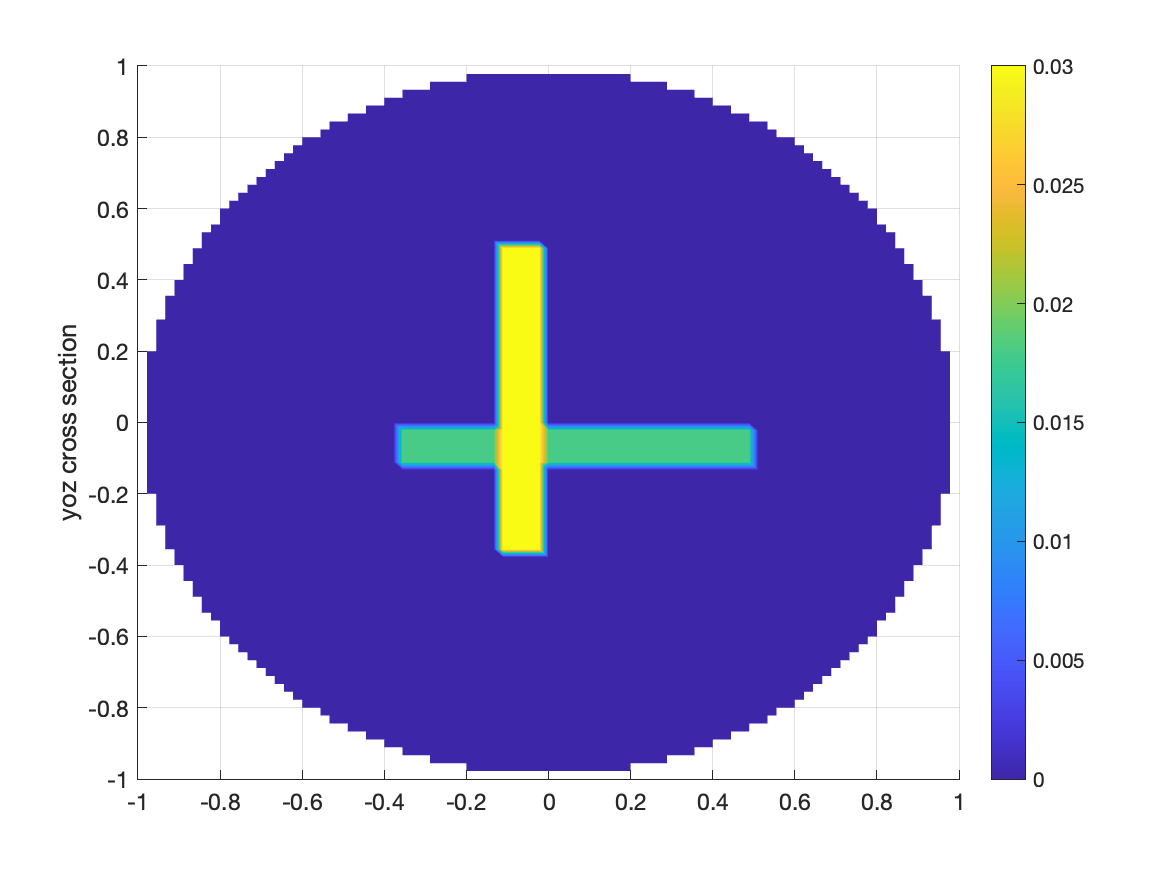} }
\subfloat[]{ \includegraphics[width=0.24\linewidth]{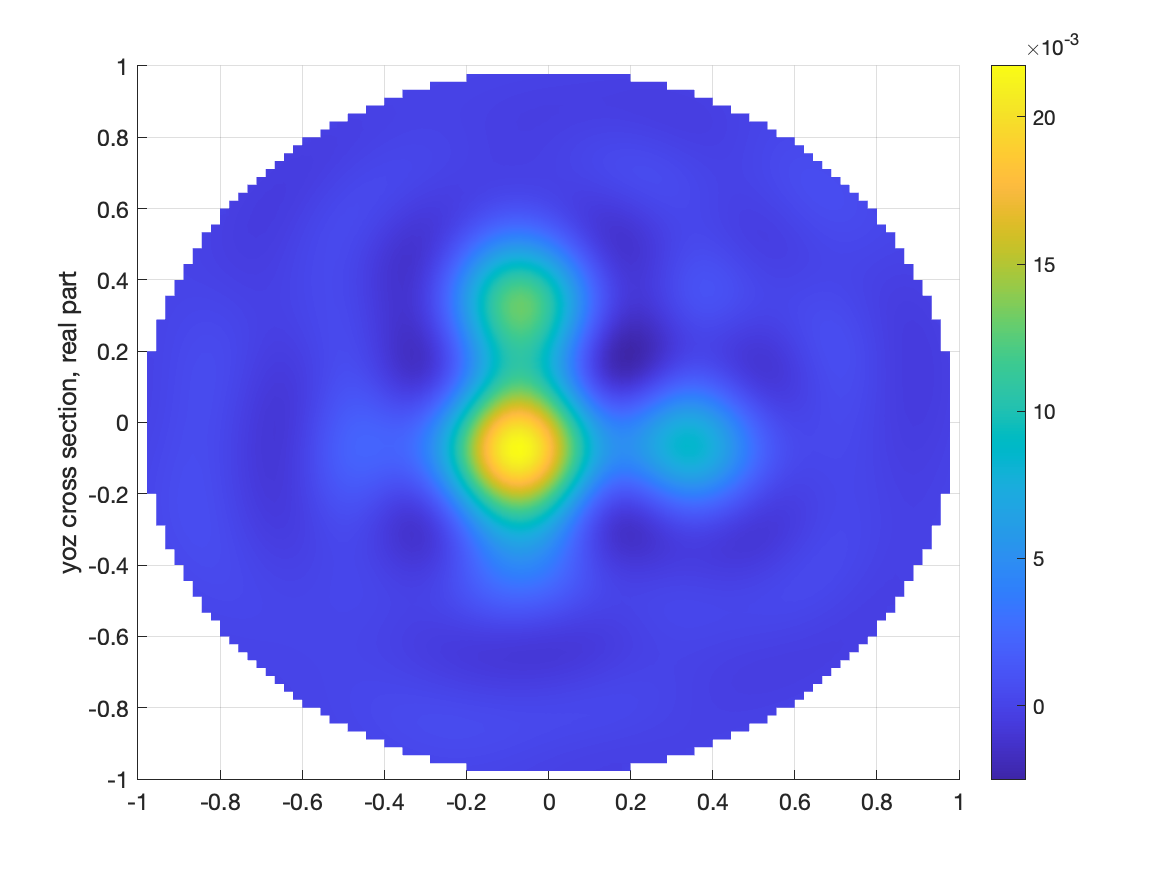} }
\subfloat[]{ \includegraphics[width=0.24\linewidth]{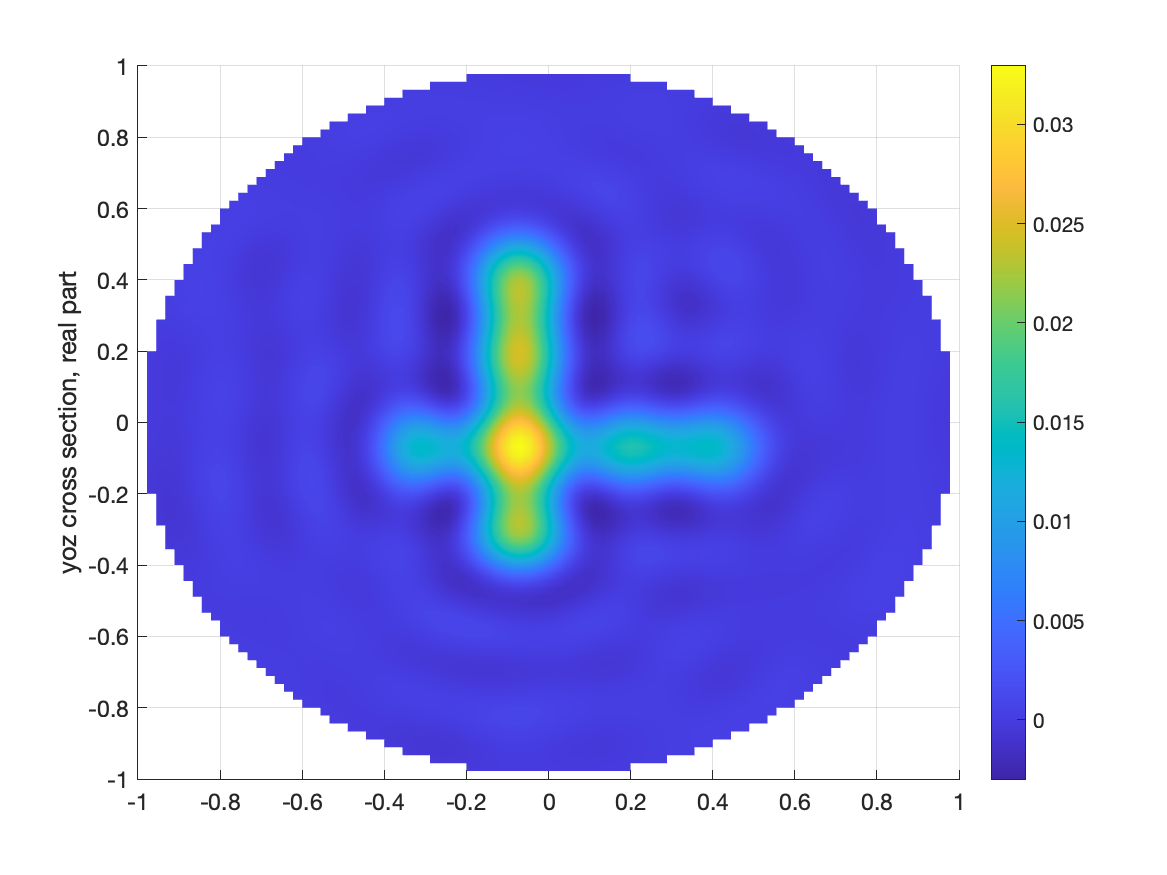} }
\subfloat[]{ \includegraphics[width=0.24\linewidth]{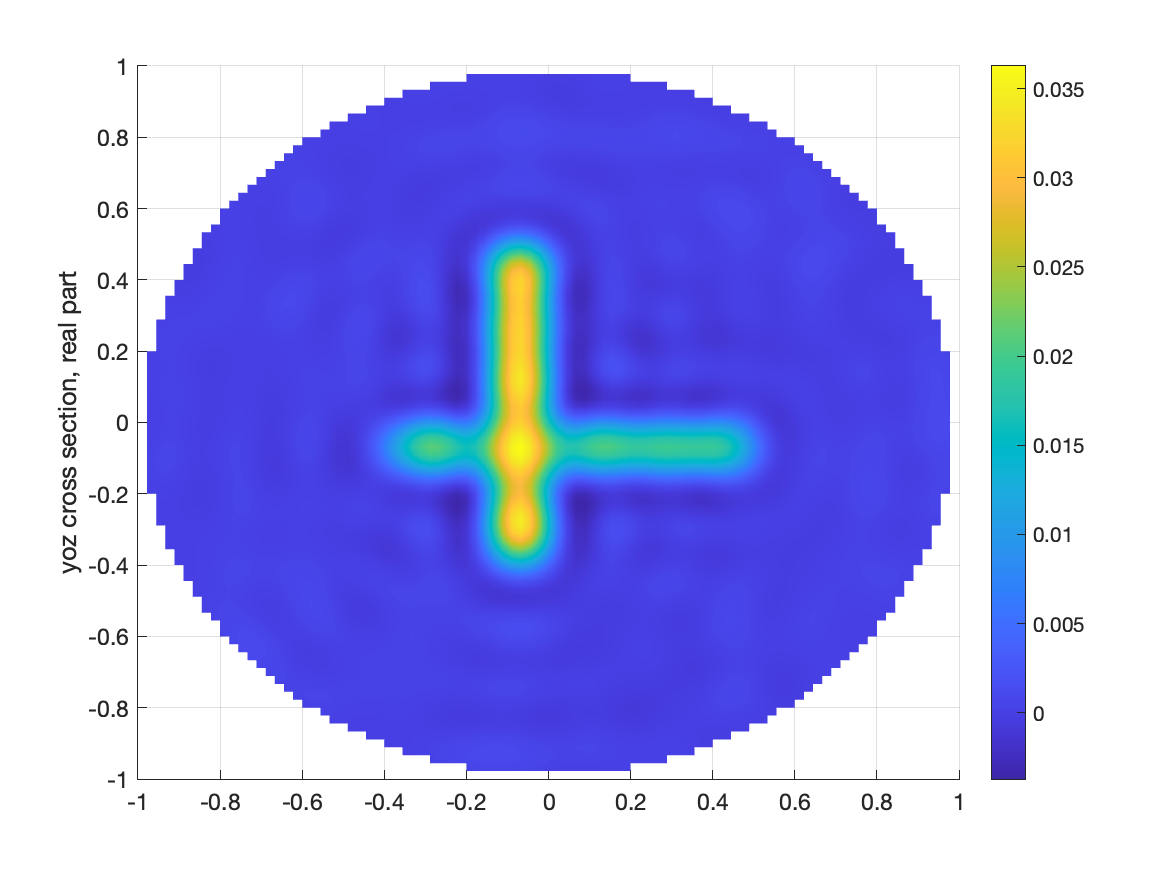} }\\

\subfloat[]{ \includegraphics[width=0.24\linewidth]{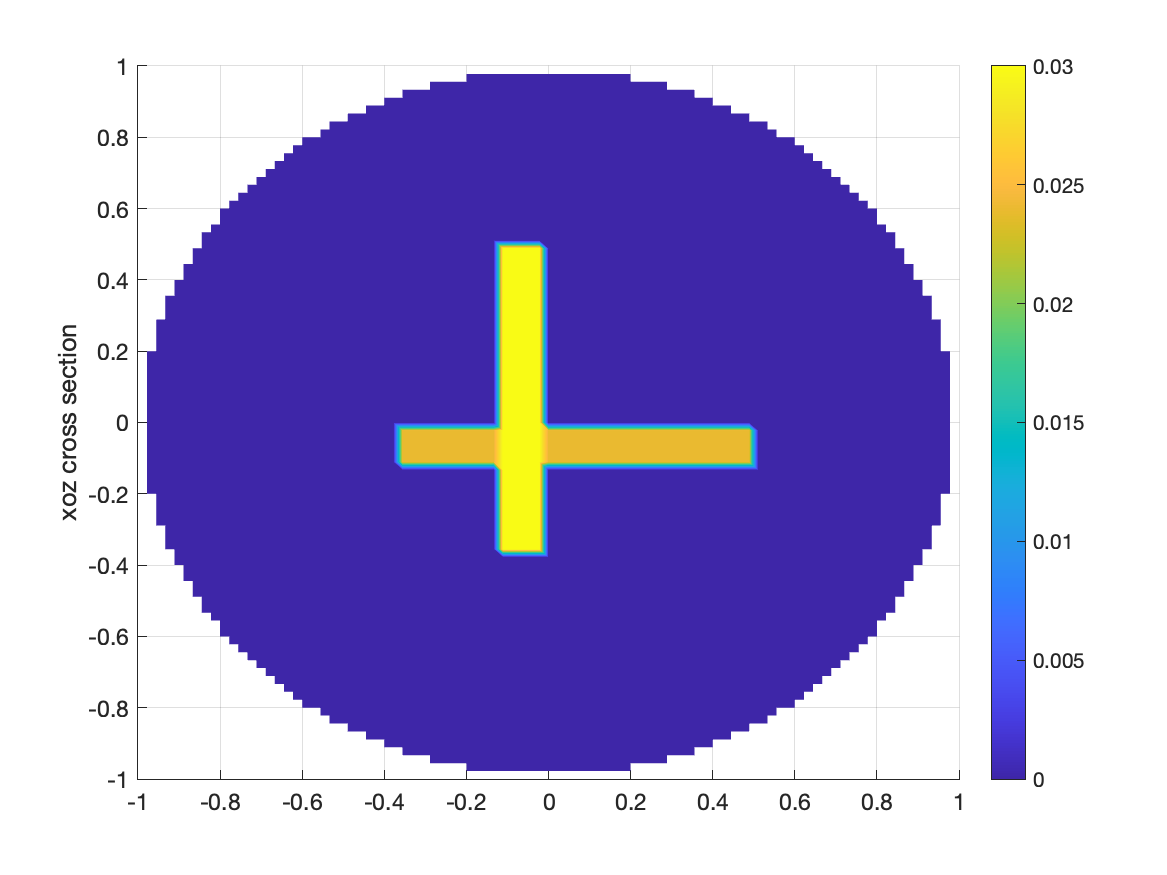} }
\subfloat[]{ \includegraphics[width=0.24\linewidth]{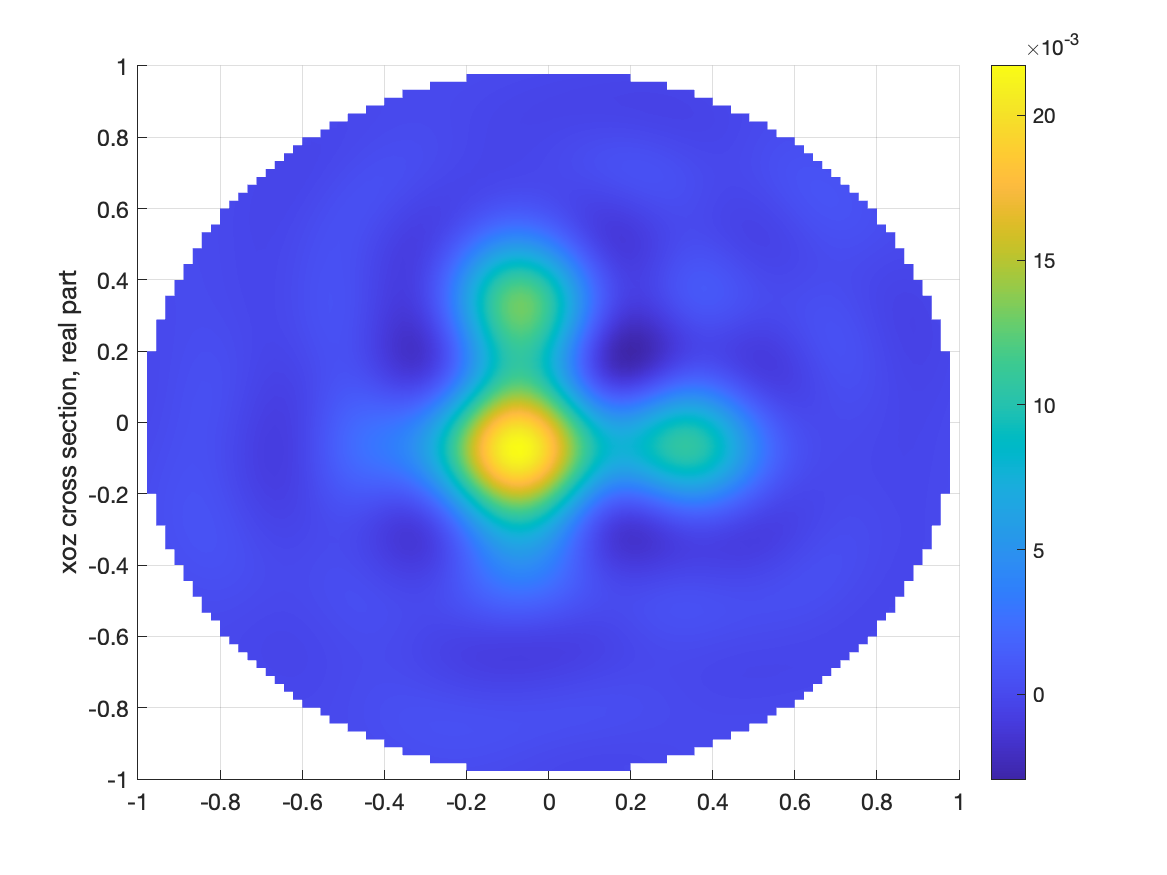} }
\subfloat[]{ \includegraphics[width=0.24\linewidth]{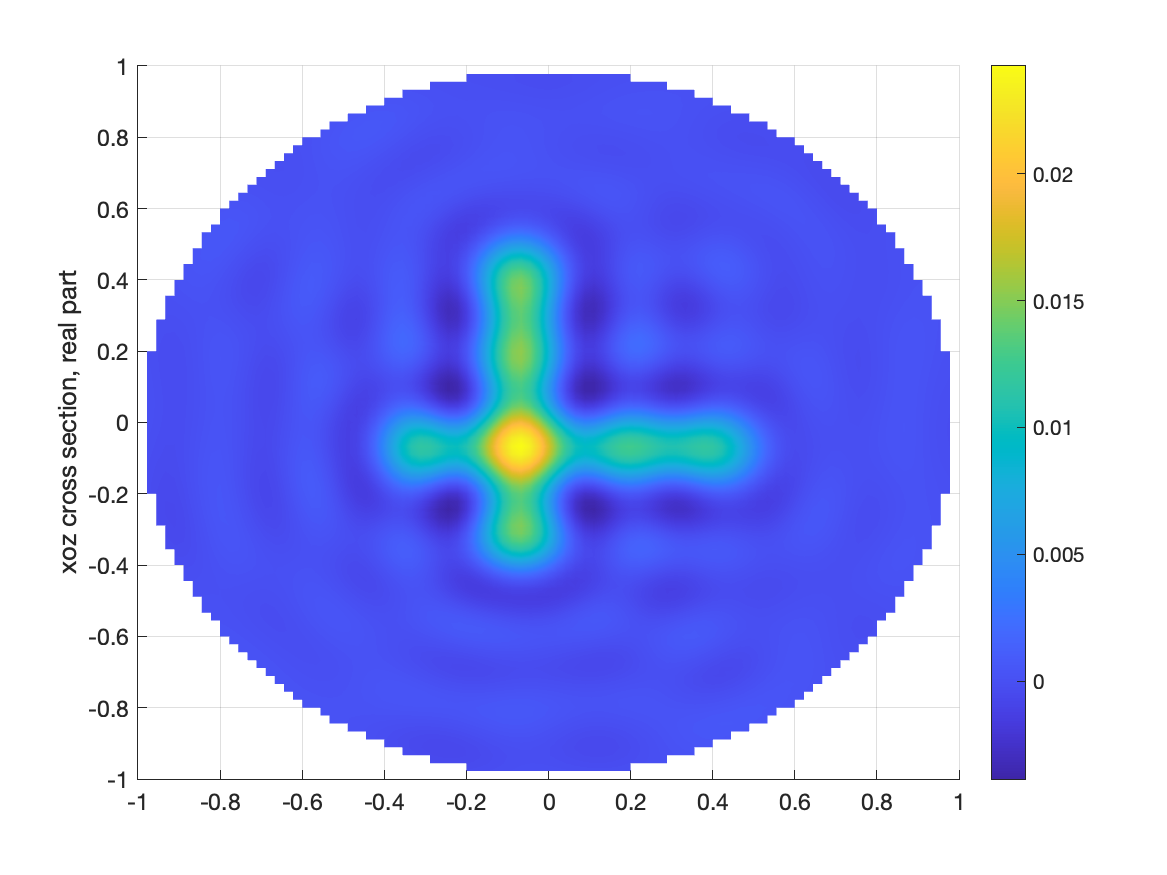} }
\subfloat[]{ \includegraphics[width=0.24\linewidth]{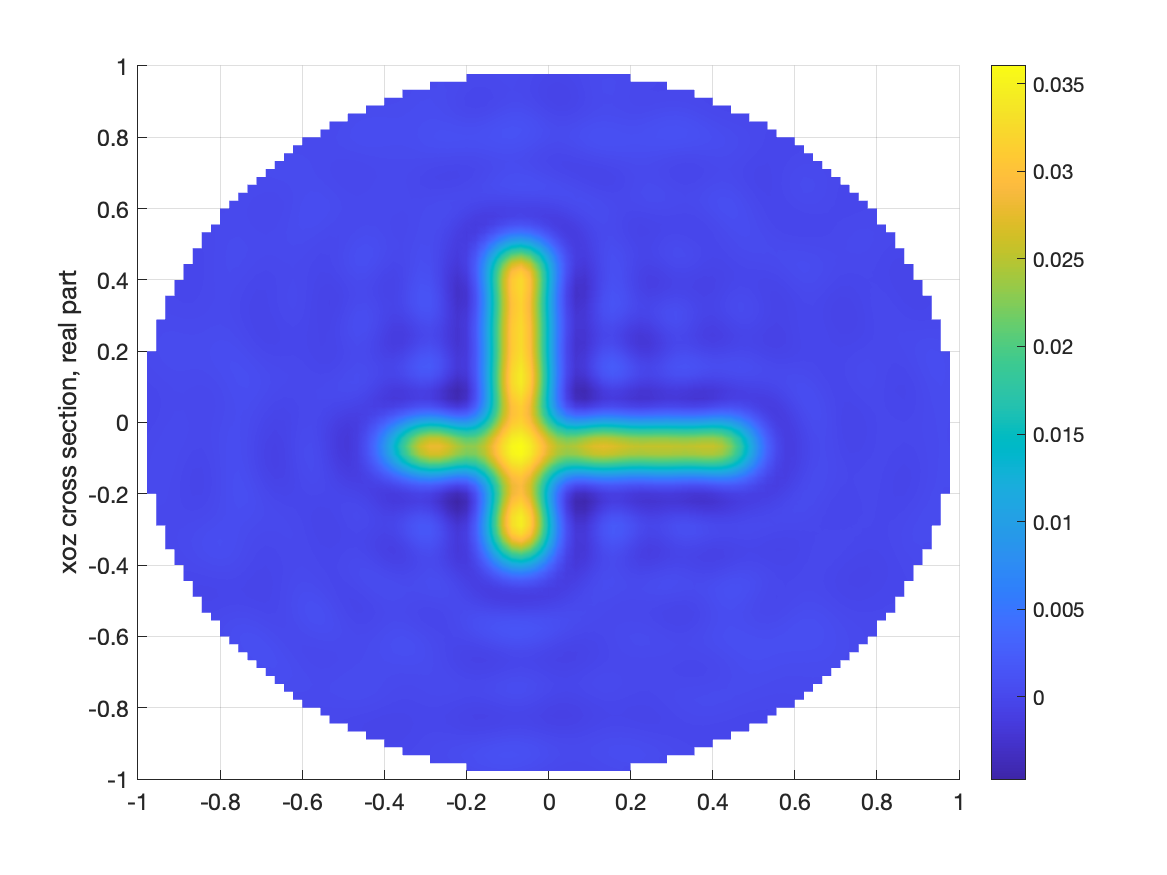} }\\

\subfloat[]{ \includegraphics[width=0.24\linewidth]{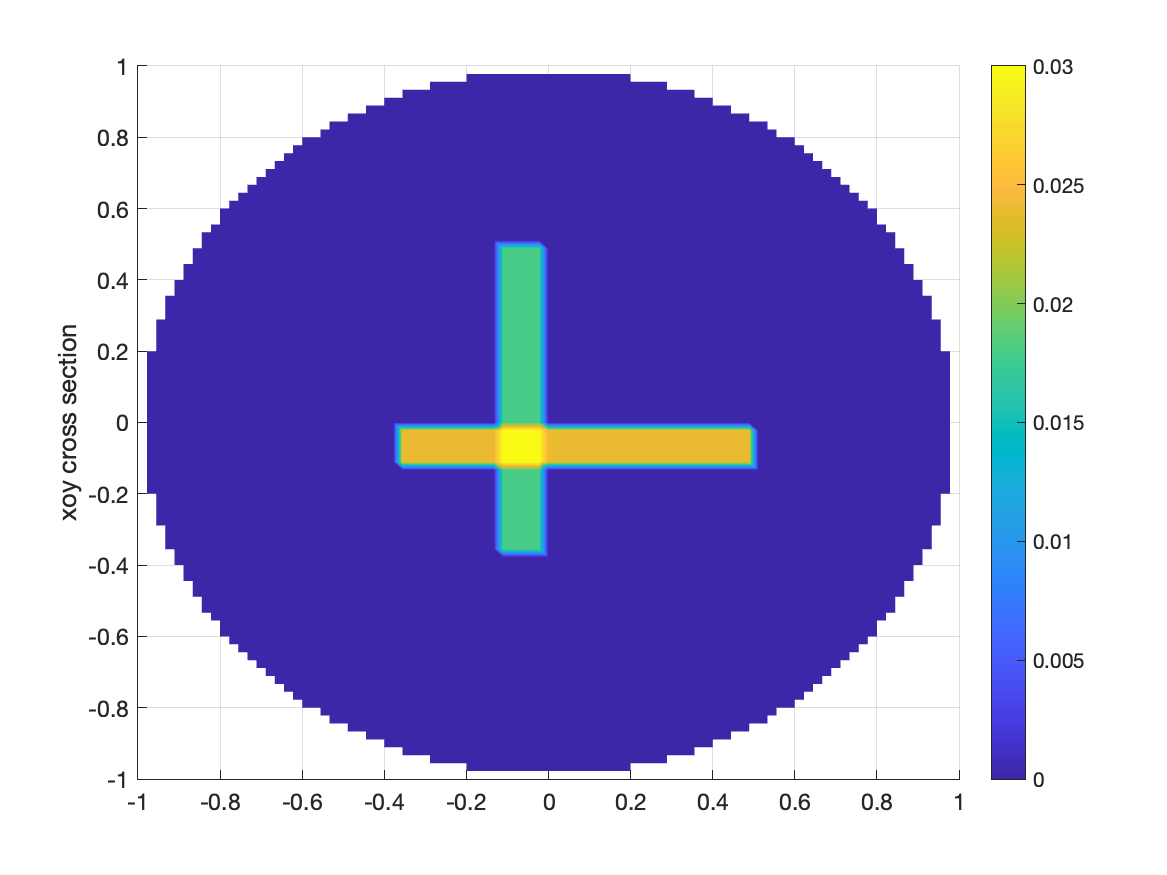} }
\subfloat[]{ \includegraphics[width=0.24\linewidth]{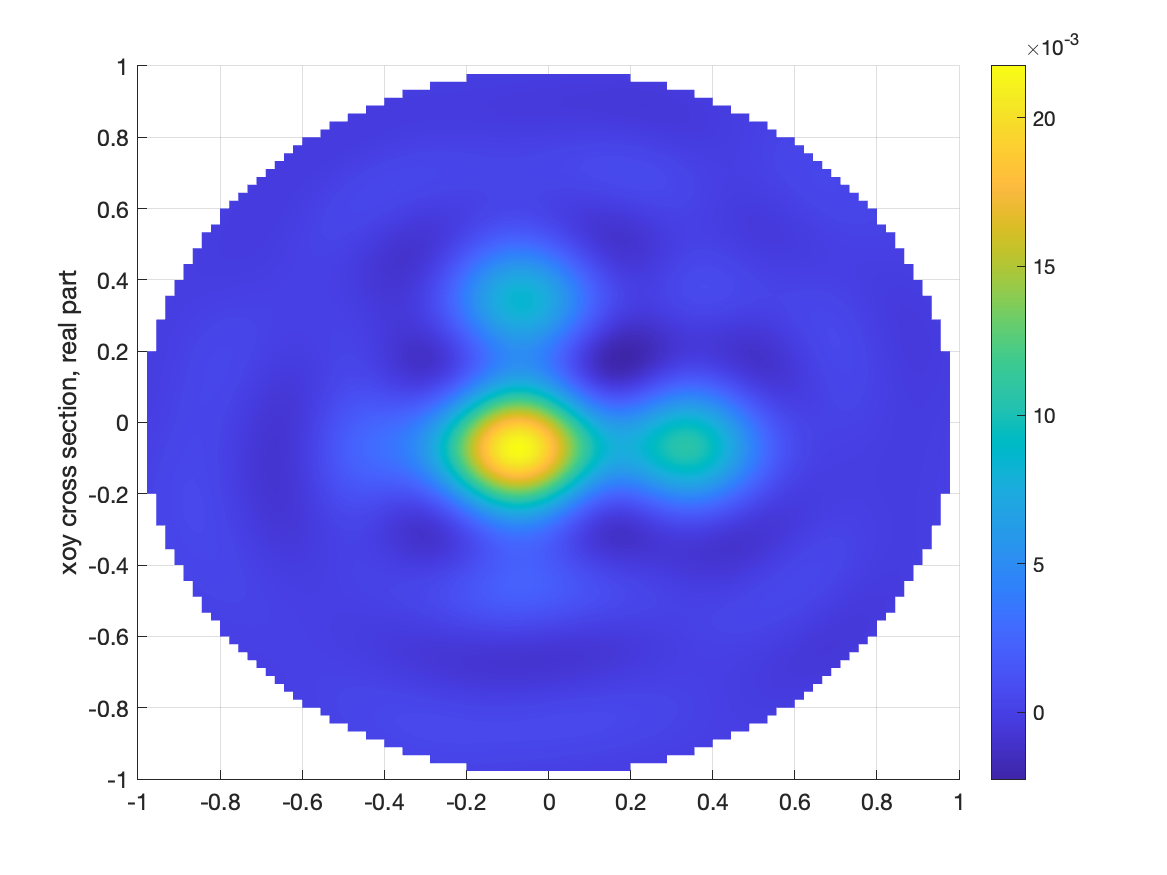} }
\subfloat[]{ \includegraphics[width=0.24\linewidth]{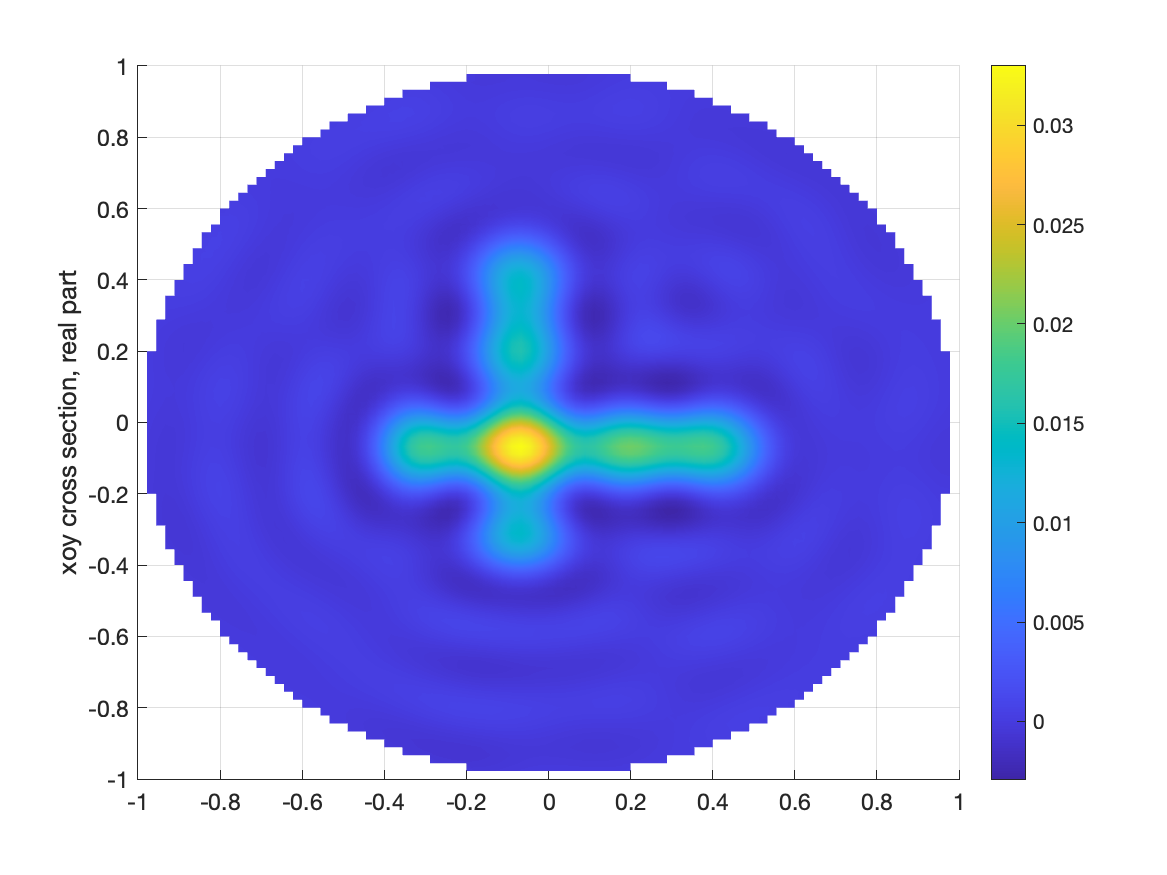} }
\subfloat[]{ \includegraphics[width=0.24\linewidth]{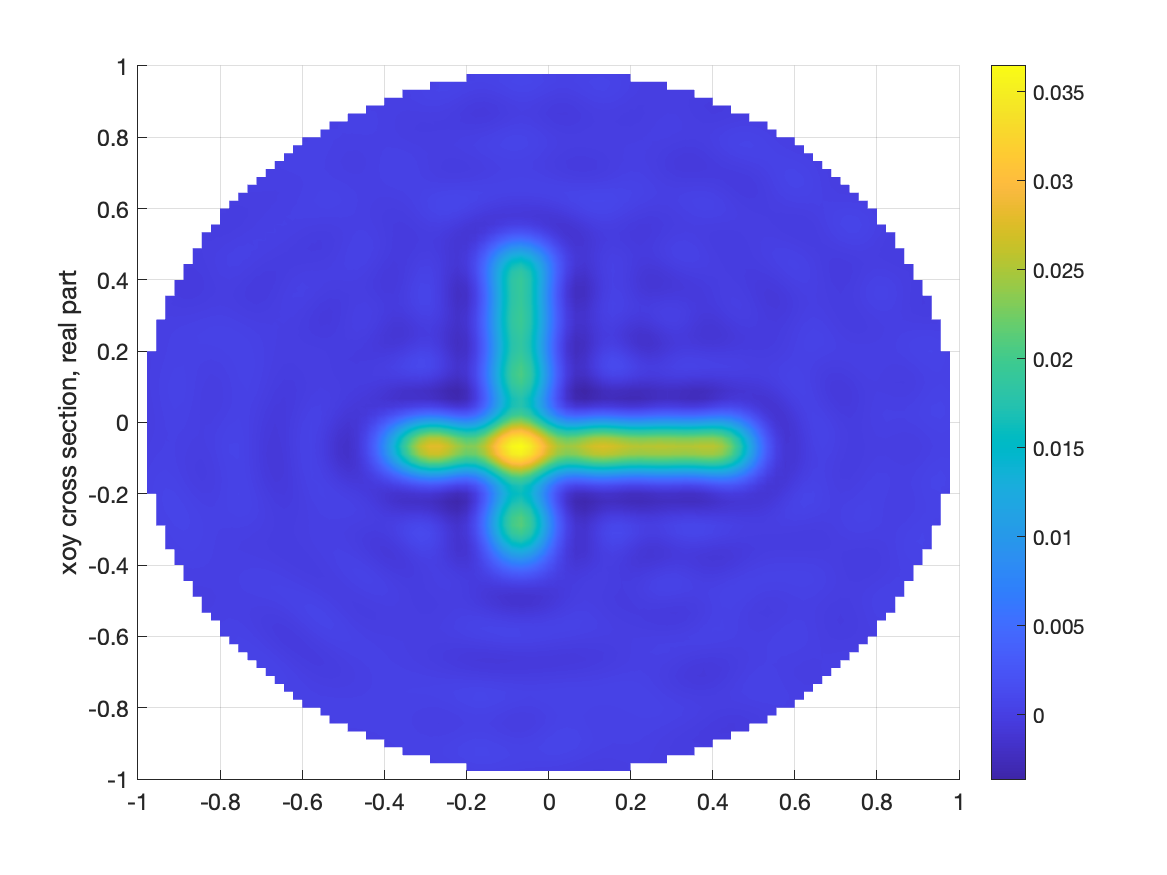} }

\caption{Reconstruction of ``cross3D" using  full far field data with observation and incident direction pairs $N_1\times N_2=101\times 101$ and noise level $\delta=0.2$. The modeling error is ${\rm rel}(k)=1.82\%,~3.28\%,~4.75\%$, respectively for $~k=10,~15,~20$. The first column: ground truth. The second, third, and fourth columns: reconstructions using $~k=10,~15,~20$, respectively. From the top to the bottom row: isosurface and the cross section views. The width of each bar is $1/8$, and the length is $7/8$. The protruding part on the shorter side is $1/4$ in length.  }\label{figure:full data cross3D}
\end{figure}

\begin{figure}[htbp]  
\centering  
\subfloat[]{ \includegraphics[width=0.25\linewidth]{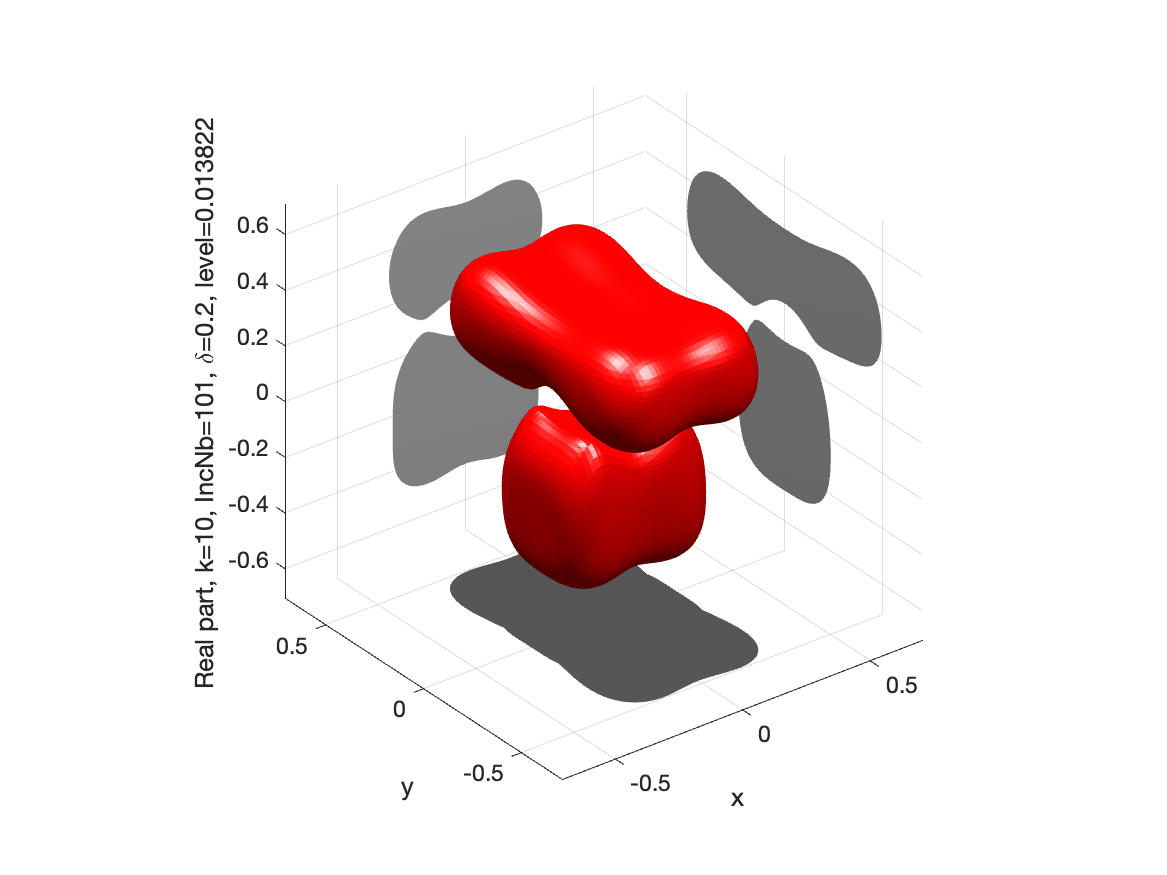} }
\subfloat[]{ \includegraphics[width=0.25\linewidth]{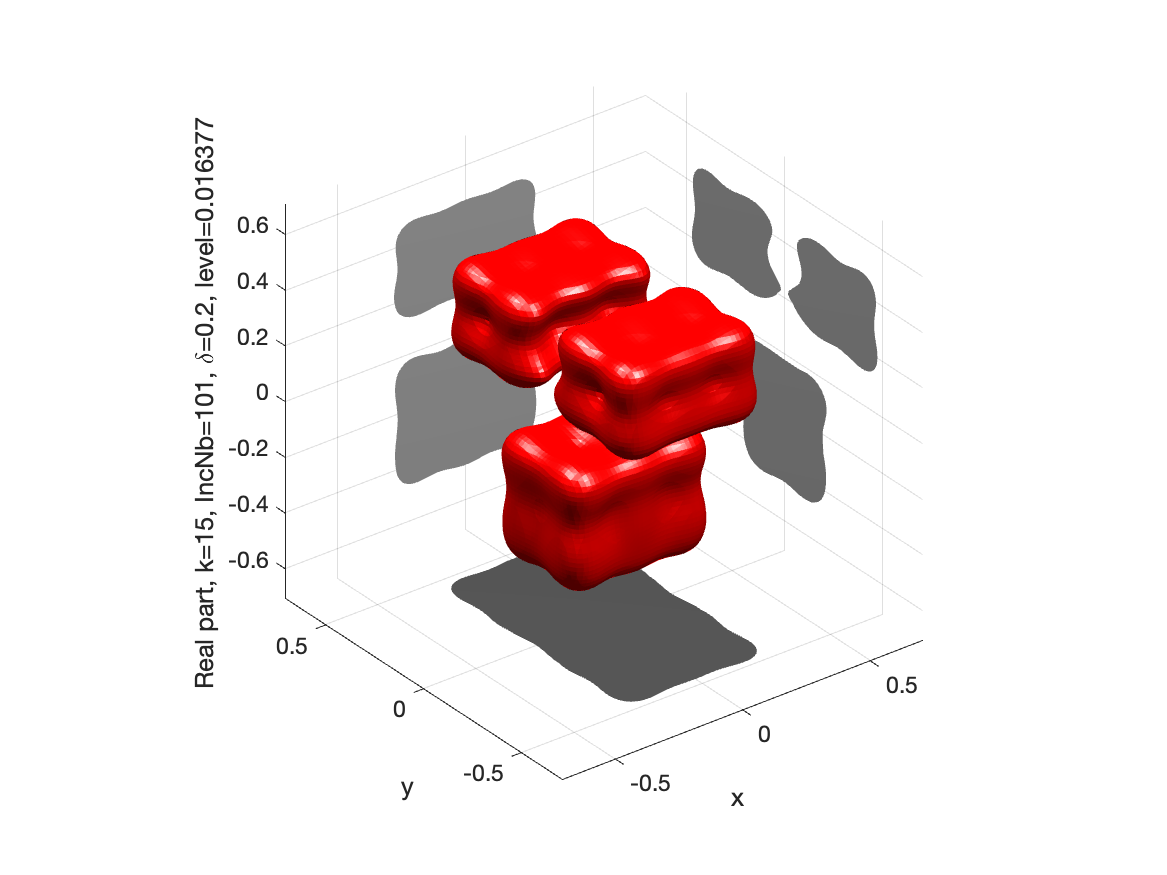} }
\subfloat[]{ \includegraphics[width=0.25\linewidth]{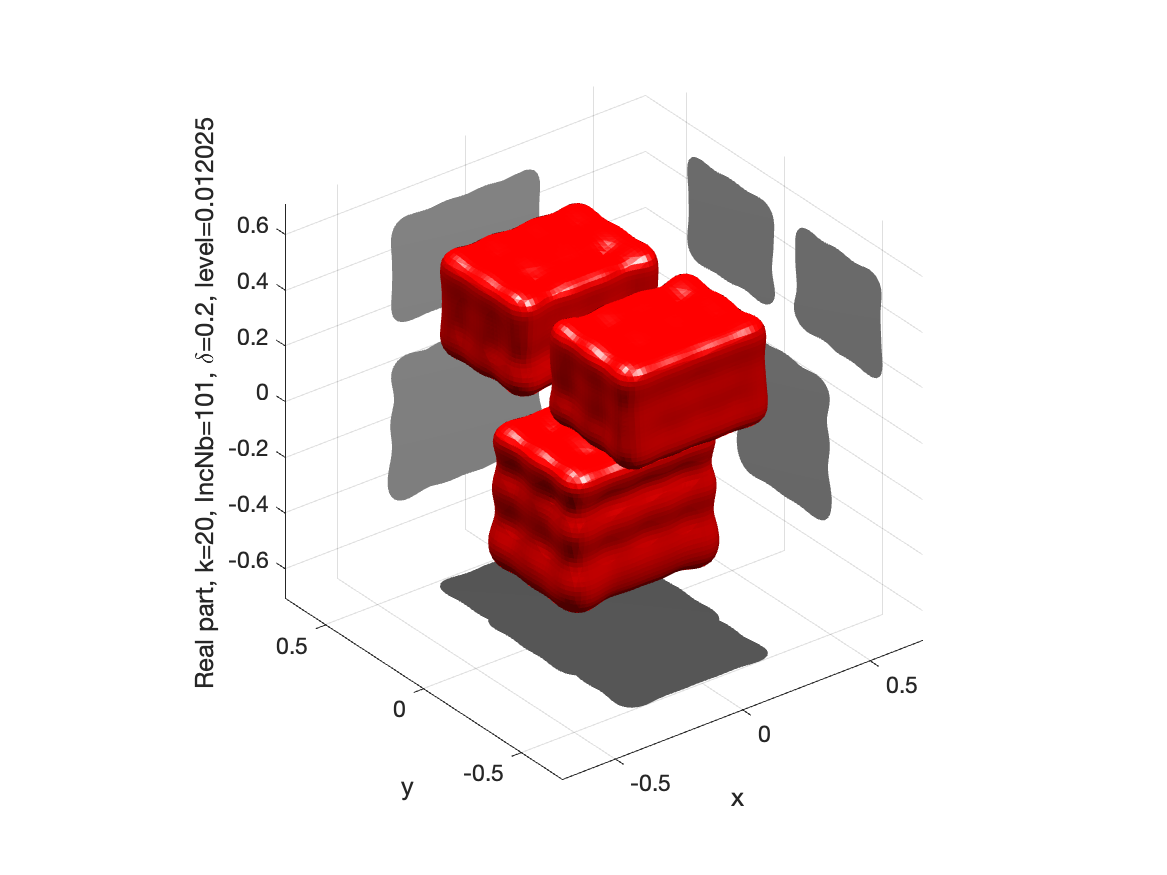} }
\subfloat[]{ \includegraphics[width=0.25\linewidth]{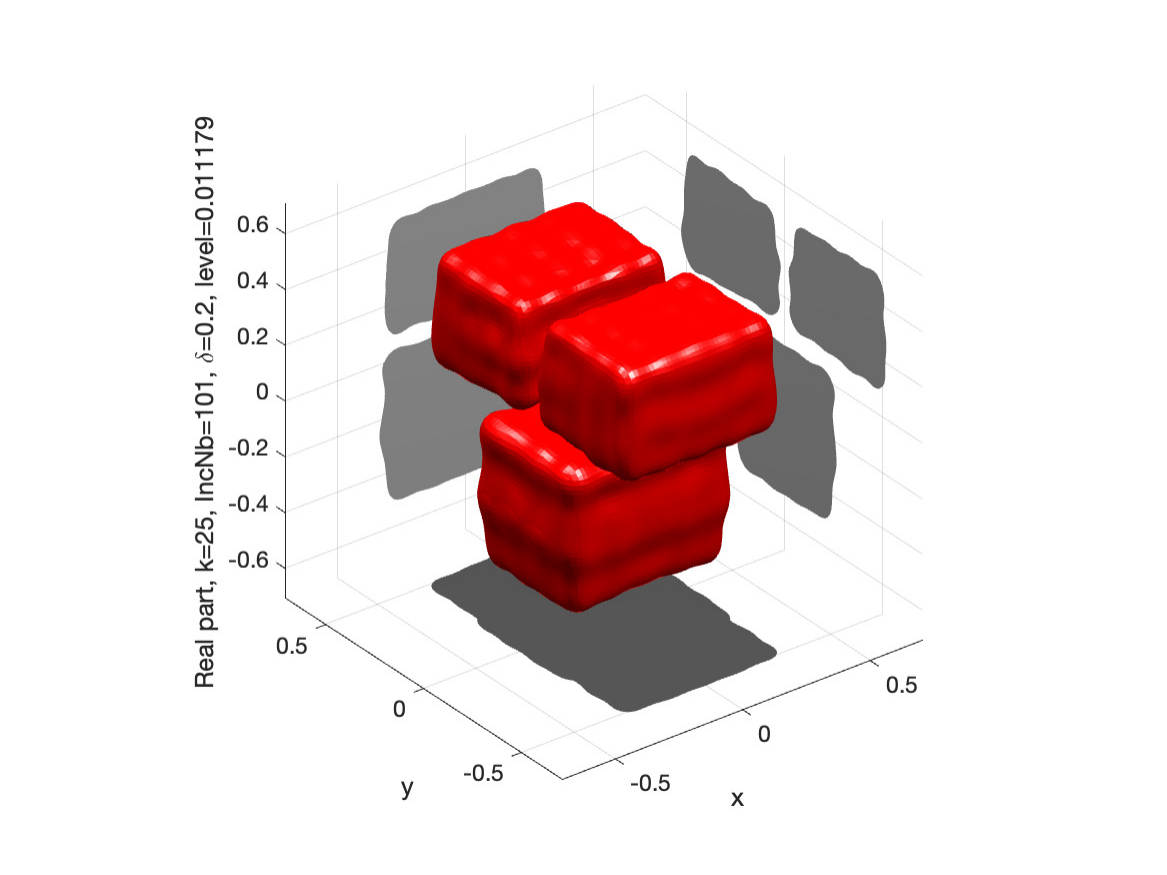} }
\caption{Reconstruction of three cubes using  full far field data with observation and incident direction pairs $N_1\times N_2=101\times 101$ and noise level $\delta=0.2$. The modeling error is ${\rm rel}(k)=1.82\%,~3.28\%,~9.15\%,~11.13\%$ for $~k=10,~15,~20,~25$, respectively. }
\label{figure:full data 3cubes3D}
\end{figure}

\subsubsection{Comparison with an iterative method}
Iterative methods typically construct a minimization functional with regularization terms and use optimization to find the minimum. The regularization terms need certain prior information about the contrast. The Toolbox IPscatt  \cite{burgel2019algorithm} offers a built-in iterative method, based on the assumption that the contrast is sparse and has sharp edges.  In iterative methods, technical difficulties may arise due to choice of regularization parameter, computational cost, as well as local minima.
The proposed method based on the low-rank structure is a direct reconstruction method, which does not require any prior information about the contrast and the regularization parameter is simple to determine. After the 3D PSWFs are precopmuted, the remaining computational cost is due to matrix vector multiplication. To fully demonstrate the potential of the proposed method, we test full far field data generated by the fully nonlinear model with relative modeling error ${\rm rel}(k)=66.82\%$ and  noise level $\delta=0.2$.
The iterative method with a fixed number of iterations ${\rm pdaN}=100$  costs more than one hour, where the proposed method only costs several seconds.  Both methods are accurate in terms of the real part, as shown in Figure \ref{figure: comparison}; however, the error of the iterative method is larger in the imaginary part, as can be seen in Figure \ref{figure: comparison error}.

\begin{figure}[htbp]  
\centering  
\subfloat[]{ \includegraphics[width=0.32\linewidth]{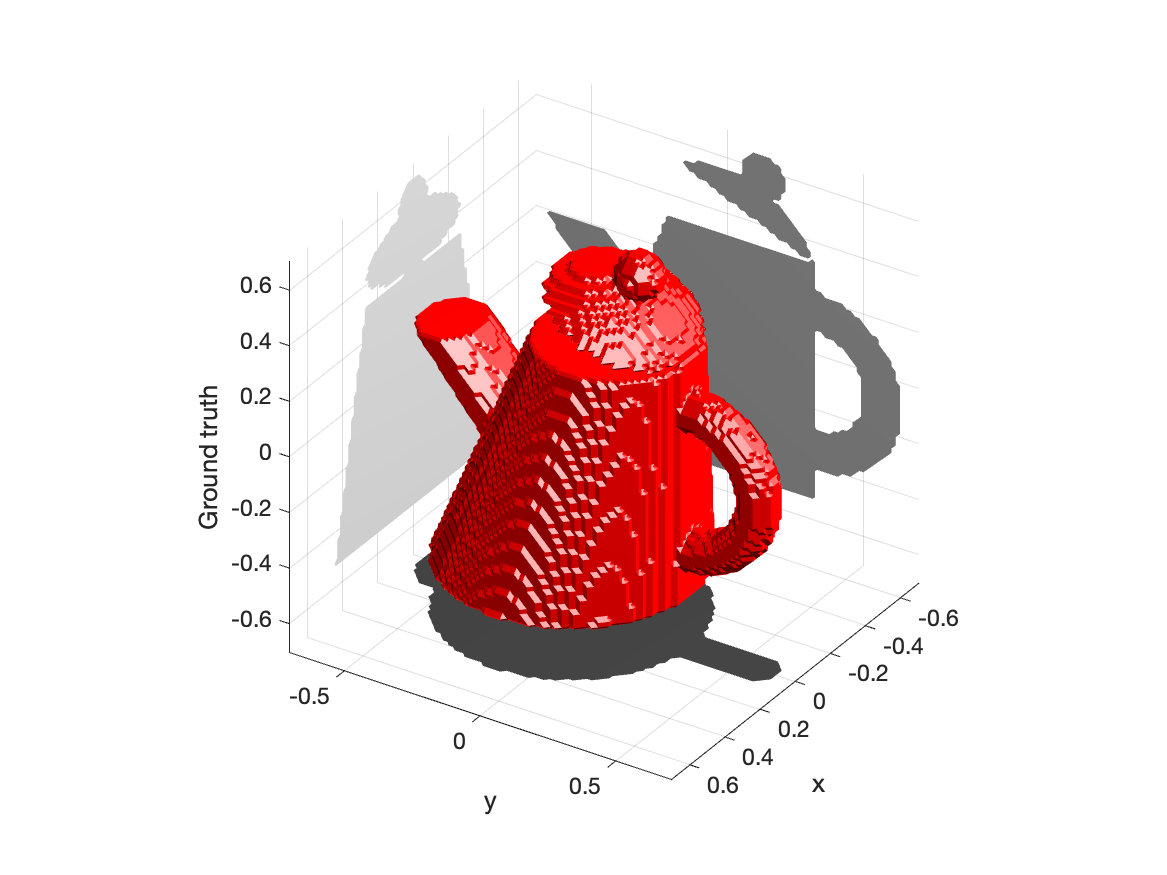} }
\subfloat[]{ \includegraphics[width=0.32\linewidth]{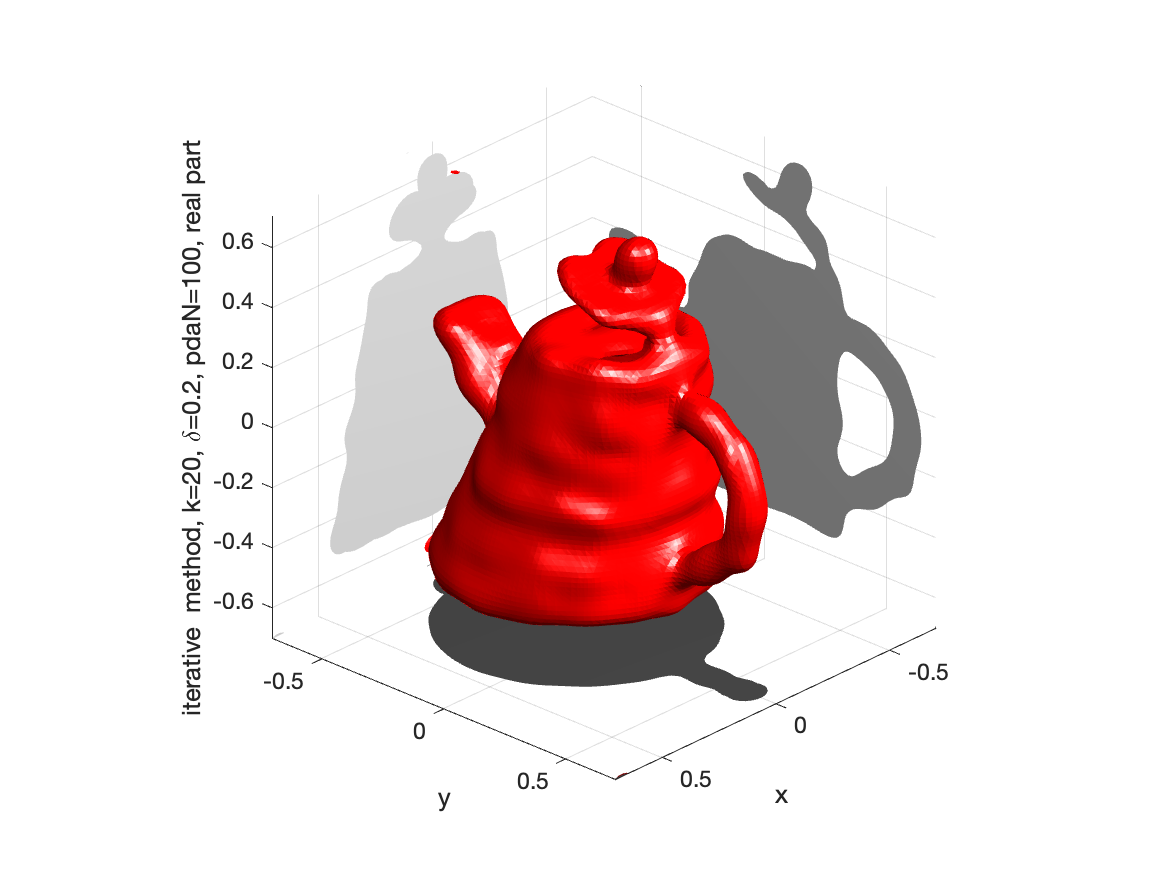} }
\subfloat[]{ \includegraphics[width=0.32\linewidth]{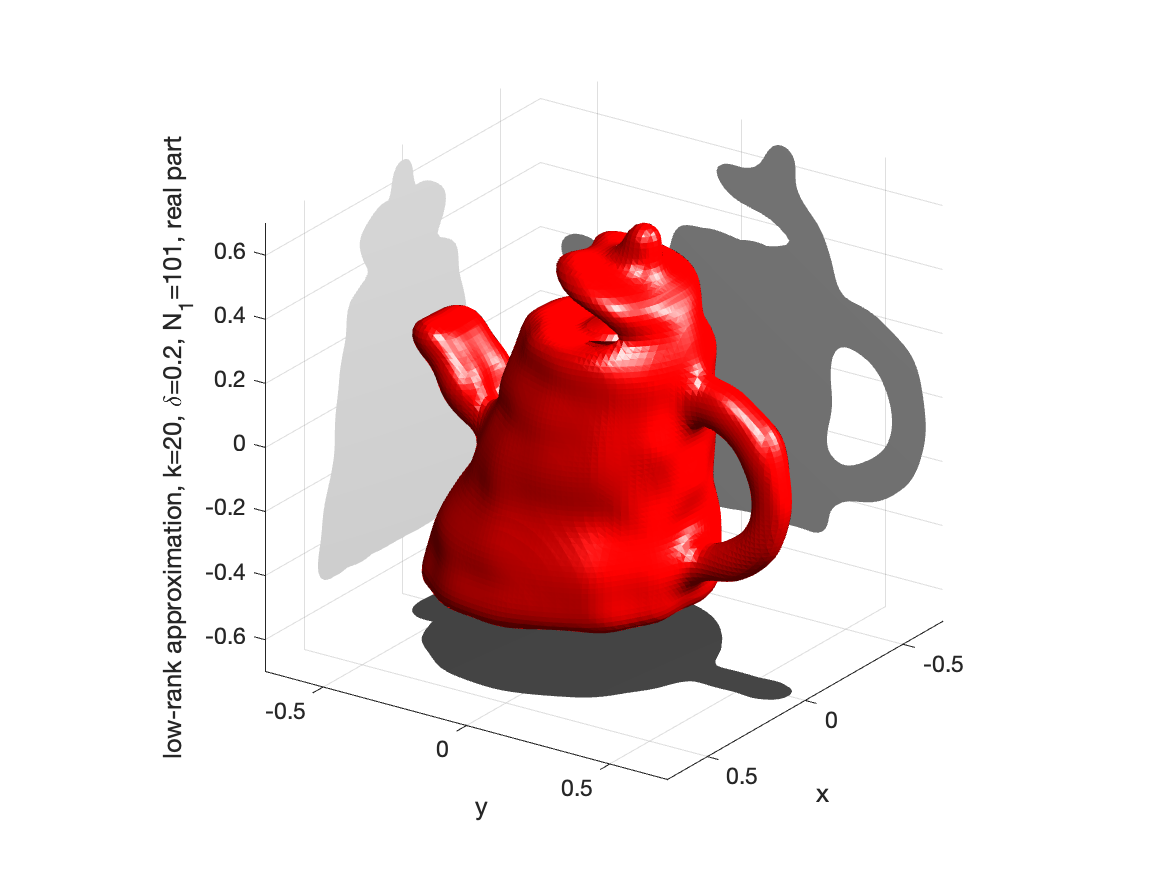} }\\

\caption{ 
Reconstruction of a teapot.
The first column: ground truth; the second column: the iterative method; the third column:  the proposed method. Wave number $k=20$, modeling error ${\rm rel}(k)=66.82\%$, noise level $\delta=0.2$, and  $N_1=N_2=101$.}\label{figure: comparison}
\end{figure}

\begin{figure}[htbp]  
\centering  
\subfloat[Iterative method]{ \includegraphics[width=0.35\linewidth]{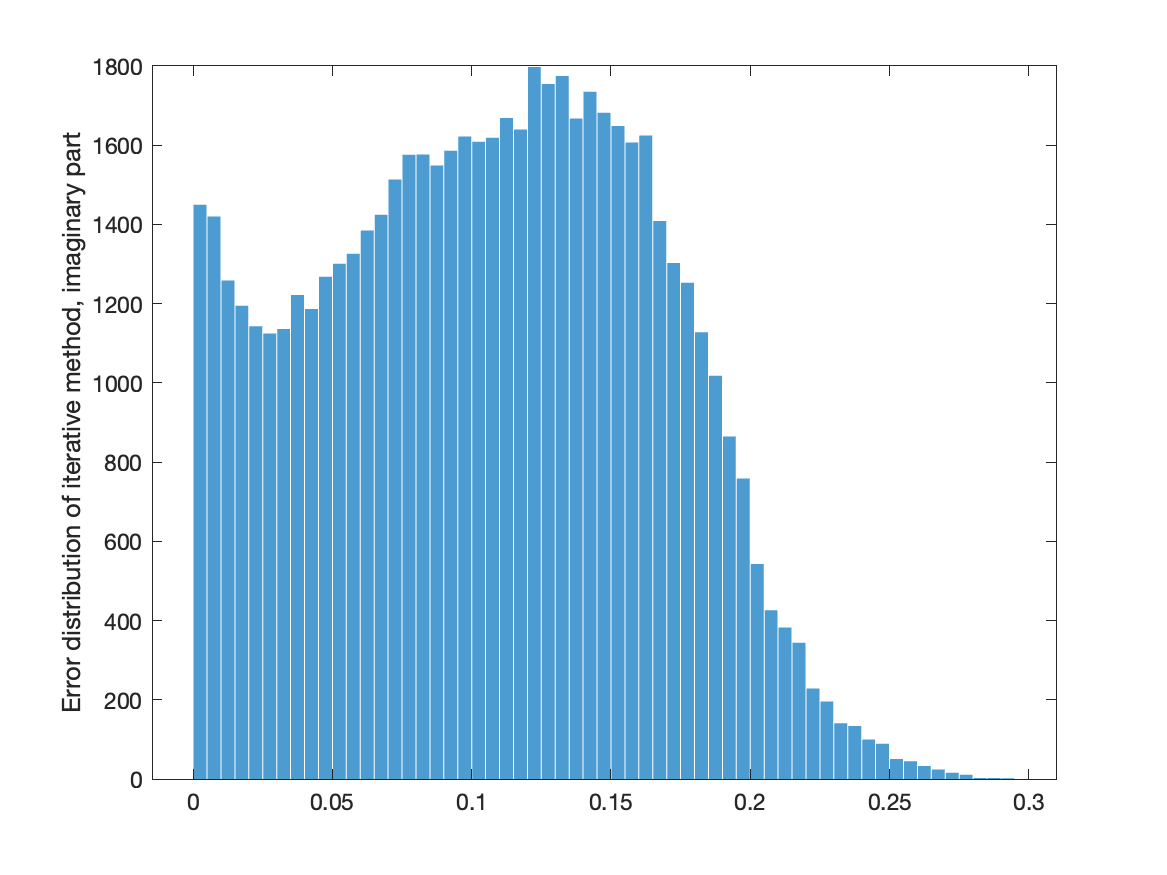} }
\subfloat[low-rank approximation]{ \includegraphics[width=0.35\linewidth]{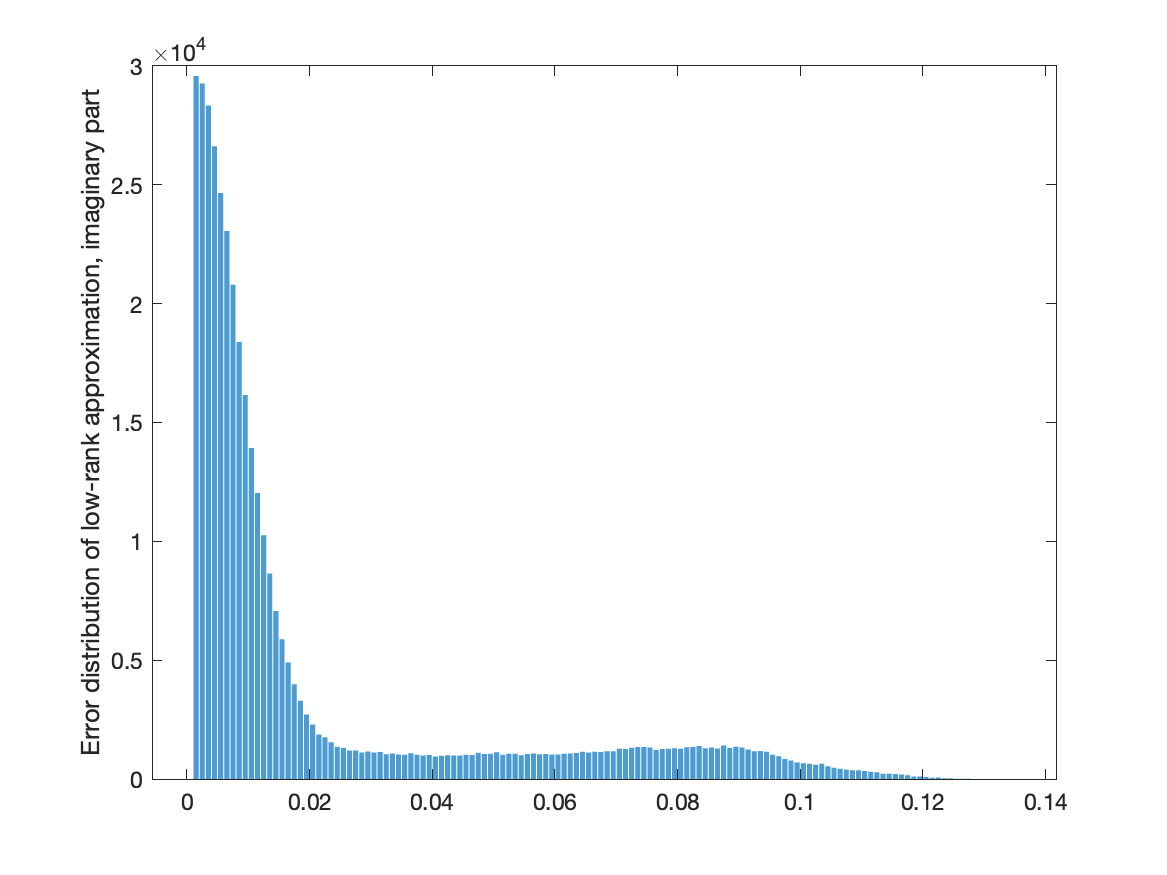} }

\caption{Reconstruction error of the imaginary part of the teapot. The imaginary part of the ground truth is zero. We plot the  errors ($x$-axis) against $91^3$ grid points ($y$-axis). Left: iterative method; right: low-rank approximation.  }\label{figure: comparison error}
\end{figure}

\subsubsection{Reconstruction using Tikhonov regularization with $H_c^s$ penalty}
The low-rank structure provides a versatile tool for regularization by penalizing the $H^s_c$-norm of solutions. The difference between the low-rank approximate solution \eqref{section stability formula pi_alpha q} and the Tikhonov regularized solution \eqref{section stability formula q_reg in H^s_c} is due to the additional term $\eta \chi_{m,n}^s$, and this will be a straightforward modification since the Sturm-Liouville eigenvalues $\chi_{m,n}$ have been conveniently computed according to Section \ref{subsection: computation of 3D PSWFs system}. We plot in Figure \ref{figure: H_c^s norm regularization} the reconstructions of two contrasts ``cross3D'' (top row) and ``teapot'' (bottom row) using the Tikhonov regularization with penalty terms in $H_c^{1/4}$ (middle column) and $H_c^{1/2}$ (right column), respectively.

\begin{figure}[htbp]  
\centering  
\subfloat[]{ \includegraphics[width=0.32\linewidth]{figures/cross3D_ground_iso_k10.eps} }
\subfloat[]{ \includegraphics[width=0.32\linewidth]{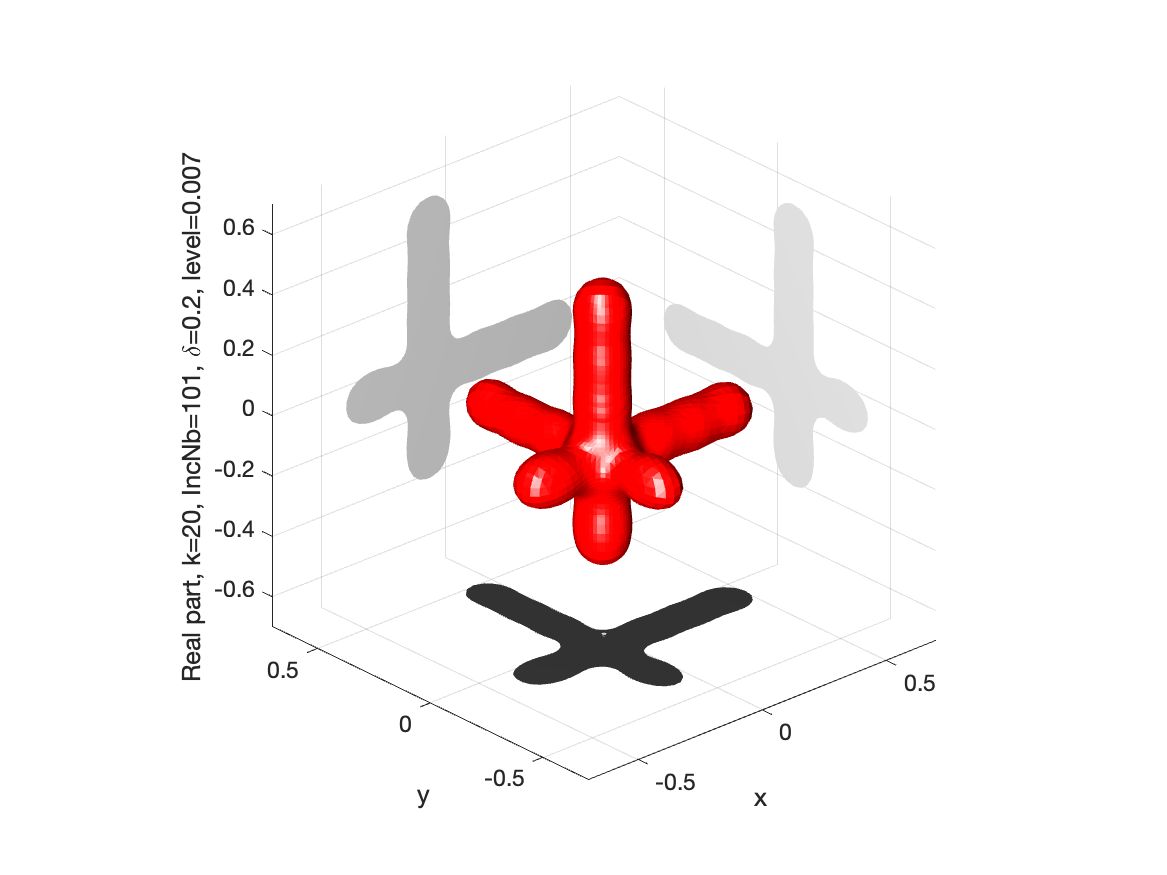} }
\subfloat[]{ \includegraphics[width=0.32\linewidth]{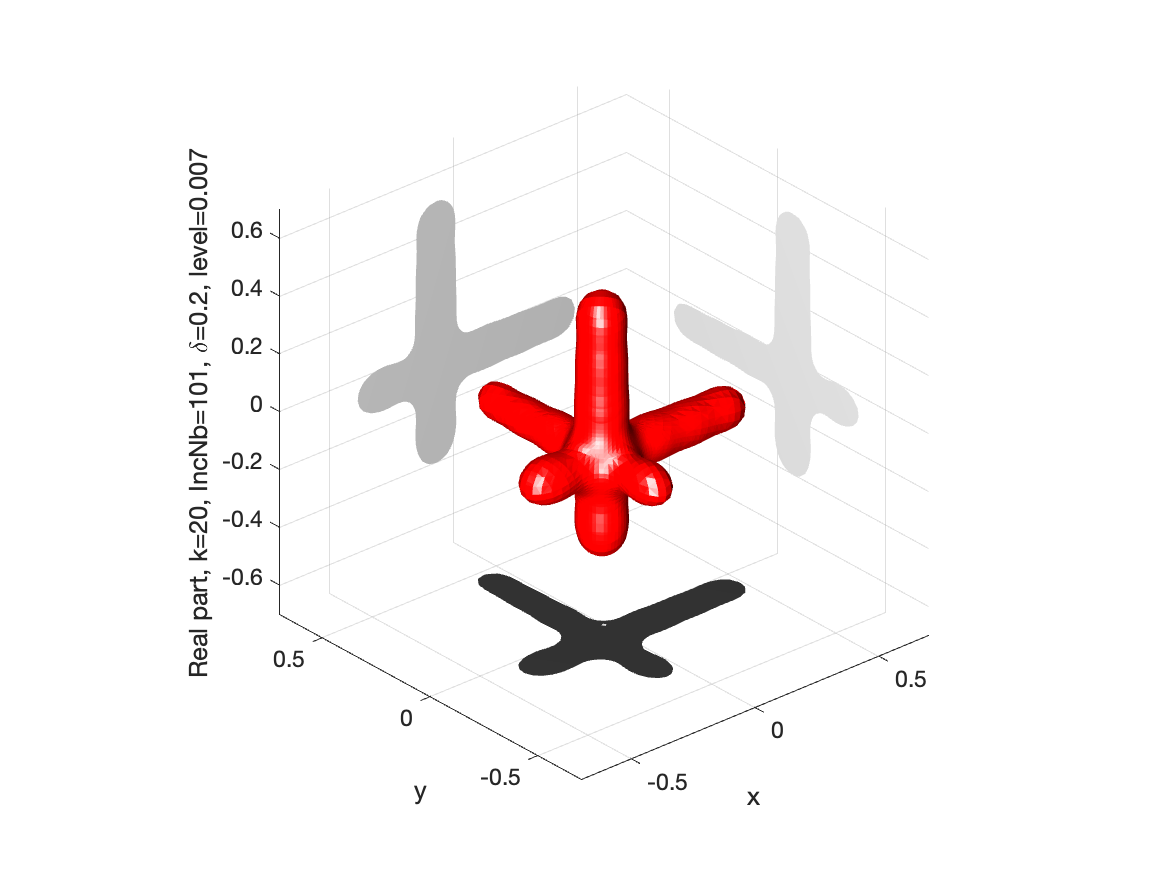} }\\
\subfloat[]{ \includegraphics[width=0.32\linewidth]{figures/teapot_GroundTruth.eps} }
\subfloat[]{ \includegraphics[width=0.32\linewidth]{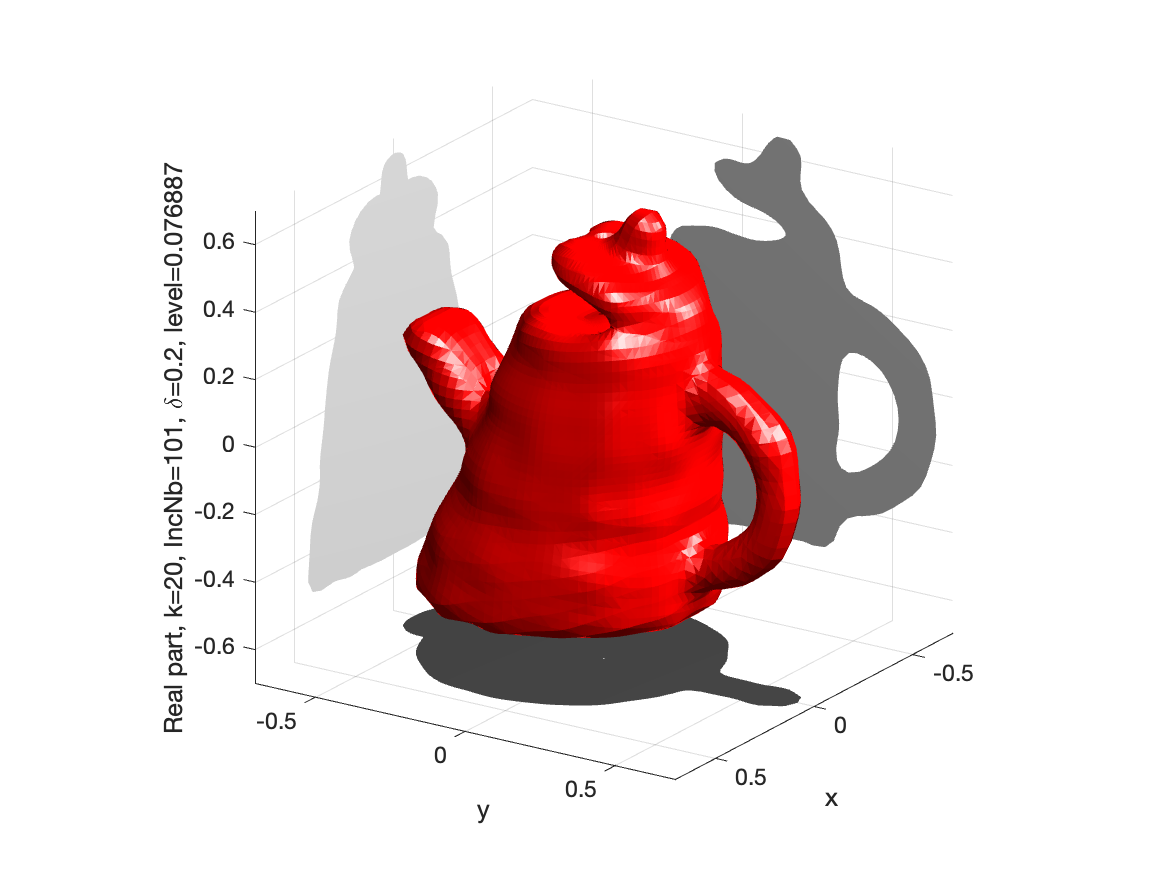} }
\subfloat[]{ \includegraphics[width=0.32\linewidth]{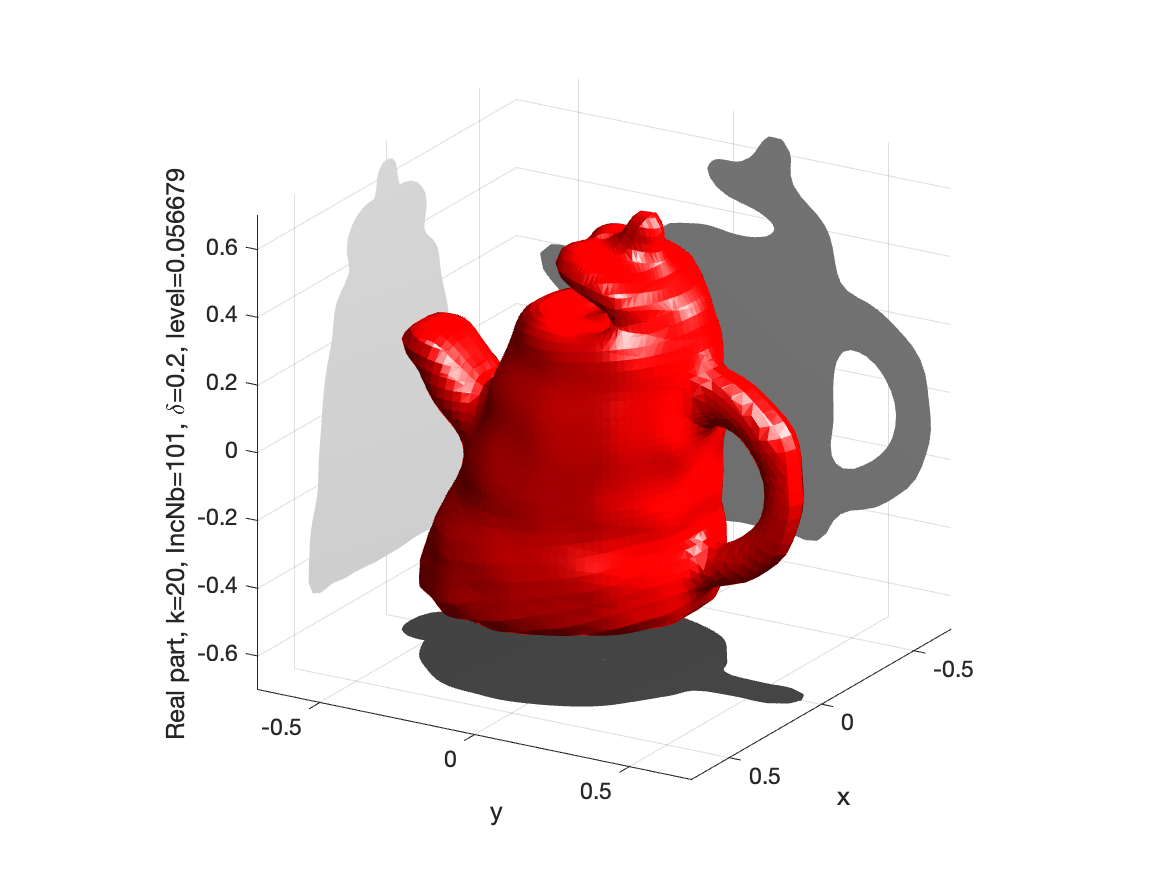} }\\

\caption{ 
Reconstructions via regularization with $H_c^s$ penalty term.
The first column: ground truth; the second column: $s=\frac{1}{4}$; the third column: $s=\frac{1}{2}$. Wave number $k=20$, noise level $\delta=0.2$,   $N_1=N_2=101$, and regularization parameter $\eta=10^{-4}$.}\label{figure: H_c^s norm regularization}
\end{figure}

\subsubsection{Localized imaging}
Finally we demonstrate the localized imaging technique in Section \ref{subsection: localized imaging}. In Figure \ref{figure: localized imaging} the unknown contrast consists of a targeting object (teapot supported in the unit ball) surrounded by other objects, and the goal is to only image the targeting teapot. The reconstruction by the proposed low-rank structure is plotted in the middle column, and the reconstruction by the tested iterative method is given in the right column. The double orthogonality offers a localized imaging technique by manipulating the data, which filters the data by its projection onto the 3D PSWFs with dominant prolate eigenvalues. The teapot is surrounded by three complex objects in the bottom case of Figure \ref{figure: localized imaging}, and its reconstruction by the low-rank approximation is still good while the computational cost remains the same despite the presence of surrounded complex objects.  This demonstrates the potential of the low-rank structure in data processing/filtering and  localized imaging techniques.

\begin{figure}[htbp]  
\centering  
\subfloat[Ground Truth]{ \includegraphics[width=0.32\linewidth]{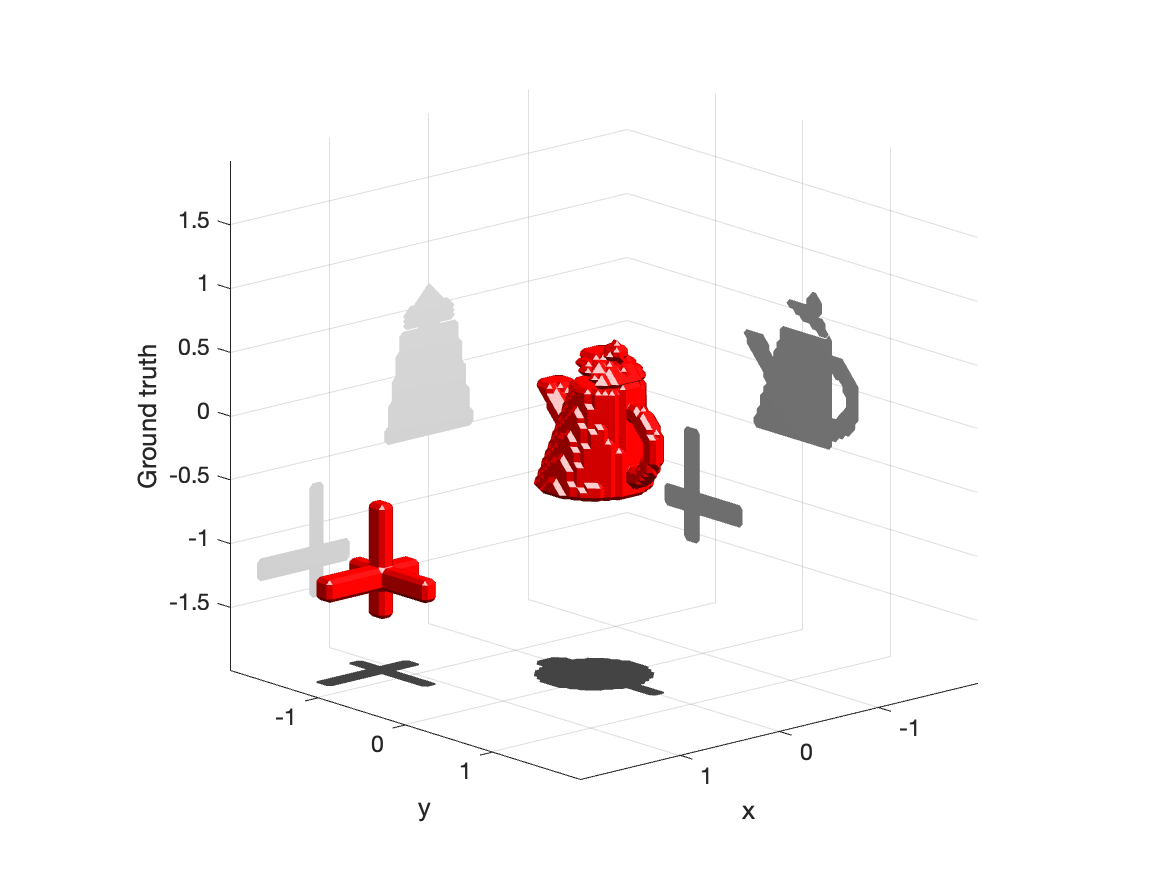} }
\subfloat[Low-rank]{ \includegraphics[width=0.32\linewidth]{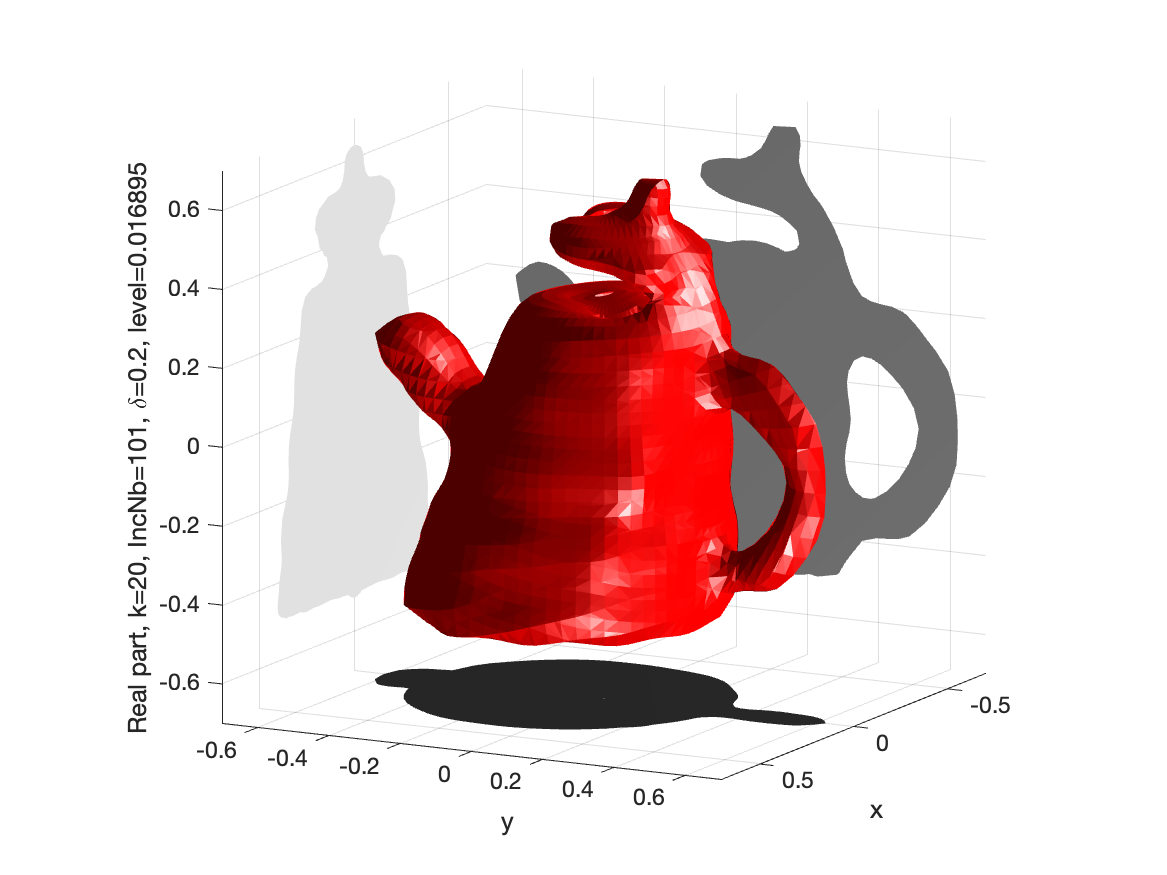} }
\subfloat[Iterative]{ \includegraphics[width=0.32\linewidth]{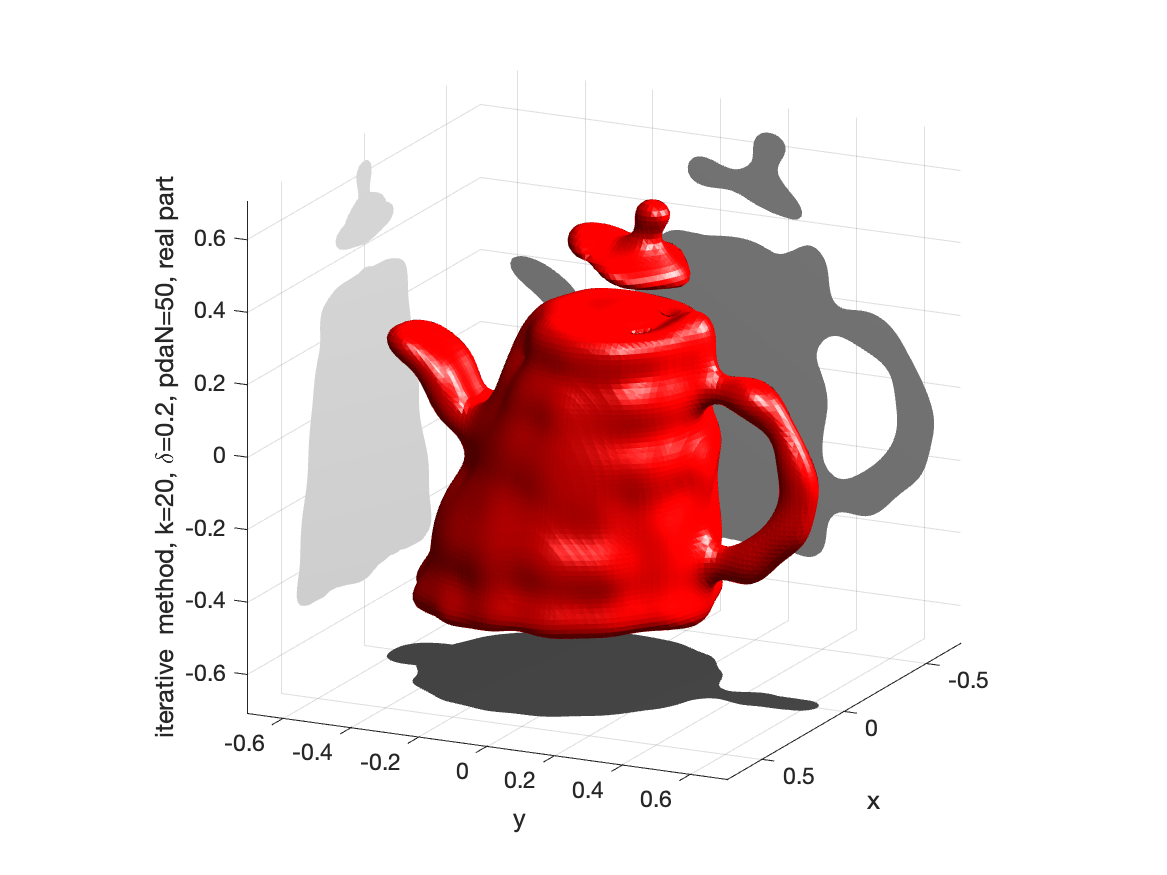} }
\\
\subfloat[Ground Truth]{ \includegraphics[width=0.32\linewidth]{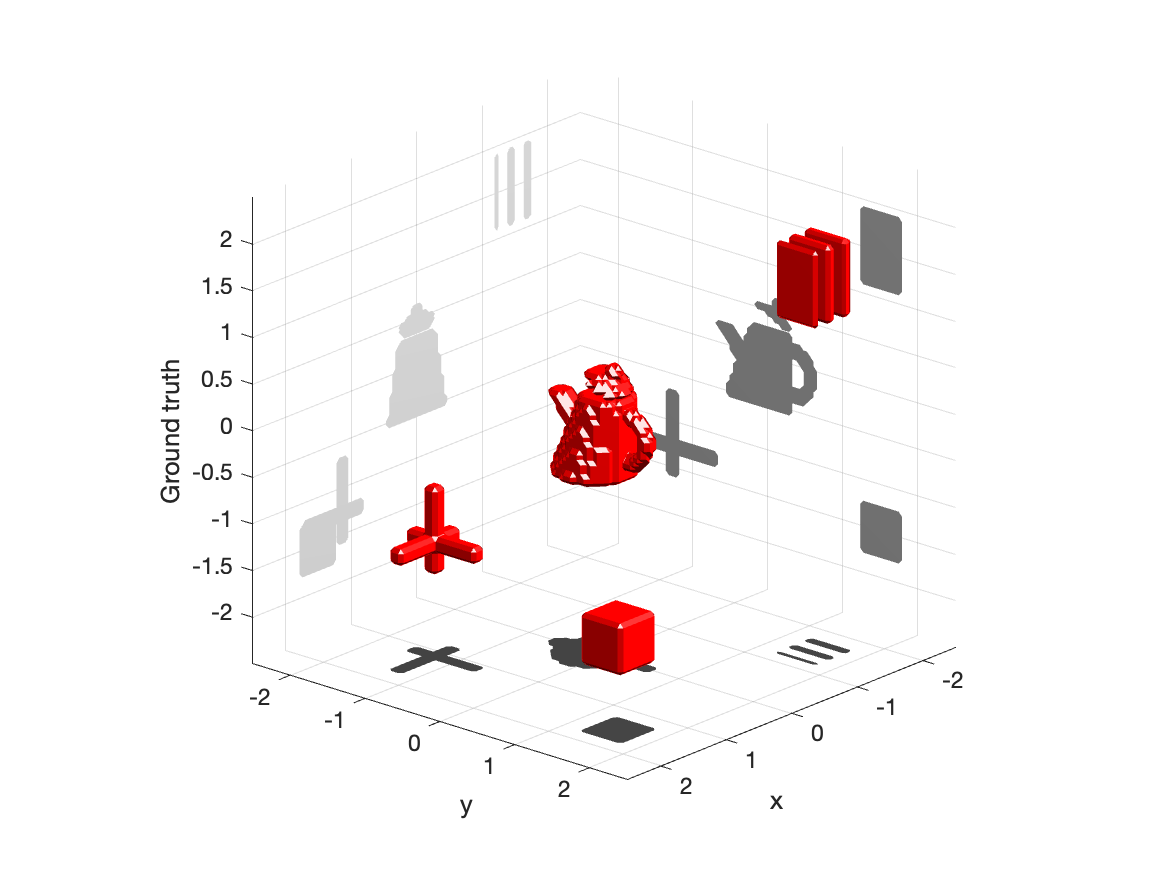} }
\subfloat[Low-rank]{ \includegraphics[width=0.32\linewidth]{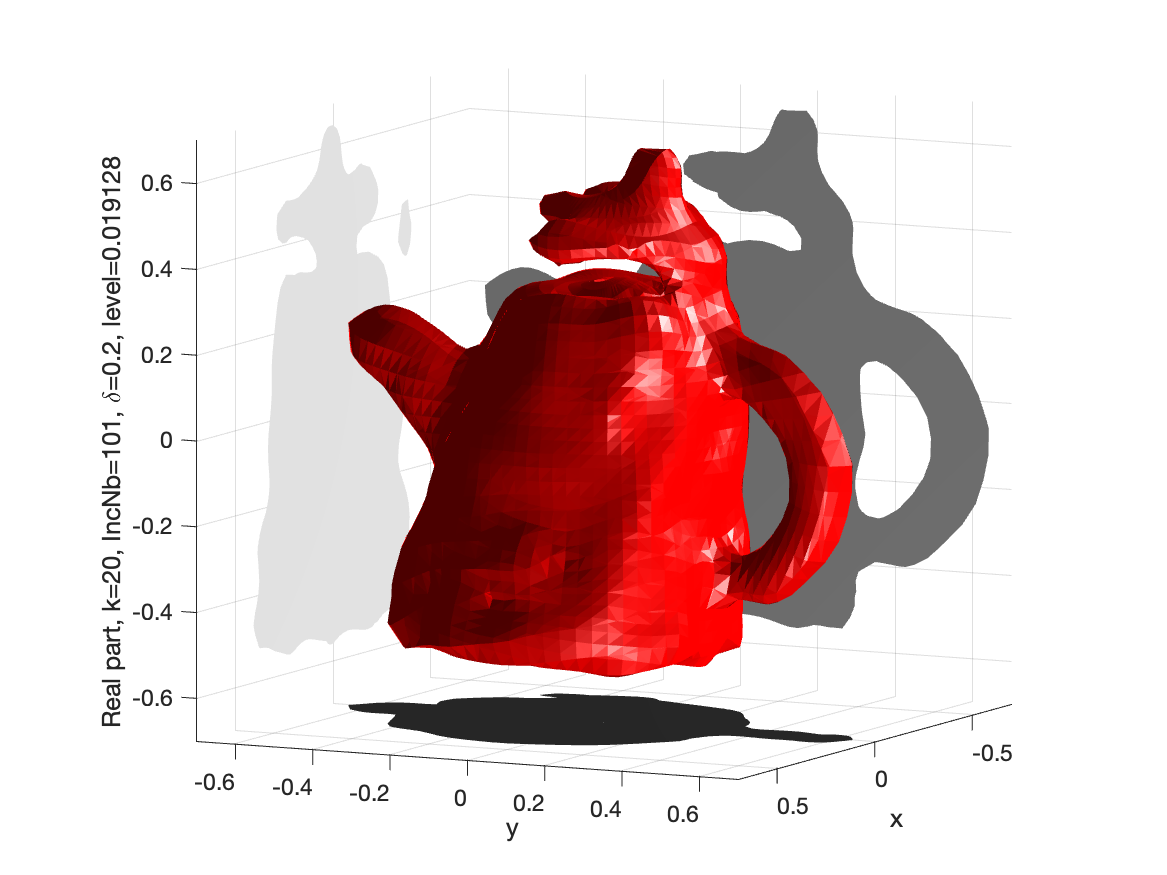}  }
\subfloat[Iterative]{ \includegraphics[width=0.32\linewidth]{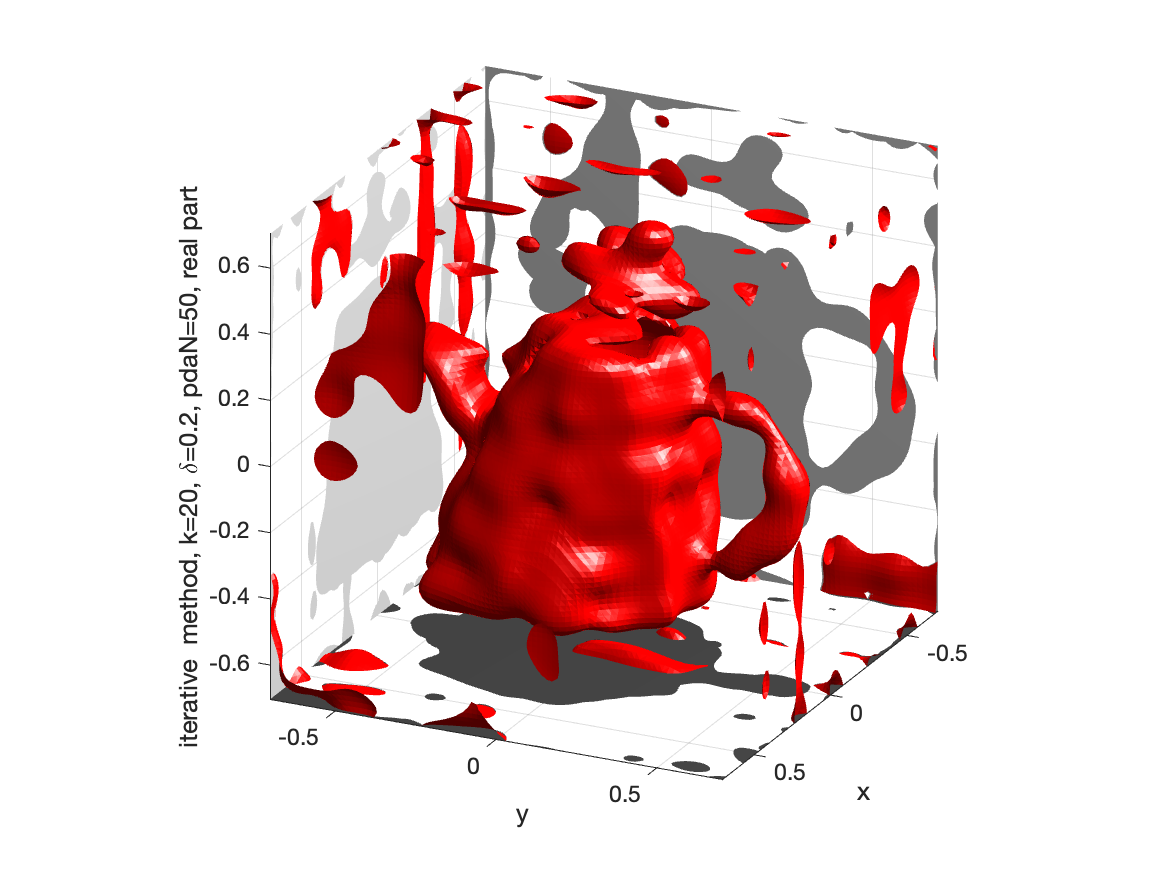} }
\caption{Imaging of a targeting teapot surrounded by other objects. Left column: ground truth; middle column: reconstruction via the low-rank structure; right column: reconstruction by the iterative method. $k=20,~\delta=20\%$.}\label{figure: localized imaging}
\end{figure}

\bibliographystyle{abbrv}
\bibliography{ref.bib}

\end{document}